\documentclass[cdm]{ipbook}

\usepackage{euscript}
\usepackage{mathrsfs}
\usepackage{bbold}%% to have nice \mathbb{1}
\usepackage{amssymb}
\usepackage{mathabx}
\usepackage{tikz}

\usepackage{float}
\startlocaldefs

\RequirePackage{color}
\definecolor{myred}{rgb}{0.75,0,0}
\definecolor{mygreen}{rgb}{0,0.5,0}
\definecolor{myblue}{rgb}{0,0,0.65}

\RequirePackage{ifpdf}
\ifpdf
  \IfFileExists{pdfsync.sty}{\RequirePackage{pdfsync}}{}
  \RequirePackage[pdftex,
   colorlinks = true,
   urlcolor = myblue, % \href{...}{...} external (URL)
   citecolor = mygreen, % \cite{...}
   linkcolor = myred, % \ref{...} and \pageref{...}
   pagebackref,
   bookmarksopen=true]{hyperref}
\else
  \RequirePackage[hypertex]{hyperref}
\fi

\RequirePackage{ae, aecompl, aeguill} % to have pretty pdf

\newcommand{\nc}{\newcommand} \newcommand{\renc}{\renewcommand}

  \def\bg{{\mathfrak b}}  
    \def\CM{{\mathbb{C}}}
    \def\DM{{\mathbb{D}}}
    
    \def\FM{{\mathbb{F}}}
  
  \def\hg{{\mathfrak h}}

  \def\lg{{\mathfrak l}}  
  \def\mg{{\mathfrak m}}  
    
    \def\OM{{\mathbb{O}}}
    \def\PM{{\mathbb{P}}}
    \def\QM{{\mathbb{Q}}}
    \def\RM{{\mathbb{R}}}

    \def\ZM{{\mathbb{Z}}}

    \def\AC{{\mathcal{A}}}
    
    \def\CC{{\mathcal{C}}}
    
    \def\EC{{\mathcal{E}}}

    \def\HC{{\mathcal{H}}}
    \def\IC{{\mathcal{I}}}

    \def\MC{{\mathcal{M}}}
    
    \def\OC{{\mathcal{O}}}
    
    \def\QC{{\mathcal{Q}}}
    \def\RC{{\mathcal{R}}}
    
    \def\TC{{\mathcal{T}}}

\def\AS{{\EuScript A}}
\def\BS{{\EuScript B}}

\def\ES{{\EuScript E}}
\def\FS{{\EuScript F}}
\def\GS{{\EuScript G}}

\def\LS{{\EuScript L}}

\def\PS{{\EuScript P}}

\def\TS{{\EuScript T}}

\def\a{\alpha}

\def\d{\delta}

\def\e{\varepsilon}

\def\z{\zeta}

\def\FSS{{\mathscr F}}
\def\GSS{{\mathscr G}}

\def\XSS{{\mathscr{X}}}

\nc{\todo}[1]{ {\color{red}XXX #1 XXX}}

\def\un{\underline}

\def\to{\rightarrow}

\def\longto{\longrightarrow}

\def\onto{\twoheadrightarrow}
\nc{\triright}{\stackrel{[1]}{\to}}
\nc{\longtriright}{\stackrel{[1]}{\longto}}

\nc{\Hb}{H^\bullet}

\nc{\Br}{\mathcal{B}}
\nc{\HotRR}{{}_R\mathcal{K}_R}
\nc{\HotR}{\mathcal{K}_R}
\nc{\excise}[1]{}
\nc{\defect}{\text{df}}
\nc{\h}[1]{\underline{H}_{#1}}

\nc{\Ga}{\mathbb{G}_a} % additive group
\nc{\Gm}{\mathbb{G}_m} % multiplicative group

\nc{\Perv}{{\mathbf{P}}}

\nc{\IH}{{\mathrm{IH}}}
\nc{\ic}{\mathbf{IC}}

\nc{\gl}{{\mathfrak{gl}}}
\renc{\sl}{{\mathfrak{sl}}}
\renc{\sp}{{\mathfrak{sp}}}

\renc{\Im}{\textrm{Im}}

\nc{\HBM}{H^{BM}}

 \DeclareMathOperator{\Hom}{Hom}
 \DeclareMathOperator{\ch}{ch}
\DeclareMathOperator{\End}{End} %\DeclareMathOperator{\BS}{BS}

\DeclareMathOperator{\Rep}{Rep}

\DeclareMathOperator{\id}{id}

\DeclareMathOperator{\Ext}{Ext}

\DeclareMathOperator{\Spec}{Spec}

\DeclareMathOperator{\Sp}{Sp}
\DeclareMathOperator{\SL}{SL}
\DeclareMathOperator{\GL}{GL}
\DeclareMathOperator{\Mat}{Mat}
\nc{\St}{\mathrm{St}}
\nc{\rot}{\mathrm{rot}}
\nc{\ext}{\mathrm{ext}}
\nc{\Tilt}{\mathrm{Tilt}}
\nc{\gen}{\mathrm{gen}}
\nc{\Graph}{\mathrm{Graph}}

\newcommand{\into}{\hookrightarrow}

\def\Fl{{\EuScript Fl}}
\def\Iw{{\textrm{Iw}}}

\nc{\simto}{\stackrel{\sim}{\to}}
\nc{\simfrom}{\stackrel{\sim}{\leftarrow}}
\DeclareMathOperator{\Fr}{Fr}

\DeclareMathOperator{\Perf}{Perf}
\DeclareMathOperator{\Sm}{Sm}

\def\acts{\hspace{1mm}\lefttorightarrow\hspace{1mm}}
\def\racts{\hspace{1mm}\righttoleftarrow\hspace{1mm}}
\nc{\gbmod}{\mathrm{-gmod-}}
\renewcommand{\bmod}{\textrm{-mod-}}
\nc{\gmod}{\mathrm{-gmod}}

\nc{\Parity}{\mathrm{Parity}}

\nc{\mult}{\mathrm{mult}}

\nc{\Hecke}{\textrm{H}}

\nc{\geom}{\mathrm{geom}}
\nc{\Soe}{\mathrm{Soe}}
\nc{\Abe}{\mathrm{Abe}}
\nc{\diag}{\mathrm{diag}}

\nc{\fin}{\textrm{finite}}

\nc{\reflect}{\RC}
\nc{\Chi}{\XSS}
\nc{\pt}{\mathrm{pt}}
\nc{\odd}{\textrm{odd}}
\nc{\even}{\textrm{even}}

\newtheorem{thm}{Theorem}[section]
\newtheorem{lem}[thm]{Lemma}

\newtheorem{prop}[thm]{Proposition}

\theoremstyle{definition}

\newtheorem{ex}[thm]{Example}

\theoremstyle{remark}
\newtheorem{remark}[thm]{Remark}
\newtheorem{question}[thm]{Question}

\endlocaldefs

\firstpage{1}
\lastpage{1}

\title[]{Modular representations and reflection subgroups}
\author[]{Geordie Williamson}
\address{University of Sydney,Australia.}
\email{g.williamson@sydney.edu.au}%\thanks{Blah blah}

\begin{document}

\begin{abstract}
The Hecke category is at the heart of several fundamental questions in
modular representation theory.  We emphasise the role of the
  ``philosophy of deformations'' both as a conceptual and
  computational tool, and suggest possible connections to Lusztig's
  ``philosophy of generations''. On the geometric side one 
  can understand deformations in terms of
  localisation in equivariant
  cohomology. Recently Treumann and Leslie-Lonergan have added Smith
  theory, which provides a useful tool when considering mod $p$
  coefficients.  In this context, we make contact with some remarkable
  work of Hazi. Using recent work of Abe on Soergel bimodules, we are
  able to reprove and generalise some of Hazi's results. Our aim is to convince
the reader that the work of Hazi and Leslie-Lonergan 
can usefully be viewed as some kind of localisation to ``good''
reflection subgroups.
These are notes for my lectures at the 2019
Current Developments in Mathematics at Harvard.
\end{abstract}

\maketitle

\tableofcontents

%%  The body

\section{Introduction} \label{sec:intro}

The following notes revolve around modular representations of
algebraic groups, symmetric groups and the Hecke category. There are
already a number of surveys on this topic \cite{JW-pcan,WTakagi,ARSurvey,WICM} and
there is even a book on the way \cite{soergelbook}. It would be silly to repeat  
this content. Instead I have tried to take a more speculative
perspective, and to dream a little. The main inspiration is recent
work of Treumann \cite{Treumann}, Hazi \cite{Hazi}, Abe \cite{Abe} and Leslie-Lonergan
\cite{LL}. As I was finishing these notes I became aware of the work of McDonnell
\cite{McDonnell}, which comes to similar conclusions to these
notes. Our hope is that our account complements McDonnell's.

Let us point out that these notes are rather inhomogeneous in their
demands on the reader. \S \ref{sec:intro} is introductory, and intended
for a general audience. \S \ref{sec:deformations} assumes quite a lot
of background from Lie theory. The material of \S \ref{sec:deformations} served as a principal motivation
for the writing of these notes, and so we thought it would be unfortunate not
to include it. \S \ref{sec:affine} and \S \ref{sec:localisation} are more
accessible. The road gets steeper in \S \ref{sec:Hecke} 
and steeper still in \S \ref{sec:heckeloc}, where I am writing for
experts.

\begin{figure}
  \centering
  \includegraphics[width=12cm]{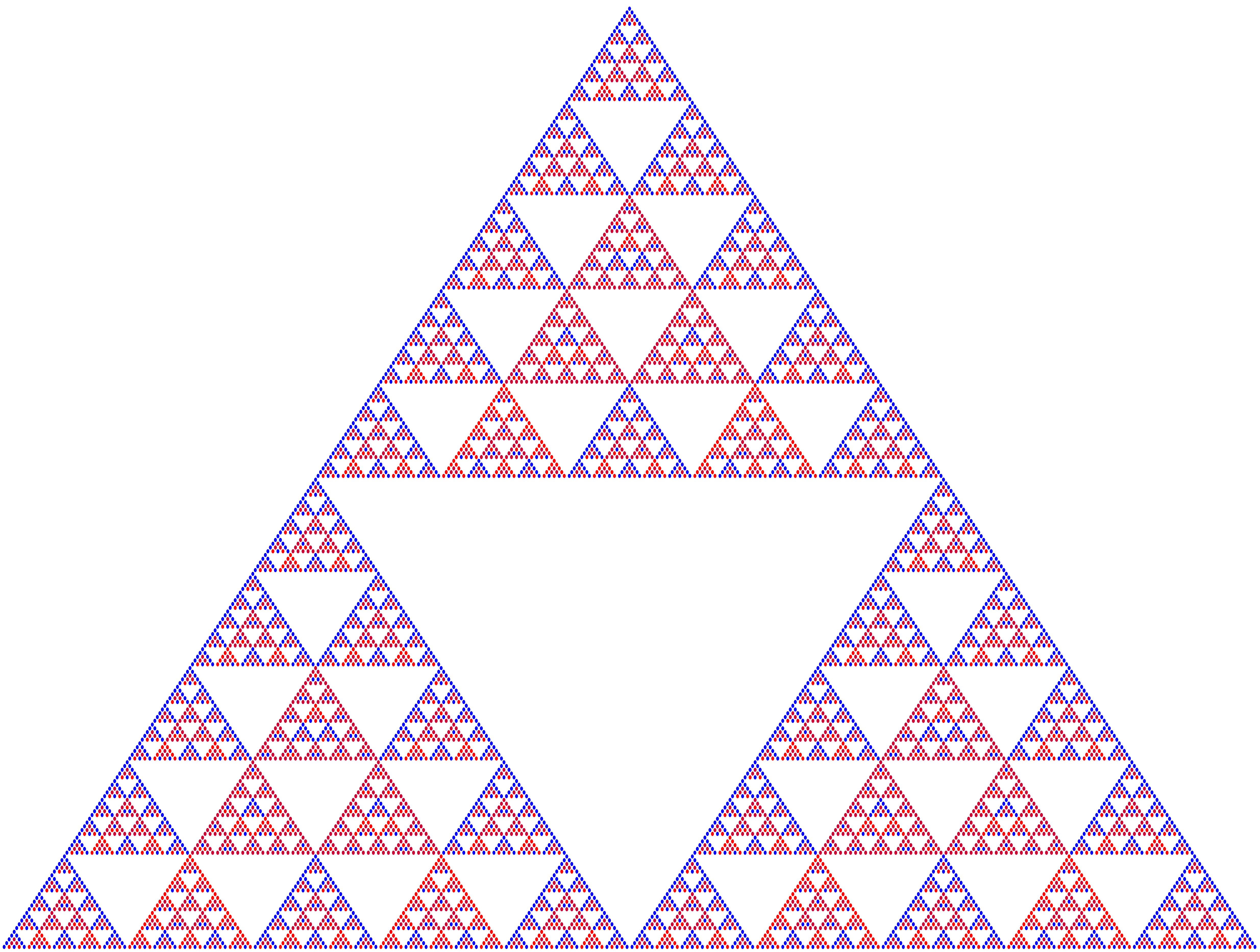}
\caption{A little mystery.}
\label{fig:p5}
\end{figure}

\subsection{A mystery}

The reader is asked to contemplate the picture in Figure
\ref{fig:p5}. What could it be? It is appears to be some Sierpi{\'n}ski like
figure, but we promise the reader that there is a little more going on. For example, the colours have a meaning. In Figure
\ref{fig:p5zoom}  one has a close up, as well as a close up of a
related picture.

\begin{figure}
  \centering
  \includegraphics[width=12cm]{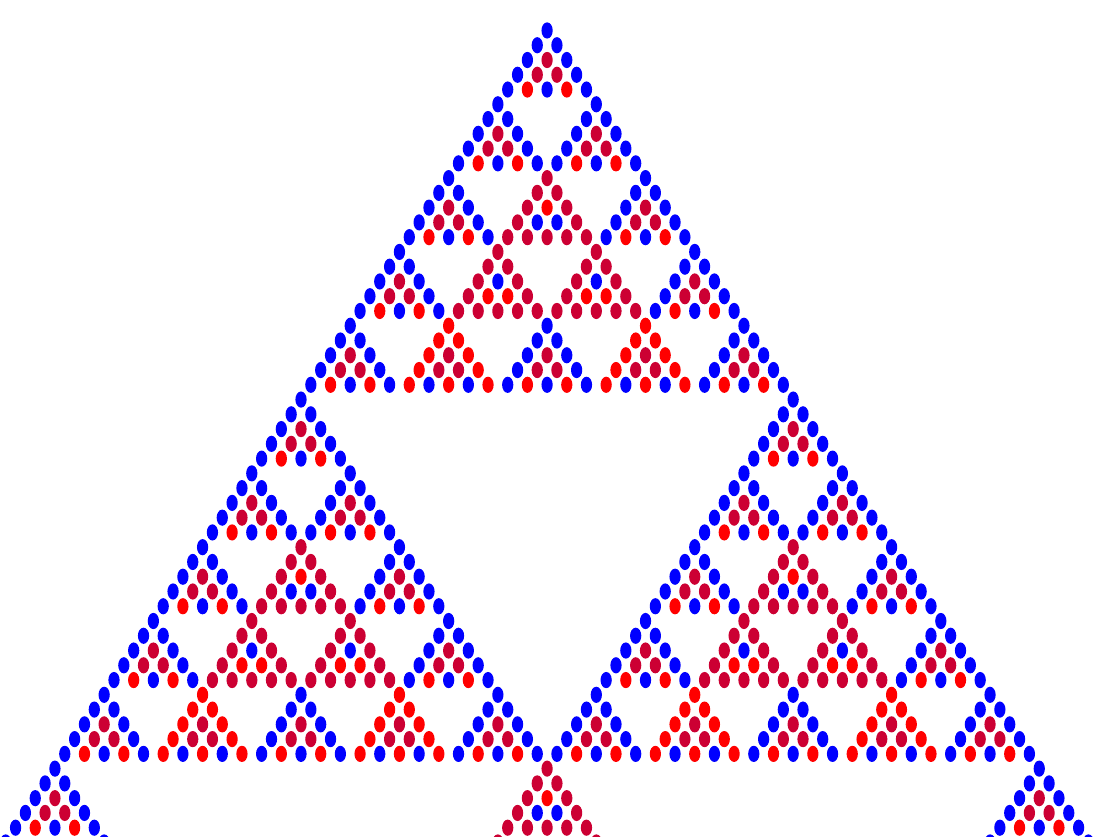} \\
\vspace{2cm}
  \includegraphics[width=12cm]{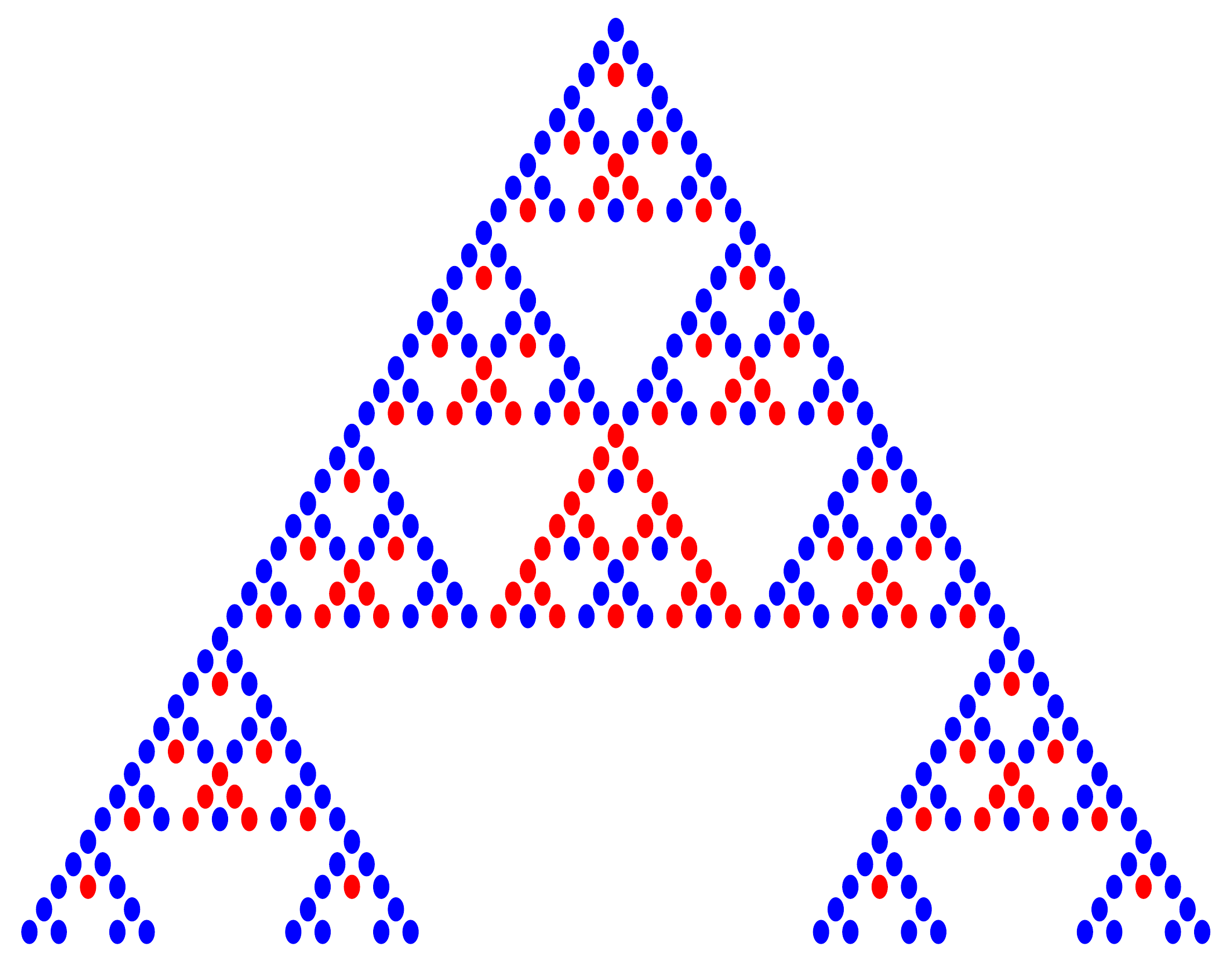}
\caption{Close ups, for $p = 5$ (top) and $p = 3$ (bottom).}
\label{fig:p5zoom}
\end{figure}

\begin{remark} \label{rem:fractal termination}
From the zoomed in pictures it is clear that the apparent fractal nature of Figure \ref{fig:p5}
actually terminates as we zoom in.  
\end{remark}

The answer is that we are staring at Pascal's triangle. More precisely,
we are staring at Pascal's triangle modulo $p$, for $p = 3$ and
$5$. The (non-trivial) colours encode the (non-trivial) residue classes modulo $p$:
\[
  \begin{array}{c} 
      \begin{tikzpicture}[xscale=0.25,yscale=0.4]
    \node at (0,0) {$1$};
    \node at (-1,-1) {$1$};
    \node at (1,-1) {$1$};
\node at (-2,-2) {$1$};
\node at (0,-2) {$2$};
\node at (2,-2) {$1$};
\node at (-3,-3) {$1$};
\node at (-1,-3) {$3$};
\node at (1,-3) {$3$};
\node at (3,-3) {$1$};
\node at (-4,-4) {$1$};
\node at (-2,-4) {$4$};
\node at (0,-4) {$6$};
\node at (2,-4) {$4$};
\node at (4,-4) {$1$};
\node at (-5,-5) {$1$};
\node at (-3,-5) {$5$};
\node at (-1,-5) {$10$};
\node at (1,-5) {$10$};
\node at (3,-5) {$5$};
\node at (5,-5) {$1$};
  \end{tikzpicture}
  \end{array}
\stackrel{\text{mod $3$}}{\longto}
  \begin{array}{c}
      \begin{tikzpicture}[xscale=0.25,yscale=0.4]
    \node at (0,0) {$1$};
    \node at (-1,-1) {$1$};
    \node at (1,-1) {$1$};
\node at (-2,-2) {$1$};
\node at (0,-2) {$2$};
\node at (2,-2) {$1$};
\node at (-3,-3) {$1$};
\node at (-1,-3) {$0$};
\node at (1,-3) {$0$};
\node at (3,-3) {$1$};
\node at (-4,-4) {$1$};
\node at (-2,-4) {$1$};
\node at (0,-4) {$0$};
\node at (2,-4) {$1$};
\node at (4,-4) {$1$};
\node at (-5,-5) {$1$};
\node at (-3,-5) {$2$};
\node at (-1,-5) {$1$};
\node at (1,-5) {$1$};
\node at (3,-5) {$2$};
\node at (5,-5) {$1$};
  \end{tikzpicture}
  \end{array}
=
\begin{array}{c}
\includegraphics[width=3cm]{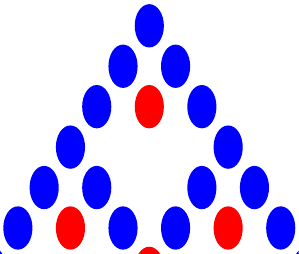}
\end{array}
\]

\subsection{Modular representations}
What on earth does this have to do with modular representations and
reflection subgroups, the title of these notes? To begin explaining the
connection, we start with one of the most fundamental of all groups
\[
G = \SL_2
\]
which we view as an algebraic group over $\Bbbk = \overline{\FM}_p$,
an algebraic closure of the finite field with $p$ elements. The reader
can think of $G$ as being the collection of $2 \times 2$-matrices\footnote{Quickly the language of group schemes becomes
invaluable; we will ignore this here.}
\[
\left \{ \left ( \begin{matrix} a & b \\ c & d \end{matrix} \right ) \in
    \Mat_2(\Bbbk) \;
  \middle | \;  ad-bc = 1 \right \}.
\]

A key feature in characteristic $p$ is the \emph{Frobenius endomorphism:}
\begin{align*}
  \Fr : \SL_2 &\to \SL_2\\
  \left ( \begin{matrix} a & b \\ c &
    d \end{matrix} \right ) &\mapsto \left ( \begin{matrix} a^p & b^p \\ c^p & d^p \end{matrix} \right )
\end{align*}
Nothing like this exists in characteristic $0$.\footnote{At least, not
  until one meets the quantum group!} It can be thought of
as a kind of ``everywhere contracting mapping''.

The questions in these notes are motivated by the study of
\emph{algebraic representations} of groups
like $\SL_2$. This means we are studying homomorphisms
\[
\rho : \SL_2 \to \GL(V)
\]
for some $\Bbbk$-vector space $V$, which are defined by polynomials in
$a, b, c$ and $d$. One of the first results in the theory is that all
representations are direct limits of finite-dimensional
representations, so we usually assume that $V$ is finite-dimensional. In
other words, we are studying polynomial homomorphisms into $GL_n$.

\begin{remark} If the reader is like the author, they might initially find the
  study of algebraic representations somewhat artificial. Why not
  stick to finite groups like $\SL_2(\FM_q)$? In fact there is a close
  connection between representations of finite groups of Lie type like
  $\SL_2(\FM_q)$ in ``their own'' characteristic, and the theory of
  algebraic representations. It turns out that the extra
  structures arising from the theory of algebraic groups are
  tremendously useful. One can think of the finite groups
  $\SL_2(\FM_q)$ as somewhat like lattices in the ``Lie group'' $\SL_2$. Thus the passage from $\SL_2(\FM_q)$ to $\SL_2$ is
  akin to the more familiar fact that the representation theory of
  connected Lie groups is easier than that of finite groups.
\end{remark}

\subsection{Some examples of algebraic representations}
Here are a few examples of representations (of dimensions $1$, $2$ and
$3$):
\begin{align*}
\left ( \begin{matrix} a & b \\ c & d \end{matrix} \right ) &
                                                                    \mapsto
                                                                    (1)
  \\
\left ( \begin{matrix} a & b \\ c & d \end{matrix} \right ) &
                                                                    \mapsto
\left ( \begin{matrix} a & b \\ c & d \end{matrix} \right )
  \\
\left ( \begin{matrix} a & b \\ c & d \end{matrix} \right ) &
                                                                    \mapsto
\left ( \begin{matrix} a^2 & ab & b^2  \\ 2ac & ad + bc & 2bd \\ c^2
    & cd & d^2 \end{matrix} \right )
% g \mapsto & \left ( \begin{matrix} a & b \\ c & d \end{matrix} \right
%                                                 ) \\
% g \mapsto & \left ( \begin{matrix} a^2 & ab & b^2 \\ c & d & blah \end{matrix} \right )                                            
\end{align*}
The first is the trivial representation. The second is the
``natural representation'', so called because it arises from the definition of
$\SL_2$ as a matrix group. The third arises as follows. The natural
representation gives us a linear action of $\SL_2$ on
\[
\SL_2 \acts V = \Bbbk X \oplus \Bbbk Y.
\]
For example,
\[
\left ( \begin{matrix} a & b \\ c & d \end{matrix} \right )  \cdot X
= aX + cY.
\]
Thus $\SL_2$ also acts on the second symmetric power
\[
\SL_2 \acts  S^2(V) = \Bbbk X^2 \oplus \Bbbk XY \oplus \Bbbk Y^2.
\]
Writing out this action in the basis of monomials produces the above
matrices. For example
\[
\left ( \begin{matrix} a & b \\ c & d \end{matrix} \right )  \cdot X^2
= (aX + cY)^2 = a^2 X + 2ac XY + c^2 Y^2
\]
which is the first column of our matrix. (This also makes it clear why
our third representation above is actually a representation, which is
not clear when presented with a formula.)

\subsection{Chevalley's theorem} Of course there was no reason to stop
at the second symmetric power above. We can consider the module
\[
\SL_2 \acts \nabla_n := S^nV = \Bbbk X^n \oplus \Bbbk X^{n-1}Y \Bbbk \oplus 
\dots \oplus \Bbbk Y^n
\]
for any $n \ge 0$. The modules $\nabla_0, \nabla_1$ and
$\nabla_2$ being the cases considered above. We have
\[
\dim \nabla_n = n +1 .
\]
If our field were of characteristic zero, these
modules would provide a complete list of inequivalent irreducible representations
of $\SL_2$.\footnote{The reader may have seen this in the guise of the
irreducible representations of $\mathfrak{sl}_2(\CM)$, or of the
compact Lie group $SU_2$, where these representations are closely tied
to the theory of spherical harmonics.}

However, in the current setting the situation is a little more
interesting. Let us for concreteness assume that $p = 3$, and consider
\[
\nabla_3 = \Bbbk X^3 \oplus \Bbbk X^2Y \oplus \Bbbk XY^2 \oplus\Bbbk Y^3.
\]
If we act on the vector $X^3 \in \nabla_3$ we have
\[
\left ( \begin{matrix} a & b \\ c & d \end{matrix} \right )  \cdot X^3
= (aX + cY)^3 = a^3X^3 + c^3 Y^3
\]
by the ``Freshman's dream'' in characteristic 3. A similar calculation
for $Y^3$ shows that
\[
\Bbbk X^2 \oplus \Bbbk Y^3 \subset \nabla_3
\]
is a non-trivial $\SL_2$-invariant submodule. Moreover, this
submodule affords the representation
\[
\left ( \begin{matrix} a & b \\ c & d \end{matrix} \right )   \mapsto 
\left ( \begin{matrix} a^3 & b^3 \\ c^3 & d^3 \end{matrix} \right )  
\]
obtained by pulling the natural module back along the Frobenius
morphism $\Fr$ from earlier. This operation on representations is
called \emph{Frobenius twist} and is fundamental to the theory.

It turns out that $\nabla_n$ always has a unique simple
submodule. Moreover, if we denote this submodule by $L_n \subset
\nabla_n$ then we have a bijection:
\begin{align*}
\ZM_{\ge 0}  &\simto \left \{ \begin{array}{c} \text{simple algebraic} \\ 
           \SL_2\text{-modules} \end{array} \right \}_{ / \cong} \\
n &\mapsto \hspace{1.5cm} L_n  
\end{align*}

\begin{remark}
This is a example of \emph{Chevalley's theorem},  which is true for any reductive algebraic
group. In the context of a general
algebraic group it tells us that simple modules are always classified
by highest weight, just as over $\CM$. (The reader is cautioned that
although they are classified independently of $p$, their structure
varies subtly based on $p$.)
\end{remark}

\subsection{Characters} Inside $\SL_2$ we can consider the \emph{maximal
torus}
\[
T := \left \{ \left ( \begin{matrix} a & 0 \\ 0 & a^{-1} \end{matrix}
  \right ) \right \} \subset \SL_2.
\]
Any algebraic representation $V$ of $\SL_2$ splits as a direct sum
\[ V = \bigoplus_{m \in \ZM} V_m \]
of its \emph{weight spaces}
\[
V_m := \left \{ v \in V \; \middle | \; \left ( \begin{matrix} a & 0 \\ 0 &
  a^{-1} \end{matrix} \right ) \cdot v = a^m v \text{ for all $a \in
\Bbbk^*$} \right \}.
\]
Probably the most basic invariant of a representation aside from its
dimension is its \emph{character}
\[
\ch(V) := \sum_{m \in \ZM} (\dim V_m) e^m \in \ZM_{\ge 0}[e^{\pm 1}].
\]
For example, $X^i Y^{n-i} \in \nabla_n$ satisfies
\[
\left ( \begin{matrix} a & 0 \\ 0 &
  a^{-1} \end{matrix} \right ) \cdot X^i Y^{n-i}  = a^i a^{i-n} X^i
Y^{n-i} = a^{-n + 2i} X^i Y^{n-i}.
\]
Thus all weight spaces of $\nabla_n$ are either 0 or 1-dimensional and
\[
\ch(\nabla_n) = e^{-n} + e^{-n+2} + \dots + e^{n-2} + e^n = \frac{e^n
  - e^{-n-2}}{1-e^{-2}}.
\]

\begin{remark}
The last equality is an example of \emph{Weyl's character formula}, which is
of central importance in the theory of compact Lie groups, as it is to
the theory of algebraic groups over $\CM$. It is also very
important in the theory we consider here, as it gives us the
characters of certain building blocks of all representations
(so-called \emph{Weyl modules} and \emph{induced modules}, of which
$\nabla_n$ is an example).
\end{remark}

\subsection{Characters of simples}  Let us assume $p = 3$. A few lines
of calculations give the following descriptions of $L_n \subset
\nabla_n$ for $n = 0, 1, 
\dots, 6$:
\[
\begin{tikzpicture}[yscale=0.6,xscale=0.75]
\node at (0,0) {$\small \Bbbk $};
\node at (-1,-1) {$\small \Bbbk Y$};
\node at (1,-1) {$\small \Bbbk X$};
\node at (-2,-2) {$\small \Bbbk Y^2$};
\node at (0,-2) {$\small \Bbbk XY$};
\node at (2,-2) {$\small \Bbbk X^2$};
\node at (-3,-3) {$\small \Bbbk Y^3$};
\node at (-1,-3) {$0$};
\node at (1,-3) {$0$};
\node at (3,-3) {$\small \Bbbk X^3$};
\node at (-4,-4) {$\small \Bbbk Y^4$};
\node at (-2,-4) {$\small \Bbbk Y^3X$};
\node at (0,-4) {$0$};
\node at (2,-4) {$\small \Bbbk YX^3$};
\node at (4,-4) {$\small \Bbbk X^4$};
\node at (-5,-5) {$\small \Bbbk Y^5$};
\node at (-3,-5) {$\small \Bbbk Y^4X$};
\node at (-1,-5) {$\small \Bbbk Y^3X^2$};
\node at (1,-5) {$\small \Bbbk Y^2X^3$};
\node at (3,-5) {$\small \Bbbk YX^4$};
\node at (5,-5) {$\small \Bbbk X^5$};
\node at (-6,-6) {$\small \Bbbk Y^6$};
\node at (-4,-6) {$0$};
\node at (-2,-6) {$0$};
\node at (0,-6) {$\small \Bbbk X^3Y^3$};
\node at (2,-6) {$0$};
\node at (4,-6) {$0$};
\node at (6,-6) {$\small \Bbbk Y^6$};
\node at (-8,0) {$L_0 = $};
\node at (-8,-1) {$L_{1} = $};
\node at (-8,-2) {$L_{2} = $};
\node at (-8,-3) {$L_{3} = $};
\node at (-8,-4) {$L_{4} = $};
\node at (-8,-5) {$L_{5} = $};
\node at (-8,-6) {$L_{6} = $};
\end{tikzpicture}
\]
The third line is our calculation using the Freshman's dream from
earlier.  Thus the characters (for $n = 0, 1, \dots, 6$ as earlier) are given as
follows:
\[
\begin{tikzpicture}[yscale=0.6,xscale=0.5]
\node at (0,0) {$e^0$};
\node at (-1,-1) {$e^{-1}$};
\node at (0,-1) {$+$};
\node at (1,-1) {$e^1$};
\node at (-2,-2) {$e^{-2}$};
\node at (-1,-2) {$+$};
\node at (0,-2) {$e^0$};
\node at (1,-2) {$+$};
\node at (2,-2) {$e^2$};
\node at (-3,-3) {$e^{-3}$};
\node at (-2,-3) {$+$};
% \node at (-1,-3) {$0$};
% \node at (0,-3) {$+$};
% \node at (1,-3) {$0$};
\node at (2,-3) {$+$};
\node at (3,-3) {$e^3$};
\node at (-4,-4) {$e^{-4}$};
\node at (-3,-4) {$+$};
\node at (-2,-4) {$e^{-2}$};
\node at (-1,-4) {$+$};
% \node at (0,-4) {$0$};
\node at (1,-4) {$+$};
\node at (2,-4) {$e^{2}$};
\node at (3,-4) {$+$};
\node at (4,-4) {$e^{4}$};
\node at (-5,-5) {$e^{-5}$};
\node at (-4,-5) {$+$};
\node at (-3,-5) {$e^{-3}$};
\node at (-2,-5) {$+$};
\node at (-1,-5) {$e^{-1}$};
\node at (0,-5) {$+$};
\node at (1,-5) {$e^{1}$};
\node at (2,-5) {$+$};
\node at (3,-5) {$e^{3}$};
\node at (4,-5) {$+$};
\node at (5,-5) {$e^{5}$};
\node at (-6,-6) {$e^{-6}$};
\node at (-5,-6) {$+$};
% \node at (-4,-6) {$0$};
% \node at (-3,-6) {$+$};
% \node at (-2,-6) {$0$};
\node at (-1,-6) {$+$};
\node at (0,-6) {$e^0$};
\node at (1,-6) {$+$};
% \node at (2,-6) {$0$};
% \node at (3,-6) {$+$};
% \node at (4,-6) {$0$};
\node at (5,-6) {$+$};
\node at (6,-6) {$e^{6}$};
% %%
\node at (-10,0) {$\ch(L_0) = $};
\node at (-10,-1) {$\ch(L_{1}) = $};
\node at (-10,-2) {$\ch(L_{2}) = $};
\node at (-10,-3) {$\ch(L_{3}) = $};
\node at (-10,-4) {$\ch(L_{4}) = $};
\node at (-10,-5) {$\ch(L_{5}) = $};
\node at (-10,-6) {$\ch(L_{6}) = $};
\end{tikzpicture}
\]
This should remind the reader of the picture earlier:
\[
\begin{array}{c}
  \includegraphics[width=12cm]{pics/cp3} 
\end{array}
\]
This is not an accident: the non-zero $(n-2i)^{th}$ weight space in
$L_n$ is non-zero (hence one-dimensional), if and only $n \choose i$
is non-zero modulo $p$.\footnote{For the experts: one can check this statement using
  Steinberg's tensor product theorem as we will see in a moment. However the best proof that I know uses
  the Jantzen filtration: the Shapavalov form on the $(n-2i)^{th}$
  weight space is
  $n \choose i$, and hence this weight space is in the $0^{th}$-layer
  of the Jantzen filtration if and only if $n \choose i$
  is non-zero modulo $p$.}

\subsection{Steinberg tensor product theorem} Let us continue our
running example of $p = 3$. We have seen that
\begin{equation} \label{eq:simples}
\nabla_0, \nabla_1, \nabla_2
\end{equation}
are simple, but that $\nabla_3$ is not. In fact, we saw that it
contains a two-dimensional submodule $L_3$ which is isomorphicto a
Frobenius twist
\[
L_3 \cong \nabla_1^{(1)} \subset \nabla_3.
  \]
where $V \mapsto V^{(m)}$ denotes the operation of precomposing $m$
times with the Frobenius map.

In fact, our three modules \eqref{eq:simples} above are the building
blocks of any representation. Given $m$ we consider its $3$-adic
expansion
\[
m = \sum_{i = 1}^\ell m_i p^i \quad \text{with $0 \le m_i < 3$}
\]
and we have
\begin{equation}
  \label{eq:steinberg}
L_m := \nabla_{m_0} \otimes \nabla_{m_1}^{(1)} \otimes \dots \otimes \nabla_{m_\ell}^{(\ell)}.  
\end{equation}
More precisely, one has an evident map given by multiplication from
the above into $\nabla_m$, and the claim is that its image is simple.

Of course there is nothing special about $p = 3$. For general $p$
we consider its $p$-adic expansion
\[
m = \sum_{i = 1}^\ell m_i p^i \quad \text{with $0 \le m_i < p$}
\]
and the above tensor decomposition holds. Taking characters this implies
\[
\ch(L_m) = \prod_{i = 1}^\ell ( e^{-m_i} + e^{-m_i + 2} + \dots + e^{m_i-2} + e^{m_i})^{(i)}
\]
where $f \mapsto f^{(i)}$ is induced by the map $e \mapsto e^{p^i}$ on
$\ZM[e^{\pm 1}]$.

\begin{remark}
  The reader can use this formula to check that the
$(n-2i)^{th}$ weight space of $L_m$ is non-zero if and only if there
are no carries when $i$ and $m$ are added $p$-adically, which is the
condition for the binomial coefficient $n \choose i$ to be non-zero,
by Kummer's theorem on valuations of binomial coefficients.
\end{remark}

\begin{remark}
  Equation \eqref{eq:steinberg} is an instance of the \emph{Steinberg tensor product
    theorem}, which is valid for any reductive group. For a general
  group however, the question of the description of the building
  blocks ($\nabla_0, \dots, \nabla_{p-1}$ for $\SL_2$) is much more
  complicated. This is the subject of Lusztig's character formula
  \cite{LusztigSomeProblems, WTakagi, LusztigGen}.
\end{remark}

\subsection{Generations} It is interesting to consider Pascal's
triangle, where we replace the binomial coefficients with their
Gau{\ss}ian cousins
\begin{equation}
  \label{eq:choose}
{ n \choose i } \rightsquigarrow \left [ \begin{matrix} n \\
    i \end{matrix} \right ]_v := \frac{[n]_v![n-1]_v! \dots
  [n-i+1]_v!}{[i]_v![i-1]_v! \dots [2]_v![1]_v!} \in \ZM[v^{\pm 1}].
\end{equation}
where
\begin{equation}
  \label{eq:factor} \small
[n]_v! = [n]_v [n-1]_v \dots [2]_v[1]_v = 
\frac{v^n-v^{-n}}{v-v^{-1}} \cdot \frac{v^{n-1}-v^{-n+1}}{v-v^{-1}}
\cdot  \dots \cdot
(v +v^{-1}) \cdot 1.
\end{equation}

That is we replace
\[
\tiny
  \begin{array}{c} 
      \begin{tikzpicture}[xscale=0.18,yscale=0.6]
    \node at (0,0) {$1$};
    \node at (-1,-1) {$1$};
    \node at (1,-1) {$1$};
\node at (-2,-2) {$1$};
\node at (0,-2) {$2$};
\node at (2,-2) {$1$};
\node at (-3,-3) {$1$};
\node at (-1,-3) {$3$};
\node at (1,-3) {$3$};
\node at (3,-3) {$1$};
\node at (-4,-4) {$1$};
\node at (-2,-4) {$4$};
\node at (0,-4) {$6$};
\node at (2,-4) {$4$};
\node at (4,-4) {$1$};
\node at (-5,-5) {$1$};
\node at (-3,-5) {$5$};
\node at (-1,-5) {$10$};
\node at (1,-5) {$10$};
\node at (3,-5) {$5$};
\node at (5,-5) {$1$};
\node[rotate=-15] at (-6,-6) {$\vdots$};
\node at (0,-6) {$\vdots$};
\node[rotate=15] at (6,-6) {$\vdots$};
  \end{tikzpicture}
  \end{array}
\rightsquigarrow 
\tiny
 \begin{array}{c} 
      \begin{tikzpicture}[xscale=0.6,yscale=0.6]
    \node at (0,0) {$1$};
    \node at (-1,-1) {$1$};
    \node at (1,-1) {$1$};
\node at (-2,-2) {$1$};
\node at (0,-2) {$[2]_v$};
\node at (2,-2) {$1$};
\node at (-3,-3) {$1$};
\node at (-1,-3) {$[3]_v$};
\node at (1,-3) {$[3]_v$};
\node at (3,-3) {$1$};
\node at (-4,-4) {$1$};
\node at (-2,-4) {$[4]_v$};
\node at (0,-4) {$[3]_v[2]_{v^2}$};
\node at (2,-4) {$[4]_v$};
\node at (4,-4) {$1$};
\node at (-5,-5) {$1$};
\node at (-3,-5) {$[5]_v$};
\node at (-1,-5) {$[5]_v[2]_{v^2}$};
\node at (1,-5) {$[5]_v[2]_{v^2}$};
\node at (3,-5) {$[5]_v$};
\node at (5,-5) {$1$};
\node[rotate=-45] at (-6,-6) {$\vdots$};
\node at (0,-6) {$\vdots$};
\node[rotate=45] at (6,-6) {$\vdots$};
  \end{tikzpicture}
  \end{array}
\]

It has been observed in several circumstances that Gau{\ss}ian
binomial coefficients at a $p^{th}$-root of unity ``imitate
characteristic $p$.'' (This observation goes back at least to
\cite{LusztigQuantum}.) 
We can see this here: if we specialise $v :=
e^{{2\pi i}/3}$ we get a vanishing behaviour that looks like the ``first
level'' of the picture that we began with (for clarity we have replaced zeroes with
empty space):
\[
\begin{tikzpicture}[scale=0.3]
\tiny
\node at (0, 0) {$1$};

\node at (-1, -1) {$1$};
\node at (1, -1) {$1$};

\node at (-2, -2) {$1$};
\node at (0, -2) {$-1$};
\node at (2, -2) {$1$};

\node at (-3, -3) {$1$};
\node at (-1, -3) {};
\node at (1, -3) {};
\node at (3, -3) {$1$};

\node at (-4, -4) {$1$};
\node at (-2, -4) {$1$};
\node at (0, -4) {};
\node at (2, -4) {$1$};
\node at (4, -4) {$1$};

\node at (-5, -5) {$1$};
\node at (-3, -5) {$-1$};
\node at (-1, -5) {$1$};
\node at (1, -5) {$1$};
\node at (3, -5) {$-1$};
\node at (5, -5) {$1$};

\node at (-6, -6) {$1$};
\node at (-4, -6) {};
\node at (-2, -6) {};
\node at (0, -6) {$2$};
\node at (2, -6) {};
\node at (4, -6) {};
\node at (6, -6) {$1$};

\node at (-7, -7) {$1$};
\node at (-5, -7) {$1$};
\node at (-3, -7) {};
\node at (-1, -7) {$2$};
\node at (1, -7) {$2$};
\node at (3, -7) {};
\node at (5, -7) {$1$};
\node at (7, -7) {$1$};

\node at (-8, -8) {$1$};
\node at (-6, -8) {$-1$};
\node at (-4, -8) {$1$};
\node at (-2, -8) {$2$};
\node at (0, -8) {$-2$};
\node at (2, -8) {$2$};
\node at (4, -8) {$1$};
\node at (6, -8) {$-1$};
\node at (8, -8) {$1$};

\node at (-9, -9) {$1$};
\node at (-7, -9) {};
\node at (-5, -9) {};
\node at (-3, -9) {$3$};
\node at (-1, -9) {};
\node at (1, -9) {};
\node at (3, -9) {$3$};
\node at (5, -9) {};
\node at (7, -9) {};
\node at (9, -9) {$1$};

\node at (-10, -10) {$1$};
\node at (-8, -10) {$1$};
\node at (-6, -10) {};
\node at (-4, -10) {$3$};
\node at (-2, -10) {$3$};
\node at (0, -10) {};
\node at (2, -10) {$3$};
\node at (4, -10) {$3$};
\node at (6, -10) {};
\node at (8, -10) {$1$};
\node at (10, -10) {$1$};

\node at (-11, -11) {$1$};
\node at (-9, -11) {$-1$};
\node at (-7, -11) {$1$};
\node at (-5, -11) {$3$};
\node at (-3, -11) {$-3$};
\node at (-1, -11) {$3$};
\node at (1, -11) {$3$};
\node at (3, -11) {$-3$};
\node at (5, -11) {$3$};
\node at (7, -11) {$1$};
\node at (9, -11) {$-1$};
\node at (11, -11) {$1$};

\node at (-12, -12) {$1$};
\node at (-10, -12) {};
\node at (-8, -12) {};
\node at (-6, -12) {$4$};
\node at (-4, -12) {};
\node at (-2, -12) {};
\node at (0, -12) {$6$};
\node at (2, -12) {};
\node at (4, -12) {};
\node at (6, -12) {$4$};
\node at (8, -12) {};
\node at (10, -12) {};
\node at (12, -12) {$1$};

\node at (-13, -13) {$1$};
\node at (-11, -13) {$1$};
\node at (-9, -13) {};
\node at (-7, -13) {$4$};
\node at (-5, -13) {$4$};
\node at (-3, -13) {};
\node at (-1, -13) {$6$};
\node at (1, -13) {$6$};
\node at (3, -13) {};
\node at (5, -13) {$4$};
\node at (7, -13) {$4$};
\node at (9, -13) {};
\node at (11, -13) {$1$};
\node at (13, -13) {$1$};

\node at (-14, -14) {$1$};
\node at (-12, -14) {$-1$};
\node at (-10, -14) {$1$};
\node at (-8, -14) {$4$};
\node at (-6, -14) {$-4$};
\node at (-4, -14) {$4$};
\node at (-2, -14) {$6$};
\node at (0, -14) {$-6$};
\node at (2, -14) {$6$};
\node at (4, -14) {$4$};
\node at (6, -14) {$-4$};
\node at (8, -14) {$4$};
\node at (10, -14) {$1$};
\node at (12, -14) {$-1$};
\node at (14, -14) {$1$};

\node at (-15, -15) {$1$};
\node at (-13, -15) {};
\node at (-11, -15) {};
\node at (-9, -15) {$5$};
\node at (-7, -15) {};
\node at (-5, -15) {};
\node at (-3, -15) {$10$};
\node at (-1, -15) {};
\node at (1, -15) {};
\node at (3, -15) {$10$};
\node at (5, -15) {};
\node at (7, -15) {};
\node at (9, -15) {$5$};
\node at (11, -15) {};
\node at (13, -15) {};
\node at (15, -15) {$1$};

\node at (-16, -16) {$1$};
\node at (-14, -16) {$1$};
\node at (-12, -16) {};
\node at (-10, -16) {$5$};
\node at (-8, -16) {$5$};
\node at (-6, -16) {};
\node at (-4, -16) {$10$};
\node at (-2, -16) {$10$};
\node at (0, -16) {};
\node at (2, -16) {$10$};
\node at (4, -16) {$10$};
\node at (6, -16) {};
\node at (8, -16) {$5$};
\node at (10, -16) {$5$};
\node at (12, -16) {};
\node at (14, -16) {$1$};
\node at (16, -16) {$1$};

\node at (-17, -17) {$1$};
\node at (-15, -17) {$-1$};
\node at (-13, -17) {$1$};
\node at (-11, -17) {$5$};
\node at (-9, -17) {$-5$};
\node at (-7, -17) {$5$};
\node at (-5, -17) {$10$};
\node at (-3, -17) {$-10$};
\node at (-1, -17) {$10$};
\node at (1, -17) {$10$};
\node at (3, -17) {$-10$};
\node at (5, -17) {$10$};
\node at (7, -17) {$5$};
\node at (9, -17) {$-5$};
\node at (11, -17) {$5$};
\node at (13, -17) {$1$};
\node at (15, -17) {$-1$};
\node at (17, -17) {$1$};

\node at (-18, -18) {$1$};
\node at (-16, -18) {};
\node at (-14, -18) {};
\node at (-12, -18) {$6$};
\node at (-10, -18) {};
\node at (-8, -18) {};
\node at (-6, -18) {$15$};
\node at (-4, -18) {};
\node at (-2, -18) {};
\node at (0, -18) {$20$};
\node at (2, -18) {};
\node at (4, -18) {};
\node at (6, -18) {$15$};
\node at (8, -18) {};
\node at (10, -18) {};
\node at (12, -18) {$6$};
\node at (14, -18) {};
\node at (16, -18) {};
\node at (18, -18) {$1$};

\node at (-19, -19) {$1$};
\node at (-17, -19) {$1$};
\node at (-15, -19) {};
\node at (-13, -19) {$6$};
\node at (-11, -19) {$6$};
\node at (-9, -19) {};
\node at (-7, -19) {$15$};
\node at (-5, -19) {$15$};
\node at (-3, -19) {};
\node at (-1, -19) {$20$};
\node at (1, -19) {$20$};
\node at (3, -19) {};
\node at (5, -19) {$15$};
\node at (7, -19) {$15$};
\node at (9, -19) {};
\node at (11, -19) {$6$};
\node at (13, -19) {$6$};
\node at (15, -19) {};
\node at (17, -19) {$1$};
\node at (19, -19) {$1$};

\node at (-20, -20) {$1$};
\node at (-18, -20) {$-1$};
\node at (-16, -20) {$1$};
\node at (-14, -20) {$6$};
\node at (-12, -20) {$-6$};
\node at (-10, -20) {$6$};
\node at (-8, -20) {$15$};
\node at (-6, -20) {$-15$};
\node at (-4, -20) {$15$};
\node at (-2, -20) {$20$};
\node at (0, -20) {$-20$};
\node at (2, -20) {$20$};
\node at (4, -20) {$15$};
\node at (6, -20) {$-15$};
\node at (8, -20) {$15$};
\node at (10, -20) {$6$};
\node at (12, -20) {$-6$};
\node at (14, -20) {$6$};
\node at (16, -20) {$1$};
\node at (18, -20) {$-1$};
\node at (20, -20) {$1$};

\end{tikzpicture}
\]
From the perspective of these notes, the most important thing to
notice about this picture is that its vanishing behaviour 
approximates the earlier pictures, but is markedly simpler.

There exists an object (Lusztig's form of the quantum group of
$\mathfrak{sl}_2$ specialised at a $p^{th}$-root of unity) whose
simple modules have characters determined by the vanishing
behaviour in the above picture. This is a shadow of the fact that the
quantum group at a root of unity imitates the representation theory of
a reductive algebraic group in characteristic $p$, but is simpler.

\begin{remark}
The analogue of the questions discussed in this note for quantum group was
discussed by Wolfgang Soergel at the Current Developments in 1997. (More precisely, due to
the intervention of a border agent they were not discussed; but that
is another story. However
the notes are available \cite{SoergelCDM}.)
\end{remark}

\subsection{Themes of these notes} This discussion of
$\SL_2$ is intended to suggest two themes, which are major motivations
for the considerations in these notes.

\emph{Self-similarity:} We expect fractal-like behaviour in modular
representation theory; both on a numerical level (dimensions,
characters); and on a categorical level.  This behaviour is often tied to the Frobenius
endomorphism
\[
\Fr : G \to G
\]
which implies that the category of representations of $G$ is
self-similar. Indeed, it contains copies of itself given by the
essential images of the Frobenius twist and their variants.\footnote{A
  particularly important variant is the functor $V \mapsto \St
  \otimes V^{(1)}$, where $\St$ denotes the Steinberg representation
  ($L_{p-1} = \nabla_{p-1}$ for $\SL_2$).} 
We suggest that this similarity is connected to the self-similarity of
affine Weyl groups, which contain many isomorphic copies of
themselves.

\emph{Philosophy of generations:} We seek ``blueprints'' for the
fractal-like behaviour, which behave in a simpler fashion. The
archetypal example of this is Pascal's triangle for Gau{\ss}ian
binomial coefficients, discussed in the previous section. These often
involve deformations of our categories (either over polynomial rings,
or to characteristic 0). This philosophy has first been enunciated by
Lusztig \cite{LusztigGen}, where he shows that simple characters display
a ``generational behaviour''. It is also pursued in 
\cite{LW,LWbilliards} for tilting characters, where it is a very
useful guiding principle. Despite its usefulness, we still
lack a rigorous definition of generations (aside from the case of
quantum groups at roots unity, which provide ``generation 1''). A second aim of these
notes is to point out that the self-similarity of affine Weyl groups
suggests a possible approach to a rigorous definition.

\begin{remark}
(Highly speculative!) Following on from Remark \ref{rem:fractal termination} one can hope
to replace $G$ with an object which has simple characters displaying a
genuinely fractal nature. It seems likely that such an object is given
by the projective limit
\[
\hat{G} := \lim_{\leftarrow} G
\]
over the tower of Frobenius maps. I believe this is a perfectoid space
in Scholze's sense \cite{Scholze}. At least one can make a sensible
guess as to what its algebra of distributions is, and study its
representation theory that way. Its representation theory is potentially tied to
the generic cohomology of \cite{CPSvdK}.
\end{remark}

\section{Philosophy of deformations} \label{sec:deformations}

In the following I will try to explain the philosophy of deformations
as it applies to representation theory, and give a few examples.

\subsection{Deformations}
The notion of deformation is fundamental to algebraic geometry. We are
particularly interested in deformations relating smooth and singular
varieties. This is encapsulated in the classic picture
\[
\includegraphics[width=6cm]{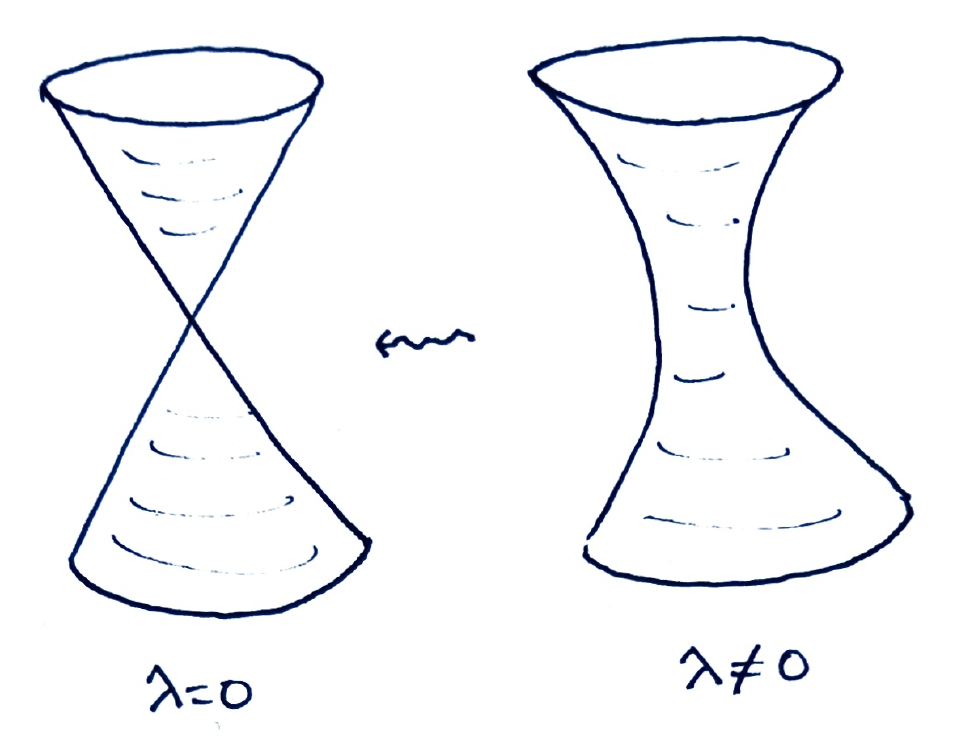}
\]
of a smooth affine conic degenerating to a singular one. In algebraic
geometry the relationship between smooth and singular goes both
ways. For example, smooth varieties are used to construct interesting
objects on singular varieties (e.g. ``vanishing cycles''), or singular
varieties are used to understand smooth varieties (e.g. in the
Gross-Siebert program in mirror symmetry).

One can draw a similar picture in the modular representation theory of finite
groups. Consider a finite group $G$.
Then one can think of $\Rep \ZM G$ as being a
family of categories lying over the line $\Spec \ZM$ with
generic fibre $\Rep \QM G$ (which is semi-simple, therefore
``smooth'') and other fibres $\Rep \FM_p G$ (which is never of finite
homological dimension if $p$ divides the order of $G$, and is
therefore ``singular'' at these points):
\[
\includegraphics[width=8cm]{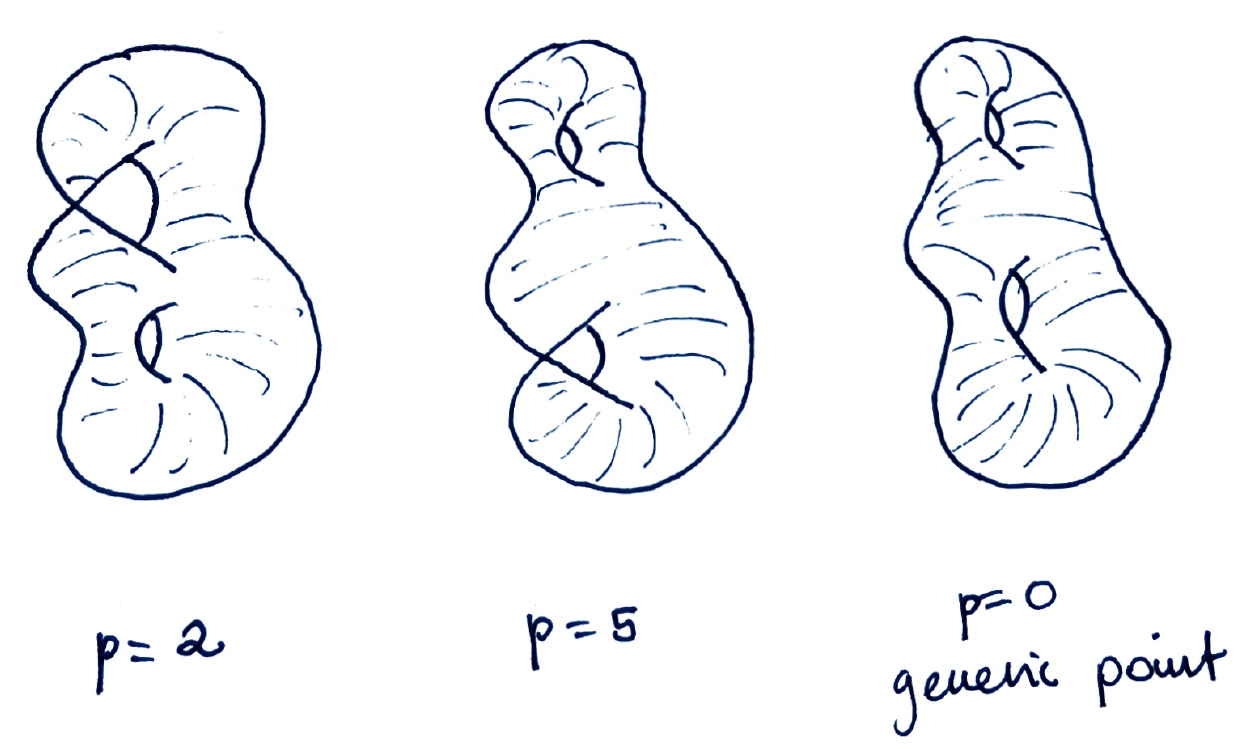}
\]
The reader is cautioned not to take this picture literally. A
slightly more accurate picture would have finitely many reduced points
(the simple representations of $G$ over $\QM_p$) degenerating to some
complicated non-reduced scheme at the singular fibre.
In representation theory the flow of information is almost always from the
``smooth'' to the ``singular''.

The above setting of modular representations of finite groups is
extraordinarily complicated, and there are few
groups where we understand well what happens at the ``singular
fibre''. However quantum objects (Hecke algebras, quantum groups) give
us simpler problems which we can attempt to study (both for their own
interest, and in the hope that they may shed some light on modular
representation theory).

\subsection{The Hecke algebra}
The simplest example of this phenomenon is the Hecke algebra of the
symmetric group. Recall that the symmetric group $S_n$ has a presentation with
generators the $i^{th}$ simple transposition $s_i = (i,i+1)$, for $i = 1, \dots, n-1$, and
relations:
\begin{align}
  \label{eq:Sym}
  s_i^2 &= \id, \\
  s_is_{i+1}s_i &= s_{i+1}s_is_{i+1}, \\
s_is_j &= s_js_i  \qquad \text{if $|i-j| > 1$}.
\end{align}
The Hecke algebra $\Hecke_{S_n}$ of the symmetric group is a $\ZM[v^{\pm 1}]$-algebra
with generators $\delta_i$, for $i = 1, \dots, n-1$, and relations:
\begin{align}
  \label{eq:Hecke}
  \d_i^2 &= (v^{-1}-v)\d_i + 1, \\
  \d_i\d_{i+1}\d_i &= \d_{i+1}\d_i\d_{i+1}, \\
\d_i\d_j &= \d_j\d_i  \qquad \text{if $|i-j| > 1$}.
\end{align}
Setting $v = 1$ we recover the relations of the symmetric group. 
The Hecke algebra is free over $\ZM[v^{\pm 1}]$ with a ``standard
basis'' $\{ \delta_x \; | \; x \in S_n \}$, which becomes  the standard
basis of the group algebra when we specialise $v := 1$. It is useful to think of the
Hecke algebra as a flat deformation of the group algebra of the
symmetric group over $\Spec \ZM[v^{\pm 1}]$. 

In the group algebra of $S_n$ one has the central element
\[
e = \sum_{x \in S_n} x
\]
which satisfies
\begin{equation}
  \label{eq:sne}
e^2 = |S_n| \cdot e = n! \cdot e.  
\end{equation}
One can use this to show that $kS_n$ is semi-simple if and only if
$n!$ is non-zero in $k$.

The analogue in $\Hecke_{S_n}$ of $e$ is
\[
e_v = \sum_{x \in S_n} v^{N-\ell(x)}\delta_x
\]
where $\ell$ is the usual length a permutation, and $N = n(n-1)/2$ is
the length of the permutation that switches $1$ and $n$, $2$ and
$n-1$, etc. The analogue of \eqref{eq:sne} is
\begin{equation}
  \label{eq:Hne}
e_v^2 = [n]_v! \cdot e_v.
\end{equation}
where $[n]_v!$ is as in \eqref{eq:factor}. Similarly to the case of $S_n$, one can use \eqref{eq:sne} to show that,
for any field and non-zero $\lambda \in k$, the specialisation
(``fibre'') $\Hecke_{S_n} \otimes_{v \mapsto \lambda} k$ is
semi-simple if and only if the evaluation of $[n]_v!$ at $v = \lambda$
is non-zero in $k$.
If we specialise $v \mapsto \lambda \in \CM$ then we deduce from
\eqref{eq:factor} that $\Hecke_{S_n} \otimes_{v \mapsto \lambda} \CM$
is semi-simple if and only if $\lambda$ is not an $m^{th}$-root of
unity, for $m = 2, \dots, n$.

Now $\Spec \ZM[v^{\pm 1}]$ is of dimension 2 and hence has two directions: a ``geometric''
direction (corresponding to specialisations of $v$); and an
``arithmetic'' direction (corresponding to specialisations of $\ZM$ to
various fields. There has been significant progress in our
understanding of the geometric direction (e.g. the LLT conjecture \cite{LLT},
proved by Ariki \cite{Ariki}). Progress in the arithmetic direction has
been more modest, but the picture with Hecke algebras has been a
fruitful source of problems and motivation (e.g. the James conjecture
\cite{JC,Wexplosion}).

%\todo{draw picture from CDM notes}

\subsection{Higher-dimensional bases} In algebraic geometry one also
encounters deformations over higher dimensional bases.
It is a fundamental observation (developed by Jantzen \cite{JModuln},
Gabber-Joseph \cite{GJ},
Soergel \cite{SHC}, \dots ) that a similar picture exists in the
infinite-dimensional representation theory of complex semi-simple Lie
algebras.

Let $\mathfrak{g}$ denote a complex semi-simple Lie algebra, with Cartan
and Borel subalgebra $\hg \subset \bg \subset \mathfrak{g}$, root system $\Phi \subset \hg^*$ and Weyl
group $W$. Let $\OC$ denote the BGG category of certain
infinite-dimensional representations of $\mathfrak{g}$. Recall that we
have a decomposition via central character
\[
\OC = \bigoplus_{ \lambda \in \hg^*/(W \cdot)} \OC_\lambda
\]
and the most complicated pieces are $\OC_\lambda$, where $\lambda$ is
dominant integral.

Let us fix $\lambda$, and suppose we wish to study
$\OC_\lambda$. Deformed category $\OC$ defines a deformation of $\OC_{\lambda}$
over $\Spec \hat{R}$, where $\hat{R}$ denotes the completion of the
symmetric algebra on $\hg$ at the maximal ideal generated by the point
$\lambda$; it should be thought of as 
an infinitesimal neighbourhood of $\lambda \in \hg^*$. The picture of
the fibres of deformed category $\OC$ are as follows:
\[
\includegraphics[width=9cm]{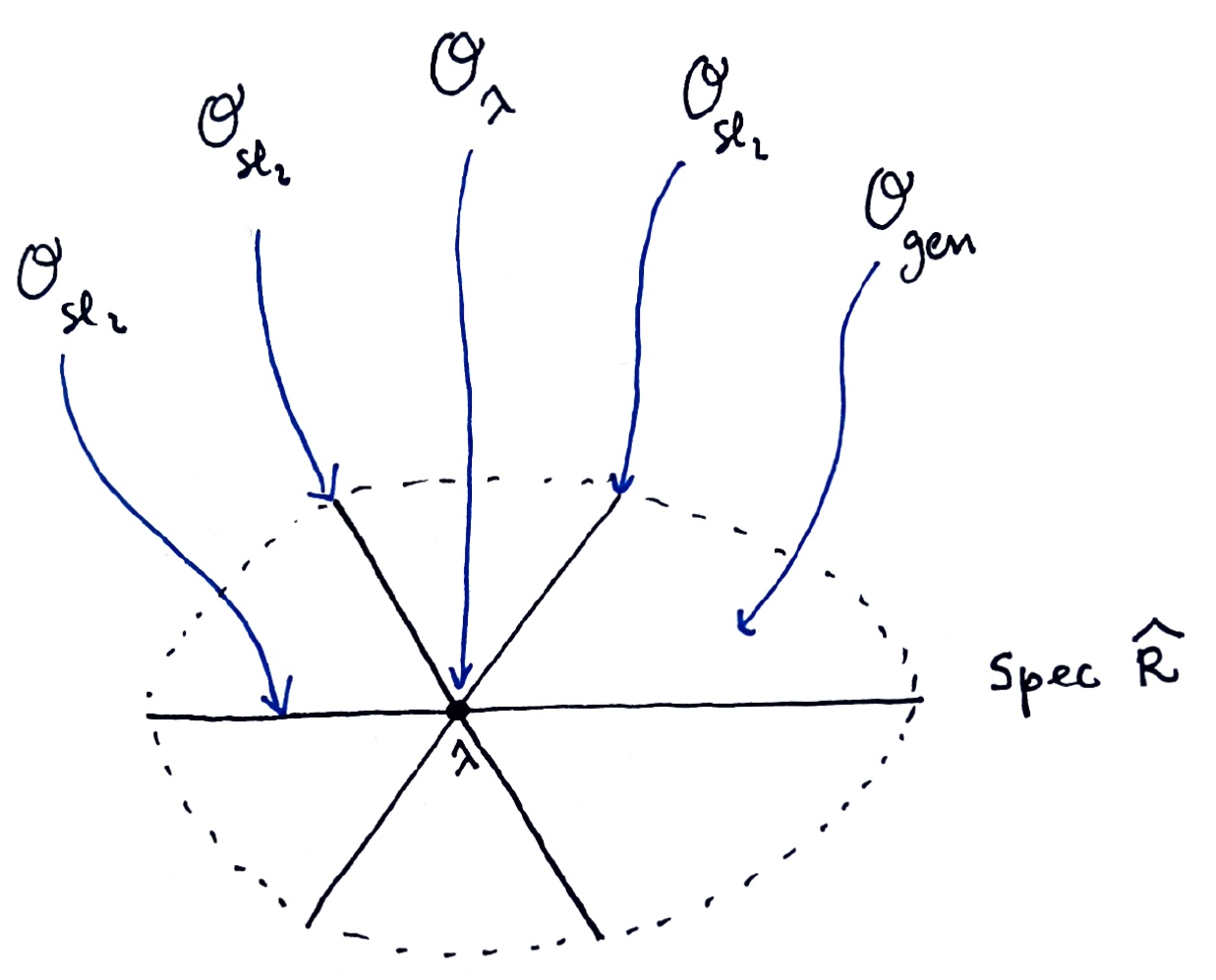}
\]
Here we have drawn the picture for $\sl_3$. The lines in $\Spec
\hat{R}$ are the root hyperplanes passing through 
$\lambda$. The fibres are as follows:
\begin{enumerate}
\item $\OC_\lambda$ is the fibre over the central point; it is
  the category that we are interested in.
\item $\OC_{\sl_2}$ is the fibre over all points that lie in precisely
  one root hyperplane. Each such fibre is equivalent to a direct sum
  of blocks of category $\OC$ for $\sl_2(\CM)$.
\item $\OC_{\gen}$ is the fibre over points which do not lie on any
  root hyperplane. Each such fibre is a semi-simple category, equivalent to $|W|$
  copies of the category of finite-dimensional $\CM$-vector spaces.
\end{enumerate}
The reason why this picture is so powerful is that it allows one to
see $\OC_\lambda$ (a very complicated category) as glued together out
of simple categories (copies of vector spaces, and categories $\OC$
for $\sl_2$, which may be described explicitly).

\begin{remark}
  A similar setting occurs in the localisation theorem in equivariant
  cohomology, where the cohomology injects into equivariant cohomology
  of fixed points, with image described by conditions coming from
  one-dimensional orbits. This similarity is more than a coincidence
  and has been exploited in Soergel's work.
\end{remark}

\begin{remark}
  This example may appear unrelated to the earlier examples
  from modular representation theory. In fact, the examples from
  modular representation theory were an important motivation for Jantzen to
  introduce deformed category $\OC$: he was seeking a way to apply the
  ideas from his study of Weyl modules in characteristic $p$ (and in particular their Jantzen
  filtrations) to category $\OC$.
\end{remark}

% \subsection{Janzten filtrations}

% \begin{remark}
%   It is a remarkable aspect of this story that the deformation picture
%   above for category $\OC$ is in fact motivated by the characteristic
%   $p$ picture.
% \end{remark}

% \subsection{James conjecture}

% Let $\OM = \ZM_p[\e]$, where $\e$ is a $p^{th}$ root of unity. This is
% a ramified extension, and the unique prime lying over $(p) \subset
% \ZM_p$ is $(\e-1) \subset \OM$. \todo{Check again!} Let $\OM[[t]]$
% denote the completion of $\OM[q]$ at the ideal $(q - \e)$, so that
% \[
% \OM[[t]]/(t) \simto \OM[q]/(q-\e) = \OM.
% \]

% We get a family of categories over $\Spec \OM[[t]$ (see CDM notes for
% a picture).

\subsection{The anti-spherical category} In this section we discuss
the various localisations of the anti-spherical category. The goal is
to try to understand the anti-spherical category via the philosophy of
deformations.

\begin{remark}The
case of the anti-spherical module was the principle motivating example when writing these notes. We
thus feel it is important not to defer the discussion of this case
until the very end, which would perhaps be logical from a
mathematical perspective.
The discussion of this section is necessarily more technical than the
previous sections. The reader unfamiliar with the Hecke category
should defer the reading of this section until they have read the
section on the Hecke category. It would also be very helpful to have read the
introductions of \cite{RWpcan} and \cite{LWasph}.
\end{remark}

Let $\HC$ denote the Hecke category of an affine Weyl group $W$. As a
right module over itself it is linear over $R$, where $R$ denotes the
symmetric algebra on $\Chi_{\rot}$, the characters of the extended
torus
\[
T_\rot := T \times \CM^*_\rot
\]
where $T$ is a maximal torus of a reductive group $G$, and
$\CM^*_\rot$ is the loop rotation $\CM^*$ inside the loop group
$\CM^*_\rot  \ltimes G((t))$. Thus we have
\[
\Chi_{\rot} := \ZM \delta \oplus \Chi
\]
where $\Chi$ denotes the character lattice of $T$. Here $\delta$
(which will play a crucial role below), is the identity character of
the loop rotation $\CM^\times_\rot \subset T_\rot$.

Let us consider the Hecke category $\HC_{\ZM_p}$
over the $p$-adic integers $\ZM_p$. The anti-spherical module is by definition the quotient
of additive categories
\[
\AS := \HC_{\ZM_p}/ \langle B_w \; | \; w \notin {}^f W \rangle_{\oplus}
\]
where ${}^f W$ denotes minimal coset representatives for the right
action of the finite Weyl group $W_f$.
This is naturally a right $ \HC_{\ZM_p}$-module, as $\langle B_w \; |
\; w \notin {}^f W \rangle_{\oplus}$ is an ideal. (We have not covered
all notation, see \cite[\S 1.3]{RWpcan} for more details.)

When we form this quotient we are forced to kill certain symmetries
of $\HC_{\ZM_p}$. For example, if $\alpha_s \in \Chi$ denotes a simple
root then $\alpha_s$ (viewed as an endomorphism of $\HC$) factors
through $B_s$, and hence is zero in $\AS$. Hence all of $\Chi$
necessarily acts by zero.\footnote{We are tacitly assuming $p$ is not too
  small, so the weight and root lattices coincide when tensored with $
 \ZM_p$.} If we denote
by $\mathbb{1}_{\AS}$ the image of the unit in $\HC_{\ZM_p}$ in $\AS$,
then it turns out that
\begin{equation} \label{eq:unit-end}
\End^{\bullet}(\mathbb{1}_{\AS}) = \ZM_p[\delta]
\end{equation}
where $\delta$ has degree 2. (In other words, the obvious
endomorphisms $\alpha_s$ just discussed are killed, and the rest
survives.)

\begin{figure} \caption{Localisations and specialisations of
    $\AS$.} \label{fig:locspec}
\[
\begin{array}{c}
\begin{tikzpicture}[xscale=1.3,yscale=1]
\draw[gray] (-4,0) -- (4,0);
\draw[gray] (0,-4) -- (0,4);
\draw[gray] (-4,4) -- (4,4) -- (4,-4) -- (-4,-4) -- (-4,4);
\node at (-2,3) {$\FM_p[\delta^{\pm 1}]\otimes_{\ZM_p[\delta]} \AS$};
\node[rotate=90] at (-2,2.25) {$\cong$};
\node at (-2,1.5) { $\bigoplus_{x \in {}^r W} \FM_p[\delta^{\pm 1}]\otimes_{\ZM_p[\delta]} \AS_x$};
\node at (-2,.75) {(\cite{Hazi}, \text{these notes})};
\node at (2,3) {$\QM_p[\delta^{\pm 1}]\otimes_{\ZM_p[\delta]} \AS$};
\node[rotate=90] at (2,2.25) {$\cong$};
\node at (2,1.5) { $\bigoplus_{x \in {}^f W} \QM_p[\delta^{\pm 1}]-\textrm{mod}$};
\node at (2,.75) {(semi-simple, \cite{LWasph})};
\node at (-2,-.75) {$\FM_p\otimes_{\ZM_p[\delta]} \AS$};
\node[rotate=90] at (-2,-1.5) {$\cong$};
\node at (-1.2,-1.5)  {\tiny (degrading)};
\node at (-2,-2.25) {$\Tilt_0 (G^\vee_{\FM_p})$};
\node at (-2,-3) {(\cite{RWpcan})};
\node at (2,-.75) {$\QM_p\otimes_{\ZM_p[\delta]} \AS$};
\node[rotate=90] at (2,-1.5) {$\cong$};
\node at (2.8,-1.5)  {\tiny (degrading)};
\node at (2,-2.25) {$\Tilt_0 (U_\epsilon(\mathfrak{g}^\vee))$};
\node at (2,-3) {(\cite{SoergelCDM, AB})};
\node at (-2,-4.5) {$\FM_p$};
\node at (2,-4.5) {$\QM_p$};
\node at (-5,2) {$\delta \ne 0$};
\node at (-5,-2) {$\delta = 0$};
\end{tikzpicture}
\end{array}
\]
\end{figure}

The crucial equality \eqref{eq:unit-end} allows us to view $\AS$ as a
family over the two-dimensional base $\ZM_p[\delta]$. Figure
\ref{fig:locspec} depicts the various localisations and
specialisations of $\AS$. Here are some more details:
\begin{enumerate}
\item ($\ZM_p[\delta] \rightsquigarrow \FM_p$): This is the source of our interest in
  $\AS$. Let $G^\vee_{\FM_p}$ denote the (split) Langlands dual group
  over $\FM_p$, and let $\Tilt_0 (G^\vee_{\FM_p})$ denote the full
  subcategory of tilting modules in the same block as the trivial
  module (the ``principal block''). We have a functor\footnote{This functor is constructed by
    combining \cite[Theorem 7.4]{AMRW} with the main result of \cite{AR}. It is
    constructed for $G = GL_n$ in \cite{RWpcan}. We are tacitly
    making further assumptions on $p$ \dots}
\[
\phi : \FM_p\otimes_{\ZM_p[\delta]} \AS \to \Tilt_0 (G^\vee_{\FM_p})
\]
which satisfies $\phi \circ [1] \cong \phi$ and
\[
\phi : \bigoplus_{m \in \ZM} \Hom( \ES, \FS[m]) \simto 
\Hom(\phi(\ES),\phi(\FS)).
\]
(In other words, $\phi$ induces an equivalence between $\Tilt_0
(G^\vee_{\FM_p})$ and the anti-spherical module specialised over
$\FM_p$ with $\delta = 0$, once we forget the grading on the latter.)
This allows one to deduce a character formula for tilting modules in
terms of the $p$-Kazhdan-Lusztig basis.
\item ($\ZM_p[\delta] \rightsquigarrow \QM_p$):  Let $\mathfrak{g}^\vee$ denote the
  complex semi-simple Lie algebra which is Langlands dual to $G$. Let
  $U_\epsilon(\mathfrak{g}^\vee)$ denote the quantum group at a
  $p^{th}$ root\footnote{Actually, the root of 1 that one takes doesn't
  matter below as long as it isn't too small, as the principal blocks
  at different roots of $1$ are all equivalent!}  of $1$ and let
$\Tilt_0(U_\e(\mathfrak{g}^\vee))$ denote its full subcategory of
tilting modules. We have a functor\footnote{This functor may be
  constructed by combining the main result of \cite{ABG} with the
  parabolic/Whittaker Koszul duality of \cite{BY}.}
\[
\phi : \QM_p\otimes_{\ZM_p[\delta]} \AS \to \Tilt_0(U_\epsilon(\mathfrak{g}^\vee))
\]
which is an ``equivalence up to degrading'' (as above). Because our
coefficients are of characteristic zero, the combinatorics of the
left-hand side are governed by Kazhdan-Lusztig theory. In this way one
obtains a formula for the characters of tilting modules in terms of
the anti-spherical module, recovering the result of \cite{Wkipp,
  SoergelCDM}.
\item ($\ZM_p[\delta] \rightsquigarrow \QM_p[\delta^{\pm 1}]$): Here the category is semi-simple,
  with objects parametrised by dominant alcoves. In terms of the
  previous equivalence, this is the statement that the principal block
  of $U_\epsilon(\mathfrak{g}^\vee)$ is semi-simple away from roots of
  unity. This semi-simplicity at the generic point is
  a key feature of \cite{LWasph}.
\item ($\ZM_p[\delta] \rightsquigarrow \FM_p[\delta^{\pm 1}]$):  This is the starting point of this
  work. In a remarkable recent work, Hazi \cite{Hazi} showed that one has an
  equivalence
  \begin{equation}
    \label{eq:hazi ss}
R[\delta^{-1}] \otimes_{\FM_p} \HC_{\FM_p} \simto \bigoplus_{ w \in
  {}^r W} R[\delta^{-1}] \otimes_{\FM_p} \HC_{\FM_p}.    
  \end{equation}
(At first pass, this is an equivalence of additive categories. We
emphasise that it is not monoidal. We will
see later on that it is naturally interpreted as an equivalence of
certain bimodule categories for the Hecke categories of $W$ and a
reflection subgroup, see Theorem \ref{thm:sm} as
well as \S \ref{sec:hypbim}, for a geometric (resp. algebraic)
version).
 In other words, $\HC$ demonstrates some
self-similarity 
after localisation. The equivalence 
  \begin{equation}
    \label{eq:as ss}
\FM_p[\delta^{\pm 1}]\otimes_{\ZM_p[\delta]} \AS
\cong 
\bigoplus_{x \in {}^r W} \FM_p[\delta^{\pm 1}]\otimes_{\ZM_p[\delta]} \AS
  \end{equation}
is a reasonably easy consequence of \eqref{eq:hazi ss}, as we try to
explain in \S \ref{sec:locas}.
\end{enumerate}

The key technical motivations for these notes are the following:
\begin{enumerate}
\item To place \eqref{eq:hazi ss} and \eqref{eq:as ss} in a broader
  consequence of localisation to ``good'' reflection subgroups. (The
  background for this occurs in \S \ref{sec:affine} and the details
  are in \S \ref{sec:heckeloc}.)
\item To explain the geometric meaning of \eqref{eq:hazi ss} and
  \eqref{eq:as ss}. When our reflection subgroup is a parabolic subgroup,
  such equivalences may be understood in terms of hyperbolic
  localisation. However when we are in characteristic $p$ (and in
  particular to get the equivalences \eqref{eq:hazi ss} and
  \eqref{eq:as ss}) we need
  Smith theory. Here we follow work of Treumann
  \cite{Treumann} and 
  Leslie-Lonergan \cite{LL}.
\end{enumerate}

\begin{remark}
  It was via the embedding
\[
\AS \into \QM_p[\delta] \otimes_{\ZM_p[\delta]} \AS 
\]
that the author was able to compute many new tilting characters for
$\SL_3$, leading to the billiards conjecture of
\cite{LWbilliards}. (See \cite[\S 3]{LWbilliards} for more on how this
was done.) More recently, Thorge Jensen has implemented the embedding
\[
\AS \into \QM_p[\delta^{\pm 1}] \otimes_{\ZM_p[\delta]} \AS
\]
and discovered that the resulting algorithm is much quicker.\footnote{The
calculations that led to \cite{LWbilliards} took 10 months. Now T. Jensen
can repeat them in two weeks.} (Roughly
speaking, it is easier to compute in a semi-simple category than in a
category given as a quiver with relations.) It is an
interesting question as to whether the ``new'' localisation
\[
\AS \into \ZM_p[\delta^{\pm 1}] \otimes_{\ZM_p[\delta]} \AS
\]
will have computational consequences.
\end{remark}

\begin{remark} One of the reasons that I like the
  above picture is that it suggests a possible definition of the
  principal block for ``higher
  generations''. Recall that such an object should complete the
  dots in the sequence
  \[
\Rep \mathfrak{g^\vee}, \Rep U_\epsilon(\mathfrak{g^\vee}),  \dots , \Rep G^\vee_{\FM_p}.
\]
with the first (semi-simple) category being generation 0, the second
generation 1, and the last generation $\infty$ (see \cite[\S
2]{LWbilliards}). Let us imagine that we have a category $\AS^g$ which
is a deformed version of the principal block in generation $g$. Recent
calculations of Hazi and Elias suggest  that the after localiation the top left of the above digram will
  consist of a number of copies of $\AS^{g-1}$.  It might be possible
  to reverse this observation, and use the above deformation
  philosophy to \emph{define} (a deformed version of) $\AS^g$ by
  taking a direct sum of copies of $\AS^{g-1}$. 
  \end{remark}

% \section{Notation} \label{sec:not}

% graded dimension, categories of modules, bimodules

% $A$-bimodule means $A\bmod A$, Wherease $(A,B)$-bimodule means blah.

% ${}^f W$, we write ${}^s W$ for ...

% Idempotent completion is denoted $\AS^e$.

% $R_x$: twisted bimodules.

% comma! for derived categories, cohomology etc.??

% \subsection{Notational challenges} Does $\Hom^\bullet$ denote morphisms
% between Soergel bimodules. Degrees? 

\section{Affine reflection groups and Kazhdan-Lusztig theory} \label{sec:affine}

In this section we review the theory of Coxeter groups, paying
particular attention to the case of affine Weyl groups and their
reflection subgroups. We briefly review Kazhdan-Lusztig theory and
give examples of positivity phenomona tied to reflection subgroups.

\subsection{Affine reflection groups} \label{sec:affine groups}

Let $V$ denote a finite-dimensional real Euclidean vector space and
consider a discrete subgroup $W$ of the affine transformations of $V$
which is generated by reflections. We will refer to such a group as an \emph{affine
reflection group} (although it might, for example, be finite). Let us
briefly review the beautiful theory that allows us to understand $W$
via its action on $V$.

We consider the set $\reflect$ of hyperplanes $H$ fixed by some element of
$W$ (then necessarily an affine reflection). Elements of $\reflect$
are called \emph{reflecting hyperplanes}. Of course $W$ acts on the 
arrangement $\reflect$. Our discreteness assumption guarantees that
$\reflect$ is a locally-finite arrangement. (``Locally finite'' means that
any point has a neighbourhood which only meets finitely many
hyperplanes.) The connected components of the
complement
\[
V \setminus \bigcup_{H \in \reflect} H
\]
of all reflecting hyperplanes are called
\emph{alcoves}.

We denote the set of alcoves by $\AC$. Any alcove $A
\in \AC$ is open, and its closure is homeomorphic to a product of:
simplices of various dimensions; copies of $\RM_{\ge 0}$; and copies
of $\RM$. (The most familiar setting is when our group is an
irreducible affine reflection group, in which case the closure of $A$ is homeomorphic
to a simplex.) On the set of alcoves we can consider something like a metric:
\[
\delta(A,B) := \# \{ H \in \reflect \; | \; H \text{ separates $A$ and $B$} \}.
\]

Let us make a fixed, arbitrary choice of an alcove $A_0 \in \AC$. Any reflecting
hyperplane which intersects the closure of $A_0$ in codimension one is
called a \emph{wall} of $A_0$. We denote by $\reflect_0$ the set of walls
of $A_0$, and $S$ the set of reflections in the hyperplanes $\reflect_0$.

The two fundamental facts which get the theory started are:
\begin{gather}
  \label{eq:generation}
  \text{$W$ is generated by $S$;} \\
  \label{eq:fundamental}
  \text{$\overline{A}_0$ is a fundamental domain for the $W$-action on
    $V$.}
\end{gather}
(A sketch of why these two properties hold: First one shows that any alcove $A$ can be moved by $W$ back to $A_0$,
by showing that $\delta(A_0, A)$ decreases by 1 when we act by a reflection
in a wall of $A_0$ separating $A_0$ and $A$. This implies that any
reflection (and hence any element) in $W$ can be expressed as a word
in $S$, which is \eqref{eq:generation}. By induction, it also implies that
\[
\delta(A_0, xA_0) = \ell(x)
\]
where $\ell(x)$ denotes the minimal length of an expression for $x$ in
the generators $S$, from which \eqref{eq:fundamental} follows.)
From \eqref{eq:fundamental} it follows that our choice of $A_0$ leads to a bijection:
\begin{equation}
  \label{eq:bijection}
  W \simto \AC : x \mapsto x(A_0).
\end{equation}

We now explain why our choice of $A_0$ also gives rise to a Coxeter presentation
of $W$. Consider two walls $H, H'$ of $A_0$, and let $s$ and $t$
denote the reflections in these walls. If $H$ and $H'$ meet then $st$
is a rotation about the axis $H \cap H'$ of some finite order
$m_{st}$. If $H$ and $H'$ do not meet then they are parallel, and $st$
is a translation (then of infinite order, in which case we set $m_{st} =
\infty$.)

The third fundamental fact is that $W$ admits a \emph{Coxeter presentation}:
\begin{gather}
  \label{eq:presentation}
W = \left \langle s \in S \; \middle | \;
\begin{array}{c} s^2 = 1 \text{ for all $s \in S$}\\ (st)^{m_{st}} =
                              1 \text{ for all $s,t \in S$}\end{array}
\right \rangle
\end{gather}
In other words, all relations in $W$ amongst the generator $S$ can be deduced from the
``obvious'' relations of the previous paragraph. Again, the action of
$W$ on $V$ and the set of alcoves can be used to give a simple
proof of \eqref{eq:presentation}.

\begin{remark}
A canonical reference for the material above is \cite[\S VI.2]{Bourbaki}. The
book of Humphreys \cite{Humphreys} is also useful.
\end{remark}

% \begin{remark}
%   \todo{Problem with my notation for reflecting hyperplanes and for the Hecke category.}
% \end{remark}

\subsection{Reflection subgroups} \label{sec:ref1}

A \emph{reflection subgroup} of $W$ is a subgroup generated by
reflections. Being itself an affine reflection group, it is clear that
the above theory applies just as well to any reflection
subgroup (e.g. it is itself a Coxeter group, \dots). However the theory just outlined also allows us to study the
interaction between $W$ and its reflection subgroups, as we now explain.

Let $W_r \subset W$ denote a reflection subgroup.
Let us denote by $\reflect_r$ and $\AC_r$ the reflecting hyperplanes (fixed
points of reflections) and alcoves for $W_r$ acting on $V$. Then the
hyperplanes of $\reflect_r$ are amongst those of $\reflect$, because $W_r$ is a
subgroup of $W$. Also, having fixed our base alcove $A_0$, we have a
canonical choice for a base alcove $A_r \in \AC_r$, namely
\[
A_r := \text{the unique alcove in $\AC_r$ containing $A_0$.}
\]

Because $A_r$ is a fundamental domain for the action of $W_r$ it
follows that we can transport any alcove $A \in \AC$ into $A_r$ in a
unique way using just $W_r$. Interpreted via our bijection \eqref{eq:bijection}, this
is telling us that we have a canonical bijection
\[
W_r \setminus W \simto \{ A \in \AC \; | \; A \subset A_r \}.
\]
Moreover, this bijection is well-adapted to the Coxeter structure. If
we identify the right hand side with a subset
\[
{}^rW \subset W
\]
via
\eqref{eq:bijection} then ${}^rW$ consists of representatives of each
right coset of \emph{minimal length}. By inverting these elements we
also obtain a set
\[
W^r \subset W
\]
of minimal length  of each left coset.

\begin{remark}
The case where $W_r$ is
generated by a subset of $S$ returns the (probably more familiar) minimal length
coset representatives of standard parabolic subgroups.  
\end{remark}

\begin{remark} It is useful to keep in mind that
  $W_r \subset W$ is rarely a normal subgroup: it is rare
  for a collection of reflecting hyperplanes to be stable under $W$.
\end{remark}

\subsection{General Coxeter groups}  \label{sec:ref2}

We briefly discuss general Coxeter groups, their reflection
representations and reflection subgroups. The statements below without references are proved in \cite[\S VI]{Bourbaki} and \cite[\S
5]{Humphreys}.

The above discussion shows that affine reflection groups are examples
of Coxeter groups, i.e. groups with a distinguished set $S$ of
generators, and a presentation of the form \eqref{eq:presentation},
where now
\[
m_{st} \in \{ 2, 3, \dots, \infty\} \quad\text{for $s, t \in S$ with
  $s \ne t$}
\]
is arbitrary.

Let $V$ be a finite-dimensional vector space. A \emph{reflection} is
a linear transformation of $V$ which fixes a hyperplane and sends a
vector to its negative. A
\emph{reflection pair} is pair $(\gamma^\vee, \gamma) \in V^* \times
V$ such that $\langle \gamma^\vee, \gamma \rangle = 2$. Any reflection
pair $(\gamma, \gamma^\vee)$ gives rise to a reflection $s$ via the
formula
\[
s(v) = v - \langle \gamma^\vee, v \rangle \gamma.
\]
(Indeed, $s$ fixes the zero set of $\gamma^\vee$ and sends $\gamma$ to
its negative.) A \emph{based
  reflection} is a reflection given by a (fixed) reflection
pair. Given a based reflection, we refer to $\gamma$ and $\gamma^\vee$ as
its \emph{root} and \emph{coroot} respectively.

\begin{remark}
  Outside of characteristic 2, the condition $\langle \gamma^\vee,
  \gamma \rangle = 2$ forces $\gamma^\vee$ and $\gamma$ to be
  non-zero, and hence the fixed points of $s$ is precisely the zero
  set of $\gamma^\vee$, and $\gamma$ spans its $-1$-eigenspace. This
  determines $\gamma$ and $\gamma^\vee$ up to simultaneous (inverse)
  scaling. It follows that there is not a great
  difference between reflections and based reflections. In
  characteristic 2 however things are tricky: reflections may be the
  identity or unipotent.
\end{remark}

% We now discuss reflection representations of Coxeter groups.
% Let $V$ be a finite-dimensional vector space.  A \emph{reflection} $s$ of $V$ is a linear
% transformation which fixes a hyperplane and sends a vector to its
% negative. The fixed hyperplane determines a vector $\gamma^\vee \in
% V^*$. The $-1$-eigenspace (a line) is spanned by a vector
% $\gamma$. Both $\gamma^\vee$ and $\gamma$ are well-defined up to
% scalar. If we scale them such that $\langle \gamma^\vee, \gamma \rangle
% = 2$, then $s$ is given by the formula
% \[
% s(v) = v - \langle \gamma^\vee, v \rangle \gamma.
% \]
% (Indeed, the right-hand side of this formula has the same fixed points
% as $s$, and sends $\gamma$ to $-\gamma$.)

Let $(W,S)$ be a general Coxeter group. The \emph{reflections} in $W$ are the conjugates
\[
T = \bigcup_{w \in W} wSw^{-1}
\]
of the set $S$. A \emph{reflection representation} is one in which the
reflections $t \in T$ are given by based reflections. (Thus for every
reflection $t$ we may refer to its root $\gamma$ and coroot
$\gamma^\vee$.)

\begin{remark}
  For finite Coxeter groups the irreducible reflection representations
  are usually unique (at least up to Galois conjugacy, and 
  ambiguity in choices of roots and coroots). For example,
  the natural representation of the symmetric group on $\{ \lambda\in
  \QM^n \; | \; \sum \lambda_i = 0\}$ is the unique irreducible
  reflection representation of $S_n$ over $\QM$, whereas the dihedral group of
  order 10 has two reflection representations. Both are of dimension
  2, are defined over $\QM(\phi)$
  where $\phi$ is the golden ratio, and are swapped by the
  action of the Galois group. Thus, up to adding copies of the trivial representation, reflection
  representations are almost unique for finite Coxeter groups. The situation
  is quite different for infinite Coxeter groups, where one may have non-trivial
  moduli of reflection representations. (This is the case for example
  for the affine Coxeter group of type $A$. A one-dimensional family
  of representations is described
  in \cite[\S 1]{EliasAlgSat}. In fact, the existence of deformations
  of the natural representation is crucial to the results of that paper.)
\end{remark}

Any Coxeter group possesses a distinguished reflection representation,
known as the \emph{geometric representation}. We take
$V$ to be a vector space over $\RM$ with a basis $\{ \alpha_s \; | \; s \in S \}$ and a
$W$-action defined by
\[
s(v) := v - B(\alpha_s,-)\alpha_s
\]
where $B$ is the symmetric bilinear form defined on our basis by
\[
B(\alpha_s, \alpha_t) = -\cos \left ( \frac{\pi}{m_{st}} \right ). 
\]
(In the above language, the root and coroot of each simple reflection
is $\alpha_s \in V$ and $B(\alpha_s, -) \in V^*$ respectively.) This representation is faithful.

\begin{remark}
  The term \emph{geometric representation} is common in the
  literature, but is perhaps unfortunate. In fact, $(W,S)$ has other
  reflection representations which play an important role in the structure theory
  of $V$, and which have an equal right to be called
  ``geometric''. The simplest is the dual of $V$ which is crucial to
  the definition of the Tits cone. Another is the representation
  considered by Soergel \cite{SB} which has the above representation
  as a subrepresentation, and its dual as a quotient. % \todo{We will
    % revisit this issue in ...}
\end{remark}

If we consider the set
\[
\Phi := W \cdot \{ \alpha_s \; | \; s \in S \}
\]
then one gets something like a root system. We have a
decomposition
\[
\Phi = \Phi^+ \sqcup \Phi^-
\]
where $\Phi^+$ consists of those elements of $\Phi$ which can be
expressed as positive linear combinations of the basis vectors $\{
\alpha_s \; | \; s \in S \}$.

If $T \subset W$ denotes the reflection as above, one has a bijection
\[
\Phi^+ \simto T
\]
which sends $\gamma \in \Phi^+$ to an element $s_\gamma$
which acts on $V$ via
\[
s_\gamma(v) = v- B(v,\gamma)\gamma.
\]
(That is, $\gamma$ and $B(\gamma, -)$ are the root and coroot of the
reflection $s_\gamma$.)

A reflection subgroup is a subgroup $W_r \subset W$ generated by a
subset of $T$. Following the intuition gained in the affine case, one
might hope that any such subgroup is again a Coxeter group, with a
canonical set of generators. This is indeed the case, as was proved by
Dyer \cite{DyerR}\footnote{Dyer's proof is short, beautiful and
  well-worth a look!} and Deodhar \cite{DeodharR}.

As above one has minimal coset representatives ${}^r W$ and $W^r$ for
the right and left cosets of $W$. They may be characterised as follows:
\begin{gather} \label{eq:minl}
  {}^r W = \{ x \in W \; | \; x^{-1}(\alpha) > 0 \text{ for all
    $\alpha \in \Phi_r$} \}, \\ 
\label{eq:minr}
  W^r = \{ x \in W \; | \; x(\alpha) > 0 \text{ for all
    $\alpha \in \Phi_r$} \}.
\end{gather}

\subsection{Self-similarity} \label{sec:self-similar}
Our intuition suggests that reflection
subgroups should be ``smaller'' than the original Coxeter group. This
intuition is deceptive, for example one can find reflection
subgroups with infinitely many Coxeter generators inside hyperbolic
reflection groups. Here we point out that affine reflection subgroups
often have natural chains of reflection subgroups which are isomorphic to the
groups themselves.

Let us return to the setting of an affine reflection group acting on a
affine space $V$.  If we fix a base-point $b \in V$ then $V$
becomes a vector space by declaring $0 := b$. Now scaling by a real
number $\ell$ makes sense.

Let us assume that there exists a real number $\ell$ such that
$\ell\reflect \subset \reflect$. There is a trivial case in which all
reflecting hyperplanes pass through $b$. In this case 
$W$ is finite, and any $\ell$ will do. In the more interesting case
when $W$ is infinite we have:

\begin{lem}
If $W$ is infinite, then $\ell$ is an integer.
\end{lem}

\begin{proof}
Our local finiteness assumption implies that the hyperplanes in $\reflect$ consist of finitely
many directions. (Otherwise, $\reflect$ would contain pairs of hyperplanes
intersecting at arbitrarily small angles, which is impossible by the
discussion of the fundamental domain earlier.) In particular, if $W$
is infinite then $\reflect$ contains two parallel hyperplanes.

By considering the line $L$ perpendicular to two such hyperplanes and
passing through $b$ we reduce
to the case of dimension $1$. We may identify $L$ with the real line
and assume that our set of reflecting hyperplanes is of the form
\[
\reflect = \{ \epsilon + \ZM \} \subset \RM
\]
for some $0 \le \epsilon < 1$. This set is stable under multiplication
by $\ell$. In particular
\[
\ell\e \in \e + \ZM \quad \text{and} \quad \ell\e + \ell \in \e + \ZM
\]
which implies that $\ell \in \ZM$ as claimed.
\end{proof}

We assume that $W$ is infinite from now on, so that $\ell$ is an
integer. In this case the reflections in the hyperplane
arrangements
\[
\reflect \supset \ell \reflect \supset  \ell^2 \reflect \supset \dots
\]
generate reflection subgroups
\[
W \supset W_1
\supset W_2 \supset
\]
which are all abstractly isomorphic to $W$.

Let us assume that $\overline{A_0}$ is compact. In this case we have a sequence
of fundamental alcoves
\[
A_0 \subset A_0^1 \subset A_0^2 \subset
\]
for each of these reflection subgroups. If we delete from $A_0^{i+1}$
all reflecting hyperplanes in $W_i$ (a set of measure zero), then the
resulting set is an $\ell^n$-sheeted covering of $A_0^i$. In particular,
each $A_0^{i+1}$ contains $\ell^n$ copies of $A_0^i$.

\begin{ex}
  Consider the affine Weyl group of a root system of type $A_1$. Here
  is a picture of the reflecting hyperplanes $\reflect$ (in grey), the
  reflecting hyperplanes $3\reflect$ (in light red), and the reflecting
  hyperplanes $9\reflect$ (in dark red):
\begin{equation*}
  \begin{tikzpicture}[scale=0.7]
\draw[fill,blue!10!white] (0,-.1) rectangle (9,.1);
\draw[fill,blue!40!white] (0,-.1) rectangle (3,.1);
\draw[fill,blue] (0,-.1) rectangle (1,.1);
\draw[gray!60!white] (-5.5,0) to (12.5,0);
\foreach \x in {-5,...,12}
{\draw[gray!60!white] (\x,-.2) to (\x,.2);}
\foreach \x in {-3,0,3,6,9,12}
{\draw[red!20!white,thick] (\x,-.2) to (\x,.2);}
\foreach \x in {0,9}
{\draw[red,thick] (\x,-.2) to (\x,.2);}
\node[below] at (.5,0) {\tiny $A_0$};
\end{tikzpicture}
\end{equation*}The fundamental alcove
  $A_0$ is dark blue, $A_0^1$ is light blue, and $A_0^2$ is even
  lighter blue. Note that $A_0^1$ (resp. $A_0^2$) contains $3 = 3^{\textrm{rank}}$
  (resp. $3^2$) copies of $A_0$.
\end{ex}

\begin{ex}
  Consider the affine Weyl group of a root system of type $C_2$. Here
  is a picture the reflecting hyperplanes $\reflect$ (in grey) and the
  reflecting hyperplanes $3\reflect$ (in red):
\begin{equation*}
  \begin{tikzpicture}
\clip (-1.7,-1.7) rectangle (4.7,5.7);
%%%%%
\draw[fill,blue!10!white] (0,0) -- (0,9) -- (4.5,4.5) -- (0,0);
\draw[fill,blue!40!white] (0,0) -- (0,3) -- (1.5,1.5) -- (0,0);
\draw[fill,blue] (0,0) -- (0,1) -- (.5,.5) -- (0,0);
%%%%
\foreach \x in {-6,...,6}
{    \draw[gray!60!white] (\x,-6) to (\x,10); }
\foreach \y in {-6,...,10}
{    \draw[gray!60!white] (-6,\y) to (6,\y); }
\foreach \z in {-5,...,6}
{    \draw[gray!60!white] (\z,-6) to (-6,\z); }
\foreach \z in {-5,...,5}
{    \draw[gray!60!white] (\z,6) to (6,\z); }
\foreach \z in {-6,...,5}
{    \draw[gray!60!white] (-6,\z) to (-\z,6); }
\foreach \z in {-5,...,5}
{    \draw[gray!60!white] (6,\z) to (-\z,-6); }
%%%%%%
\foreach \x in {-6,-3,0,3,6}
{    \draw[red] (\x,-6) to (\x,6); }
\foreach \y in {-6,-3,0,3,6}
{    \draw[red] (-6,\y) to (6,\y); }
\foreach \z in {-6,-3,0,3,6}
{    \draw[red] (\z,-6) to (-6,\z); }
\foreach \z in {-6,-3,0,3,6}
{    \draw[red] (\z,6) to (6,\z); }
\foreach \z in {-6,-3,0,3,6}
{    \draw[red] (-6,\z) to (-\z,6); }
\foreach \z in {-6,-3,0,3,6}
{    \draw[red] (6,\z) to (-\z,-6); }
  \end{tikzpicture}
\end{equation*}The fundamental alcove
  $A_0$ is dark blue, $A_0^1$ is light blue, and $A_0^2$ is even
  lighter blue. Note that $A_0^1$ (resp. $A_0^2$) contains $9 =
  3^{\textrm{rank}}$ (resp. 81) copies of $A_0$.
%\todo{make all of the light blue visible}
\end{ex}

\subsection{Kazhdan-Lusztig theory} \label{sec:kl}
Fix a Coxeter system $(W,S)$ as above and let
\[
\ell : W \to \ZM_{\ge 0}
\]
denote its length function. Let $H = H_W$ denote the Hecke algebra, which is
free over $\ZM[v^{\pm 1}]$ with basis $\{ \delta_x \; | \; x \in W
\}$, unit $1 = \delta_{\id}$ and multiplication determined by
\[
  \delta_s \delta_x =  \begin{cases}
\delta_{sx} & \text{if $\ell(sx) = \ell(x) + 1$,} \\
\delta_{sx} + (v^{-1} - v) \delta_x & \text{if $\ell(sx) = \ell(x) -
  1$.}
\end{cases}
\]
One checks that $\delta_s$ is invertible with inverse $\delta_s^{-1} =
\delta_s + (v-v^{-1})$ and concludes that $\delta_x$ is invertible for
all $x \in W$. The Kazhdan-Lusztig involution $h \mapsto
\overline{h}$ is the unique algebra involution that sends $v$ to 
$v^{-1}$ and $\delta_x$ to $\delta_{x^{-1}}^{-1}$.

The following theorem of Kazhdan and Lusztig \cite{KL} is fundamental to
geometric representation theory. We use the normalisations of
\cite{SoeKL}, where one can also find a very readable proof:

\begin{thm}
  There exists a unique basis $\{ b_x \; | \; x \in W \}$ of $H$
  satisfying the following two conditions:
  \begin{enumerate}
  \item $\overline{b_x} = b_x$;
  \item $b_x = \delta_x + \sum_{y \ne x} h_{y,x} \delta_y$ with $h_{y,x}
    \in v\ZM[v]$.
  \end{enumerate}
\end{thm}

For example, one has
\[
b_s = \delta_s + v \quad \text{for $s \in S$.}
\]
It turns out that $h_{y,x} = 0$ unless $y$ preceeds $x$ in Bruhat
order. The basis $\{ b_x\}$ is the \emph{Kazhdan-Lusztig basis} and
the polynomials $h_{y,x}$ are \emph{Kazhdan-Lusztig 
  polynomials} (up to a normalisation).

When we come to discuss the Hecke category in \S \ref{sec:Hecke} a
certain form
\[
(- , - ) : H \times H \to \ZM[v^{\pm 1}]
\]
will take on great importance. It is defined by
\[
(h, h') := \text{coefficient of $\delta_{\id}$ in $\tau(h)h'$ in the
  standard basis.}
\]
where $\tau$ is the anti-involution of $H$ which fixes $b_s$ and
sends $v$ to $v^{-1}$.
  It is a nice exercise to check that $\tau(\delta_x) = \delta_x^{-1}$
  and that $\{ \delta_x^{-1} \}$ and $\{ \delta_x \}$ are dual bases
  for $(-,-)$.

We will also need the $p$-Kazhdan-Lusztig basis below. Assume that
$(W,S)$ is crystallographic, and that we have fixed an integral root system
giving rise to $W$. The following is explained in \cite{JW-pcan}:

\begin{thm}
  For all primes $p$ there exists a basis $\{ {}^p b_x \; | \; x \in
  W \}$ of $H$, given by the classes of indecomposable objects in the
  Hecke category in characteristic $p$. This basis is self-dual (i.e. $\overline{{}^p b_x} =
  {}^p b_x$); has positive structure constants; and admits
a positive expression
\[
{}^pb_x = \sum {}^p a_{y,x} b_y \quad \text{with ${}^p a_{y,x} \in
  \ZM_{\ge 0}[v^{\pm 1}]$}
  \]
in the Kazhdan-Lusztig basis.
\end{thm}

Note that the crucial condition $h_{y,x} \in v\ZM[v]$ is missing from
the $p$-Kazhdan-Lusztig basis, and in general is not satisfied. If we write
\[
{}^p b_x = \sum {}^p h_{y,x} \delta_y
\]
we refer to ${}^p h_{y,x} $ as \emph{$p$-Kazhdan-Lusztig polynomials.}
(They are Laurent polynomials in general.)

\begin{remark}
  We should emphasise that the definition of the Kazhdan-Lusztig basis
  is combinatorial. This is not so for the $p$-Kazhdan-Lusztig basis,
  which requires the Hecke category for its calculation. It seems
  unlikely that the $p$-Kazhdan-Lusztig basis will admit a uniform
  combinatorial description for all $p$. For much more detail on the
  $p$-Kazhdan-Lusztig basis, see \cite{JW-pcan} and \cite[Chapter 27]{soergelbook}.
\end{remark}

\begin{ex} \label{ex:sp4}
  Consider the root system of type $C_2$
\begin{equation*}
  \begin{tikzpicture}[scale=0.8]
    \draw[->] (0,0) to (2,0); \node[right] at (2,0) {\tiny $\beta$};
    \draw[->] (0,0) to (0,2);
    \draw[->] (0,0) to (-2,0);
    \draw[->] (0,0) to (0,-2);
    \draw[->] (0,0) to (1,1);
    \draw[->] (0,0) to (1,-1);
    \draw[->] (0,0) to (-1,1); \node[above,left] at (-1,1) {\tiny $\alpha$};
    \draw[->] (0,0) to (-1,-1);
  \end{tikzpicture}
\end{equation*}
with simple roots $\alpha = \e_1 - \e_2$ and $\beta = 2\e_2$. Let $W$
denote the corresponding Weyl group, with corresponding simple
reflections $s = s_\alpha$ and $t = s_\beta$.

We can depict the elements of $W$ via their chambers as follows: 
\begin{equation*}
  \begin{tikzpicture}[scale=1]
    \draw[gray!50!white] (0,0) to (2,0);
    \draw[gray!50!white] (0,0) to (0,2);
    \draw[gray!50!white] (0,0) to (-2,0);
    \draw[gray!50!white] (0,0) to (0,-2);
    \draw[gray!50!white] (0,0) to (2,2);
    \draw[gray!50!white] (0,0) to (2,-2);
    \draw[gray!50!white] (0,0) to (-2,2);
    \draw[gray!50!white] (0,0) to (-2,-2);
\node at (67.5:1.6) {\small $\id$};
\node at (112.5:1.6) {\small $t$};
\node at (157.5:1.6) {\small $t s$};
\node at (202.5:1.6) {\small $t s t$};
\node at (247.5:1.6) {\small $s t s t $};
\node at (292.5:1.6) {\small $s t s $};
\node at (337.5:1.6) {\small $s t$};
\node at (382.5:1.6) {\small $s$};
  \end{tikzpicture}
\end{equation*}
%\todo{Perhaps we can shade the identity alcove throughout in these examples.}
We can use such pictures to illustrate Kazhdan-Lusztig
polynomials, by writing the $h_{y,x}$ in the alcove $y$ (we read off
$x$ by locating the unique 1 in the picture). For example
\[
b_s := \delta_s + v \delta_{\id}
\]
is depicted as follows:
\begin{equation*}
  \begin{tikzpicture}[scale=.6]
    \draw[gray!50!white] (0,0) to (2,0);
    \draw[gray!50!white] (0,0) to (0,2);
    \draw[gray!50!white] (0,0) to (-2,0);
    \draw[gray!50!white] (0,0) to (0,-2);
    \draw[gray!50!white] (0,0) to (2,2);
    \draw[gray!50!white] (0,0) to (2,-2);
    \draw[gray!50!white] (0,0) to (-2,2);
    \draw[gray!50!white] (0,0) to (-2,-2);
\node at (67.5:1.6) {\small $v$};
% \node at (112.5:1.6) {\small $t$};
% \node at (157.5:1.6) {\small $t s$};
% \node at (202.5:1.6) {\small $t s t$};
% \node at (247.5:1.6) {\small $s t s t $};
% \node at (292.5:1.6) {\small $s t s $};
% \node at (337.5:1.6) {\small $s t$};
\node at (382.5:1.6) {\small $1$};
  \end{tikzpicture}
\end{equation*}

For later use, we depict 2 elements of the Kazhdan-Lusztig basis:
\begin{equation*}
  \begin{array}{ccc}
b_{st} = \begin{array}{c} 
  \begin{tikzpicture}[scale=1]
    \draw[gray!50!white] (0,0) to (2,0);
    \draw[gray!50!white] (0,0) to (0,2);
    \draw[gray!50!white] (0,0) to (-2,0);
    \draw[gray!50!white] (0,0) to (0,-2);
    \draw[gray!50!white] (0,0) to (2,2);
    \draw[gray!50!white] (0,0) to (2,-2);
    \draw[gray!50!white] (0,0) to (-2,2);
    \draw[gray!50!white] (0,0) to (-2,-2);
\node at (67.5:1.6) {\small $v^2$};
\node at (112.5:1.6) {\small $v$};
% \node at (157.5:1.6) {\small $t s$};
% \node at (202.5:1.6) {\small $t s t$};
% \node at (247.5:1.6) {\small $s t s t $};
% \node at (292.5:1.6) {\small $s t s $};
\node at (337.5:1.6) {\small $1$};
\node at (382.5:1.6) {\small $v$};
  \end{tikzpicture}\end{array} & &
b_{sts} = \begin{array}{c} 
    \begin{tikzpicture}[scale=1]
    \draw[gray!50!white] (0,0) to (2,0);
    \draw[gray!50!white] (0,0) to (0,2);
    \draw[gray!50!white] (0,0) to (-2,0);
    \draw[gray!50!white] (0,0) to (0,-2);
    \draw[gray!50!white] (0,0) to (2,2);
    \draw[gray!50!white] (0,0) to (2,-2);
    \draw[gray!50!white] (0,0) to (-2,2);
    \draw[gray!50!white] (0,0) to (-2,-2);
\node at (67.5:1.6) {\small $v^3$};
\node at (112.5:1.6) {\small $v^2$};
\node at (157.5:1.6) {\small $v$};
% \node at (202.5:1.6) {\small $t s t$};
% \node at (247.5:1.6) {\small $s t s t $};
\node at (292.5:1.6) {\small $1$};
\node at (337.5:1.6) {\small $v$};
\node at (382.5:1.6) {\small $v^2$};
  \end{tikzpicture}  
  \end{array} 
  \end{array}
\end{equation*}
We also draw 2 elements of the 2-canonical basis:
\begin{equation*}
  \begin{array}{ccc}
{}^2b_{st} = \begin{array}{c} 
  \begin{tikzpicture}[scale=1]
    \draw[gray!50!white] (0,0) to (2,0);
    \draw[gray!50!white] (0,0) to (0,2);
    \draw[gray!50!white] (0,0) to (-2,0);
    \draw[gray!50!white] (0,0) to (0,-2);
    \draw[gray!50!white] (0,0) to (2,2);
    \draw[gray!50!white] (0,0) to (2,-2);
    \draw[gray!50!white] (0,0) to (-2,2);
    \draw[gray!50!white] (0,0) to (-2,-2);
\node at (67.5:1.6) {\small $v^2$};
\node at (112.5:1.6) {\small $v$};
% \node at (157.5:1.6) {\small $t s$};
% \node at (202.5:1.6) {\small $t s t$};
% \node at (247.5:1.6) {\small $s t s t $};
% \node at (292.5:1.6) {\small $s t s $};
\node at (337.5:1.6) {\small $1$};
\node at (382.5:1.6) {\small $v$};
  \end{tikzpicture}\end{array} & &
                                   {}^2b_{sts} = \begin{array}{c} 
    \begin{tikzpicture}[scale=1]
    \draw[gray!50!white] (0,0) to (2,0);
    \draw[gray!50!white] (0,0) to (0,2);
    \draw[gray!50!white] (0,0) to (-2,0);
    \draw[gray!50!white] (0,0) to (0,-2);
    \draw[gray!50!white] (0,0) to (2,2);
    \draw[gray!50!white] (0,0) to (2,-2);
    \draw[gray!50!white] (0,0) to (-2,2);
    \draw[gray!50!white] (0,0) to (-2,-2);
\node at (67.5:1.6) {\small $v^3 + v$};
\node at (112.5:1.6) {\small $v$};
\node at (157.5:1.6) {\small $v$};
% \node at (202.5:1.6) {\small $t s t$};
% \node at (247.5:1.6) {\small $s t s t $};
\node at (292.5:1.6) {\small $1$};
\node at (337.5:1.6) {\small $v$};
\node at (382.5:1.6) {\small $v^2 + 1$};
  \end{tikzpicture}  
  \end{array} 
  \end{array}
\end{equation*}
\end{ex}

\subsection{Kazhdan-Lusztig theory and standard parabolic
  subgroups}
In the following sections we explore various positivity properties of
the Kazhdan-Lusztig basis with respect to reflection subgroups.

We start with the easiest case of
a standard parabolic subgroup.
Let $(W,S)$ denote a Coxeter group. Assume our reflection subgroup is of
the form
\[
W_r := \langle s \; | \; s \in I \rangle
\]
for some subset $I \subset S$. Such subgroups are called standard
parabolic subgroups. Recall that ${}^rW$ denotes the set of
distinguished representatives for the right cosets $W_r\setminus
W$. If $H$ denotes the Hecke algebra we can consider the following set
\begin{equation}
  \label{eq:mixed_basis}
\{ b_v\delta_x \; | \; v \in W_r, x \in {}^r W \}.  
\end{equation}
For $v \in W_r$ let us write
\[
b_v = \sum_{u \le v} h_{u,v} \delta_u
\]
Because the lengths add under multiplication $W_r \times {}^rW \to W$
we see that, if $x \in {}^r W$, then
\[
b_v\delta_x = \sum_{u \in W_r \atop u \le v} h_{u,v} \delta_{ux}.
\]
in particular, the set \eqref{eq:mixed_basis} is a basis of $H$, which
we call the \emph{mixed basis}. The following ``mixed positivity'' is
comparatively easy but  useful:

\begin{thm} \label{thm:pos1}
If we write
\[
b_w = \sum_{v \in W_r \atop x \in {}^r W} h^r_{v,x,w} \cdot b_v
\delta_x \qquad\text{with $h^r_{v,x,w} \in \ZM[v^{\pm 1}]$}
\]
then the $h^I_{v,x,w} $ have non-negative coefficients.
\end{thm}

\begin{remark}
  The case $I = \emptyset$ is equivalent to the positivity of Kazhdan-Lusztig
  polynomials.
\end{remark}

\begin{remark} The consideration of the mixed basis is influenced by the work of Geck
  \cite{Geck}. Recently Jensen and Patimo \cite{JP} have been able to
  obtain a similar result for the $p$-Kazhdan-Lusztig basis.
\end{remark}

\begin{ex}
  We illustrate this positivity, continuing Example \ref{ex:sp4}. Let $W_r =
  \langle t \rangle$, so that our minimal coset representatives ${}^rW$ 
  are indexed by those chambers in the right-hand half-space:
\begin{equation*}
  \begin{tikzpicture}[scale=1]
    \draw[gray!50!white] (0,0) to (2,0);
    \draw[gray!50!white] (0,0) to (0,2);
    \draw[gray!50!white] (0,0) to (-2,0);
    \draw[gray!50!white] (0,0) to (0,-2);
    \draw[gray!50!white] (0,0) to (2,2);
    \draw[gray!50!white] (0,0) to (2,-2);
    \draw[gray!50!white] (0,0) to (-2,2);
    \draw[gray!50!white] (0,0) to (-2,-2);
\node at (67.5:1.6) {\small $\id$};
% \node at (112.5:1.6) {\small $t$};
% \node at (157.5:1.6) {\small $t s$};
% \node at (202.5:1.6) {\small $t s t$};
% \node at (247.5:1.6) {\small $s t s t $};
\node at (292.5:1.6) {\small $s t s $};
\node at (337.5:1.6) {\small $s t$};
\node at (382.5:1.6) {\small $s$};
  \end{tikzpicture}
\end{equation*}
Here two interesting cases are given by the elements $b_{st}$ and
$b_{sts}$, which have already been written down in Example
\ref{ex:sp4}. For $b_{st}$ we have
\begin{align*}
\begin{array}{c}
\begin{tikzpicture}[scale=.5]
    \draw[gray!50!white] (0,0) to (2,0);
    \draw[gray!50!white] (0,0) to (0,2);
    \draw[gray!50!white] (0,0) to (-2,0);
    \draw[gray!50!white] (0,0) to (0,-2);
    \draw[gray!50!white] (0,0) to (2,2);
    \draw[gray!50!white] (0,0) to (2,-2);
    \draw[gray!50!white] (0,0) to (-2,2);
    \draw[gray!50!white] (0,0) to (-2,-2);
\node at (67.5:1.6) {\small $v^2$};
\node at (112.5:1.6) {\small $v$};
% \node at (157.5:1.6) {\small $t s$};
% \node at (202.5:1.6) {\small $t s t$};
% \node at (247.5:1.6) {\small $s t s t $};
% \node at (292.5:1.6) {\small $s t s $};
\node at (337.5:1.6) {\small $1$};
\node at (382.5:1.6) {\small $v$};
  \end{tikzpicture}
\end{array}
\quad = \quad
 \begin{array}{c}
\begin{array}{c}
\begin{tikzpicture}[scale=.5]
    \draw[gray!50!white] (0,0) to (2,0);
    \draw[gray!50!white] (0,0) to (0,2);
    \draw[gray!50!white] (0,0) to (-2,0);
    \draw[gray!50!white] (0,0) to (0,-2);
    \draw[gray!50!white] (0,0) to (2,2);
    \draw[gray!50!white] (0,0) to (2,-2);
    \draw[gray!50!white] (0,0) to (-2,2);
    \draw[gray!50!white] (0,0) to (-2,-2);
\node at (67.5:1.6) {\small $v^2$};
\node at (112.5:1.6) {\small $v$};
% \node at (157.5:1.6) {\small $t s$};
% \node at (202.5:1.6) {\small $t s t$};
% \node at (247.5:1.6) {\small $s t s t $};
% \node at (292.5:1.6) {\small $s t s $};
%\node at (337.5:1.6) {\small $1$};
%\node at (382.5:1.6) {\small $v$};
  \end{tikzpicture}
\end{array} \\
+ \\
\begin{array}{c}
\begin{tikzpicture}[scale=.5]
    \draw[gray!50!white] (0,0) to (2,0);
    \draw[gray!50!white] (0,0) to (0,2);
    \draw[gray!50!white] (0,0) to (-2,0);
    \draw[gray!50!white] (0,0) to (0,-2);
    \draw[gray!50!white] (0,0) to (2,2);
    \draw[gray!50!white] (0,0) to (2,-2);
    \draw[gray!50!white] (0,0) to (-2,2);
    \draw[gray!50!white] (0,0) to (-2,-2);
%\node at (67.5:1.6) {\small $v^2$};
%\node at (112.5:1.6) {\small $v$};
% \node at (157.5:1.6) {\small $t s$};
% \node at (202.5:1.6) {\small $t s t$};
% \node at (247.5:1.6) {\small $s t s t $};
% \node at (292.5:1.6) {\small $s t s $};
%\node at (337.5:1.6) {\small $1$};
\node at (382.5:1.6) {\small $v$};
  \end{tikzpicture}
\end{array} \\
+ \\
\begin{array}{c}
\begin{tikzpicture}[scale=.5]
    \draw[gray!50!white] (0,0) to (2,0);
    \draw[gray!50!white] (0,0) to (0,2);
    \draw[gray!50!white] (0,0) to (-2,0);
    \draw[gray!50!white] (0,0) to (0,-2);
    \draw[gray!50!white] (0,0) to (2,2);
    \draw[gray!50!white] (0,0) to (2,-2);
    \draw[gray!50!white] (0,0) to (-2,2);
    \draw[gray!50!white] (0,0) to (-2,-2);
%\node at (67.5:1.6) {\small $v^2$};
%\node at (112.5:1.6) {\small $v$};
% \node at (157.5:1.6) {\small $t s$};
% \node at (202.5:1.6) {\small $t s t$};
% \node at (247.5:1.6) {\small $s t s t $};
% \node at (292.5:1.6) {\small $s t s $};
\node at (337.5:1.6) {\small $1$};
%\node at (382.5:1.6) {\small $v$};
  \end{tikzpicture}
\end{array}
\end{array}
\quad = \quad 
\begin{array}{c}
v \cdot b_t\delta_{\id} \\
\\
\\
+ \\
\\
\\
v \cdot b_{\id}\delta_{s} \\
\\
\\
+ \\
\\
\\
1 \cdot b_{\id}\delta_{st} 
\end{array}
\end{align*}
We leave it to the reader to find a similar (positive) expansion for $b_{sts}$.
\end{ex}

\subsection{Kazhdan-Lusztig theory and parabolic subgroups}
\label{sec:parabolic}

We now continue a generalisation to a broader class of reflection
subgroups, namely parabolic subgroups. A \emph{parabolic subgroup} is
a subgroup of $W$ which is conjugate to a standard parabolic
subgroup. Let us fix a parabolic subgroup $W_r \subset W$, of course
$W_r$ is a reflection subgroup, and hence inherits a natural set of
Coxeter generators $S_r \subset T$, as in Sections \ref{sec:ref1} and
\ref{sec:ref2}. Let ${}^rW$ denote minimal representatives as usual. It is important to keep in mind that in general the
elements of $S_r$ are not simple reflections, and the multiplication
map
\[
W_r \times {}^r W \to W
\]
is not compatible with lengths.

Let us consider the group algebra $\ZM[W]$ as a bimodule over the
Hecke algebra of $W_r$ and $W$ respectively, via specialisation at $v
:= 1$:
\[ \begin{tikzpicture}[scale=0.8]
  \node at (-4,0) {$H_{W_r}$};
\node at (-2.6,0) {$\acts$};
\node at (-2.6,-.3) {\tiny $v := 1$};
\node at (0,0) {$\ZM[W] = $};
\node[rotate=90] at (0,-1) {$=$};
\node at (0,-2) {$\bigoplus_{x \in W} \ZM \hat{\delta}_x$};
\node at (2.6,0) {$\racts$};
\node at (2.6,-.3) {\tiny $v := 1$};
  \node at (4,0) {$H_{W}$};
\end{tikzpicture}
\]
(We denote the standard basis of the group algebra by $\{
\hat{\delta}_x \}$.)
When viewed as a bimodule in this way, we refer to $\ZM[W]$ as the
\emph{hyperbolic bimodule}.\footnote{The main reason
  for the name is that it sounds impressive. A second reason is that it is
  related to hyperbolic localisation, as we will see.} 
In this bimodule we can consider an analogue of the mixed basis:
\[
\{ b_v \hat{\delta}_x \; | \; v \in W_r, x \in {}^r W \}.
\]
Now the multiplication $W_r \times {}^r W \to W$ does not preserve lengths, but we still have
\[
b_v\delta_x = \sum_{u \in W_r \atop u \le v} h_{u,v}(1) \delta_{ux}.
\]
because we have specialised $v := 1$. In particular, we still get a basis.

\begin{thm}[\cite{BilleyBraden}] If we write \label{thm:pos2}
\[
\hat{\delta}_{\id} \cdot b_w = \sum_{v \in W_r \atop x \in {}^r W}  h^r_{v,x,w} b_v \hat{\delta}_x
\qquad\text{with $h^r_{v,x,w} \in \ZM$}
\]
 then the $h^r_{v,x,w}$ are non-negative.
\end{thm}

\begin{remark}
The key difference between this theorem and Theorem \ref{thm:pos1} in
the previous section is that we have to specialise $v := 1$ to get a
positivity statement. Later we will explain (in both geometric and
Soergel bimodule language) that this need to specialise occurs because
we lose a grading.
\end{remark}

\begin{ex}
  We illustrate this positivity in Example \ref{ex:sp4}. Let $W_r =
  \langle u := sts \rangle$. Now $W_r$ is not a standard parabolic
  subgroup. We have $S_r = \{ u \}$ and our minimal coset representatives ${}^rW$ are
  indexed by those chambers in the upper half-space:
\begin{equation*}
  \begin{tikzpicture}[scale=.7]
    \draw[gray!50!white] (0,0) to (2,0);
    \draw[gray!50!white] (0,0) to (0,2);
    \draw[gray!50!white] (0,0) to (-2,0);
    \draw[gray!50!white] (0,0) to (0,-2);
    \draw[gray!50!white] (0,0) to (2,2);
    \draw[gray!50!white] (0,0) to (2,-2);
    \draw[gray!50!white] (0,0) to (-2,2);
    \draw[gray!50!white] (0,0) to (-2,-2);
\node at (67.5:1.6) {\small $\id$};
\node at (112.5:1.6) {\small $t$};
\node at (157.5:1.6) {\small $t s$};
% \node at (202.5:1.6) {\small $t s t$};
% \node at (247.5:1.6) {\small $s t s t $};
% \node at (292.5:1.6) {\small $s t s $};
% \node at (337.5:1.6) {\small $s t$};
\node at (382.5:1.6) {\small $s$};
  \end{tikzpicture}
\end{equation*}
Here an interesting case is given by the element $b_{sts}$ which has
already been considered in Example \ref{ex:sp4}. We have:
\begin{gather*}
  \begin{array}{c} 
    \begin{tikzpicture}[scale=.7]
    \draw[gray!50!white] (0,0) to (2,0);
    \draw[gray!50!white] (0,0) to (0,2);
    \draw[gray!50!white] (0,0) to (-2,0);
    \draw[gray!50!white] (0,0) to (0,-2);
    \draw[gray!50!white] (0,0) to (2,2);
    \draw[gray!50!white] (0,0) to (2,-2);
    \draw[gray!50!white] (0,0) to (-2,2);
    \draw[gray!50!white] (0,0) to (-2,-2);
\node at (67.5:1.6) {\small $v^3$};
\node at (112.5:1.6) {\small $v^2$};
\node at (157.5:1.6) {\small $v$};
% \node at (202.5:1.6) {\small $t s t$};
% \node at (247.5:1.6) {\small $s t s t $};
\node at (292.5:1.6) {\small $1$};
\node at (337.5:1.6) {\small $v$};
\node at (382.5:1.6) {\small $v^2$};
  \end{tikzpicture} 
\end{array}
 \stackrel{ \small v := 1}{\longrightarrow} 
  \begin{array}{c} 
    \begin{tikzpicture}[scale=.5]
    \draw[gray!50!white] (0,0) to (2,0);
    \draw[gray!50!white] (0,0) to (0,2);
    \draw[gray!50!white] (0,0) to (-2,0);
    \draw[gray!50!white] (0,0) to (0,-2);
    \draw[gray!50!white] (0,0) to (2,2);
    \draw[gray!50!white] (0,0) to (2,-2);
    \draw[gray!50!white] (0,0) to (-2,2);
    \draw[gray!50!white] (0,0) to (-2,-2);
\node at (67.5:1.6) {\small $1$};
\node at (112.5:1.6) {\small $1$};
\node at (157.5:1.6) {\small $1$};
% \node at (202.5:1.6) {\small $t s t$};
% \node at (247.5:1.6) {\small $s t s t $};
\node at (292.5:1.6) {\small $1$};
\node at (337.5:1.6) {\small $1$};
\node at (382.5:1.6) {\small $1$};
  \end{tikzpicture} 
\end{array} = \\
 = 
  \begin{array}{c} 
    \begin{tikzpicture}[scale=.5]
    \draw[gray!50!white] (0,0) to (2,0);
    \draw[gray!50!white] (0,0) to (0,2);
    \draw[gray!50!white] (0,0) to (-2,0);
    \draw[gray!50!white] (0,0) to (0,-2);
    \draw[gray!50!white] (0,0) to (2,2);
    \draw[gray!50!white] (0,0) to (2,-2);
    \draw[gray!50!white] (0,0) to (-2,2);
    \draw[gray!50!white] (0,0) to (-2,-2);
%\node at (67.5:1.6) {\small $1$};
%\node at (112.5:1.6) {\small $1$};
\node at (157.5:1.6) {\small $1$};
% \node at (202.5:1.6) {\small $t s t$};
% \node at (247.5:1.6) {\small $s t s t $};
%\node at (292.5:1.6) {\small $1$};
%\node at (337.5:1.6) {\small $1$};
%\node at (382.5:1.6) {\small $1$};
  \end{tikzpicture} 
\end{array} + 
  \begin{array}{c} 
    \begin{tikzpicture}[scale=.5]
    \draw[gray!50!white] (0,0) to (2,0);
    \draw[gray!50!white] (0,0) to (0,2);
    \draw[gray!50!white] (0,0) to (-2,0);
    \draw[gray!50!white] (0,0) to (0,-2);
    \draw[gray!50!white] (0,0) to (2,2);
    \draw[gray!50!white] (0,0) to (2,-2);
    \draw[gray!50!white] (0,0) to (-2,2);
    \draw[gray!50!white] (0,0) to (-2,-2);
%\node at (67.5:1.6) {\small $1$};
\node at (112.5:1.6) {\small $1$};
%\node at (157.5:1.6) {\small $1$};
% \node at (202.5:1.6) {\small $t s t$};
% \node at (247.5:1.6) {\small $s t s t $};
%\node at (292.5:1.6) {\small $1$};
%\node at (337.5:1.6) {\small $1$};
%\node at (382.5:1.6) {\small $1$};
  \end{tikzpicture} 
\end{array}+ 
  \begin{array}{c} 
    \begin{tikzpicture}[scale=.5]
    \draw[gray!50!white] (0,0) to (2,0);
    \draw[gray!50!white] (0,0) to (0,2);
    \draw[gray!50!white] (0,0) to (-2,0);
    \draw[gray!50!white] (0,0) to (0,-2);
    \draw[gray!50!white] (0,0) to (2,2);
    \draw[gray!50!white] (0,0) to (2,-2);
    \draw[gray!50!white] (0,0) to (-2,2);
    \draw[gray!50!white] (0,0) to (-2,-2);
\node at (67.5:1.6) {\small $1$};
%\node at (112.5:1.6) {\small $1$};
%\node at (157.5:1.6) {\small $1$};
% \node at (202.5:1.6) {\small $t s t$};
% \node at (247.5:1.6) {\small $s t s t $};
\node at (292.5:1.6) {\small $1$};
%\node at (337.5:1.6) {\small $1$};
%\node at (382.5:1.6) {\small $1$};
  \end{tikzpicture} 
\end{array}
+ 
  \begin{array}{c} 
    \begin{tikzpicture}[scale=.5]
    \draw[gray!50!white] (0,0) to (2,0);
    \draw[gray!50!white] (0,0) to (0,2);
    \draw[gray!50!white] (0,0) to (-2,0);
    \draw[gray!50!white] (0,0) to (0,-2);
    \draw[gray!50!white] (0,0) to (2,2);
    \draw[gray!50!white] (0,0) to (2,-2);
    \draw[gray!50!white] (0,0) to (-2,2);
    \draw[gray!50!white] (0,0) to (-2,-2);
%\node at (67.5:1.6) {\small $1$};
%\node at (112.5:1.6) {\small $1$};
%\node at (157.5:1.6) {\small $1$};
% \node at (202.5:1.6) {\small $t s t$};
% \node at (247.5:1.6) {\small $s t s t $};
%\node at (292.5:1.6) {\small $1$};
\node at (337.5:1.6) {\small $1$};
\node at (382.5:1.6) {\small $1$};
  \end{tikzpicture} \end{array} 
\\
= \qquad b_{\id} \cdot \hat{\delta}_{ts} \qquad+\qquad 
b_{\id} \cdot \hat{\delta}_{t} \qquad+\qquad
b_{u} \cdot \hat{\delta}_{\id} \qquad+\qquad 
b_{u} \cdot \hat{\delta}_{s}
\end{gather*}
The third term shows that we cannot expect positivity
before specialising $v := 1$.
\end{ex}

\subsection{Kazhdan-Lusztig theory and good reflection subgroups} \label{sec:good}
We hope that \S\ref{sec:self-similar} convinced the reader that
Coxeter groups may have many interesting reflection subgroups which are
not parabolic. A central tenant of these notes is that some sort of positivity holds
for more general reflection subgroups, which we now define.

Let us fix a finite-dimensional reflection representation $V$ of our Coxeter group,
and a subspace $V'$. Consider the reflection subgroup
\[
W_r := \langle s_\gamma \; | \; \gamma \in V' \rangle.
\]
That is, we consider the reflection subgroup generated by all
reflections whose roots lie in $V'$. Alternatively, we have
\[
\gamma \in V' \Leftrightarrow 
\text{ the contragredient action of $s_\gamma$ fixes $(V')^\perp \subset V^*$}.
\]
We will call a reflection subgroup that arises in this manner
$V$-\emph{good} (or simply \emph{good} if the context is clear).
The following principle is at the heart of this paper, it says roughly
that $V$-analogues of Kazhdan-Lusztig polynomials satisfy an analogue of
  Theorem \ref{thm:pos2}, for $V$-good reflection subgroups. Here is
  the meta theorem:

\begin{thm}
  If we write \label{thm:pos3}
\[
\hat{\delta}_{\id} \cdot b^V_w = \sum_{v \in W_r \atop x \in {}^r W}  h^r_{v,x,w} b^V_v \hat{\delta}_x
\qquad\text{with $h^r_{v,x,w} \in \ZM$}
\]
 then the $h^r_{v,x,w}$ are non-negative.
\end{thm}

Roughly, the element $b^V_w$ means some Kazhdan-Lusztig like (or
$p$-Kazhdan-Lusztig like) basis computed via Soergel bimodules or
variants thereof using the reflection representation $V$. We will not
be precise at this stage, and instead give two examples.

\begin{ex} \label{ex:c2}
  Consider the root system of type $C_2$ again, continuing Example
  \ref{ex:sp4}:
\begin{equation*}
  \begin{tikzpicture}[scale=0.8]
    \draw[->] (0,0) to (2,0); \node[right] at (2,0) {\tiny $\beta$};
    \draw[->] (0,0) to (0,2);
    \draw[->] (0,0) to (-2,0);
    \draw[->] (0,0) to (0,-2);
    \draw[->] (0,0) to (1,1);
    \draw[->] (0,0) to (1,-1);
    \draw[->] (0,0) to (-1,1); \node[above,left] at (-1,1) {\tiny $\alpha$};
    \draw[->] (0,0) to (-1,-1);
  \end{tikzpicture}
\end{equation*}
The
simplest example of a reflection subgroup that is not a parabolic
subgroup is the subgroup
\[
W_r := \langle t, u =  sts \rangle = S_2 \times
S_2 \subset W
\]
generated by reflections in the long roots $2\e_1$ and $2\e_2$.
We have $S_r = \{ t, u \}$ and our minimal coset representatives ${}^rW$
  are indexed by those chambers in the upper right quadrant:
\begin{equation*}
  \begin{tikzpicture}[scale=.7]
    \draw[gray!50!white] (0,0) to (2,0);
    \draw[gray!50!white] (0,0) to (0,2);
    \draw[gray!50!white] (0,0) to (-2,0);
    \draw[gray!50!white] (0,0) to (0,-2);
    \draw[gray!50!white] (0,0) to (2,2);
    \draw[gray!50!white] (0,0) to (2,-2);
    \draw[gray!50!white] (0,0) to (-2,2);
    \draw[gray!50!white] (0,0) to (-2,-2);
\node at (67.5:1.6) {\small $\id$};
%\node at (112.5:1.6) {\small $t$};
%\node at (157.5:1.6) {\small $t s$};
% \node at (202.5:1.6) {\small $t s t$};
% \node at (247.5:1.6) {\small $s t s t $};
% \node at (292.5:1.6) {\small $s t s $};
% \node at (337.5:1.6) {\small $s t$};
\node at (382.5:1.6) {\small $s$};
  \end{tikzpicture}
\end{equation*}
Here an interesting case is given by the element $b_{sts}$ which has
already been considered in Example \ref{ex:sp4}. We have:
\begin{align*}
  \begin{array}{c} 
    \begin{tikzpicture}[scale=.7]
    \draw[gray!50!white] (0,0) to (2,0);
    \draw[gray!50!white] (0,0) to (0,2);
    \draw[gray!50!white] (0,0) to (-2,0);
    \draw[gray!50!white] (0,0) to (0,-2);
    \draw[gray!50!white] (0,0) to (2,2);
    \draw[gray!50!white] (0,0) to (2,-2);
    \draw[gray!50!white] (0,0) to (-2,2);
    \draw[gray!50!white] (0,0) to (-2,-2);
\node at (67.5:1.6) {\small $v^3$};
\node at (112.5:1.6) {\small $v^2$};
\node at (157.5:1.6) {\small $v$};
% \node at (202.5:1.6) {\small $t s t$};
% \node at (247.5:1.6) {\small $s t s t $};
\node at (292.5:1.6) {\small $1$};
\node at (337.5:1.6) {\small $v$};
\node at (382.5:1.6) {\small $v^2$};
  \end{tikzpicture} 
\end{array}
& \stackrel{ \small v := 1}{\longrightarrow} 
  \begin{array}{c} 
    \begin{tikzpicture}[scale=.5]
    \draw[gray!50!white] (0,0) to (2,0);
    \draw[gray!50!white] (0,0) to (0,2);
    \draw[gray!50!white] (0,0) to (-2,0);
    \draw[gray!50!white] (0,0) to (0,-2);
    \draw[gray!50!white] (0,0) to (2,2);
    \draw[gray!50!white] (0,0) to (2,-2);
    \draw[gray!50!white] (0,0) to (-2,2);
    \draw[gray!50!white] (0,0) to (-2,-2);
\node at (67.5:1.6) {\small $1$};
\node at (112.5:1.6) {\small $1$};
\node at (157.5:1.6) {\small $1$};
% \node at (202.5:1.6) {\small $t s t$};
% \node at (247.5:1.6) {\small $s t s t $};
\node at (292.5:1.6) {\small $1$};
\node at (337.5:1.6) {\small $1$};
\node at (382.5:1.6) {\small $1$};
  \end{tikzpicture} 
\end{array} = \\
& = 
  \begin{array}{c} 
    \begin{tikzpicture}[scale=.5]
    \draw[gray!50!white] (0,0) to (2,0);
    \draw[gray!50!white] (0,0) to (0,2);
    \draw[gray!50!white] (0,0) to (-2,0);
    \draw[gray!50!white] (0,0) to (0,-2);
    \draw[gray!50!white] (0,0) to (2,2);
    \draw[gray!50!white] (0,0) to (2,-2);
    \draw[gray!50!white] (0,0) to (-2,2);
    \draw[gray!50!white] (0,0) to (-2,-2);
%\node at (67.5:1.6) {\small $1$};
%\node at (112.5:1.6) {\small $1$};
\node at (157.5:1.6) {\small $1$};
% \node at (202.5:1.6) {\small $t s t$};
% \node at (247.5:1.6) {\small $s t s t $};
%\node at (292.5:1.6) {\small $1$};
\node at (337.5:1.6) {\small $1$};
\node at (382.5:1.6) {\small $1$};
  \end{tikzpicture} 
\end{array} + 
  \begin{array}{c} 
    \begin{tikzpicture}[scale=.5]
    \draw[gray!50!white] (0,0) to (2,0);
    \draw[gray!50!white] (0,0) to (0,2);
    \draw[gray!50!white] (0,0) to (-2,0);
    \draw[gray!50!white] (0,0) to (0,-2);
    \draw[gray!50!white] (0,0) to (2,2);
    \draw[gray!50!white] (0,0) to (2,-2);
    \draw[gray!50!white] (0,0) to (-2,2);
    \draw[gray!50!white] (0,0) to (-2,-2);
\node at (67.5:1.6) {\small $1$};
\node at (112.5:1.6) {\small $1$};
%\node at (157.5:1.6) {\small $1$};
% \node at (202.5:1.6) {\small $t s t$};
% \node at (247.5:1.6) {\small $s t s t $};
\node at (292.5:1.6) {\small $1$};
%\node at (337.5:1.6) {\small $1$};
%\node at (382.5:1.6) {\small $1$};
  \end{tikzpicture} 
\end{array}
\end{align*}
which does not admit a positive expansion in the mixed basis. (We do
not expect it to be either, as $W_r$ is not a good reflection subgroup for our
representation over $\QM$.)

However, if we reduce the weight lattice modulo $2$, then both long
roots $2\e_1$ and $2\e_2$ become zero. In particular, $W_r$ is a good
reflection subgroup in characteristic 2! In this case, the Theorem
\ref{thm:pos3} amounts to the following (already rather remarkable!) positivity:
\begin{gather*}
{}^2 b_{sts} = 
  \begin{array}{c} 
    \begin{tikzpicture}[scale=.8]
    \draw[gray!50!white] (0,0) to (2,0);
    \draw[gray!50!white] (0,0) to (0,2);
    \draw[gray!50!white] (0,0) to (-2,0);
    \draw[gray!50!white] (0,0) to (0,-2);
    \draw[gray!50!white] (0,0) to (2,2);
    \draw[gray!50!white] (0,0) to (2,-2);
    \draw[gray!50!white] (0,0) to (-2,2);
    \draw[gray!50!white] (0,0) to (-2,-2);
\node at (67.5:1.6) {\small $v + v^3$};
\node at (112.5:1.6) {\small $v^2$};
\node at (157.5:1.6) {\small $v$};
% \node at (202.5:1.6) {\small $t s t$};
% \node at (247.5:1.6) {\small $s t s t $};
\node at (292.5:1.6) {\small $1$};
\node at (337.5:1.6) {\small $v$};
\node at (382.5:1.6) {\small $1 + v^2$};
  \end{tikzpicture} 
\end{array}
 \stackrel{ \small v := 1}{\longrightarrow} 
  \begin{array}{c} 
    \begin{tikzpicture}[scale=.8]
    \draw[gray!50!white] (0,0) to (2,0);
    \draw[gray!50!white] (0,0) to (0,2);
    \draw[gray!50!white] (0,0) to (-2,0);
    \draw[gray!50!white] (0,0) to (0,-2);
    \draw[gray!50!white] (0,0) to (2,2);
    \draw[gray!50!white] (0,0) to (2,-2);
    \draw[gray!50!white] (0,0) to (-2,2);
    \draw[gray!50!white] (0,0) to (-2,-2);
\node at (67.5:1.6) {\small $2$};
\node at (112.5:1.6) {\small $1$};
\node at (157.5:1.6) {\small $1$};
% \node at (202.5:1.6) {\small $t s t$};
% \node at (247.5:1.6) {\small $s t s t $};
\node at (292.5:1.6) {\small $1$};
\node at (337.5:1.6) {\small $1$};
\node at (382.5:1.6) {\small $2$};
  \end{tikzpicture} 
\end{array} = \\
 = 
  \begin{array}{c} 
    \begin{tikzpicture}[scale=.5]
    \draw[gray!50!white] (0,0) to (2,0);
    \draw[gray!50!white] (0,0) to (0,2);
    \draw[gray!50!white] (0,0) to (-2,0);
    \draw[gray!50!white] (0,0) to (0,-2);
    \draw[gray!50!white] (0,0) to (2,2);
    \draw[gray!50!white] (0,0) to (2,-2);
    \draw[gray!50!white] (0,0) to (-2,2);
    \draw[gray!50!white] (0,0) to (-2,-2);
%\node at (67.5:1.6) {\small $1$};
%\node at (112.5:1.6) {\small $1$};
\node at (157.5:1.6) {\small $1$};
% \node at (202.5:1.6) {\small $t s t$};
% \node at (247.5:1.6) {\small $s t s t $};
%\node at (292.5:1.6) {\small $1$};
\node at (337.5:1.6) {\small $1$};
\node at (382.5:1.6) {\small $2$};
  \end{tikzpicture} 
\end{array} + 
  \begin{array}{c} 
    \begin{tikzpicture}[scale=.5]
    \draw[gray!50!white] (0,0) to (2,0);
    \draw[gray!50!white] (0,0) to (0,2);
    \draw[gray!50!white] (0,0) to (-2,0);
    \draw[gray!50!white] (0,0) to (0,-2);
    \draw[gray!50!white] (0,0) to (2,2);
    \draw[gray!50!white] (0,0) to (2,-2);
    \draw[gray!50!white] (0,0) to (-2,2);
    \draw[gray!50!white] (0,0) to (-2,-2);
\node at (67.5:1.6) {\small $2$};
\node at (112.5:1.6) {\small $1$};
%\node at (157.5:1.6) {\small $1$};
% \node at (202.5:1.6) {\small $t s t$};
% \node at (247.5:1.6) {\small $s t s t $};
\node at (292.5:1.6) {\small $1$};
%\node at (337.5:1.6) {\small $1$};
%\node at (382.5:1.6) {\small $1$};
  \end{tikzpicture} 
\end{array}
= (b_t \hat{\delta}_s + b_u \hat{\delta}_s) + (b_t \hat{\delta}_{\id} + b_u \hat{\delta}_{\id}).
\end{gather*}
\end{ex}

\begin{remark} \label{rem:geomc2}
In the previous example,  $W_r$ arises as Weyl group of the centraliser (isomorphic to $\SL_2
  \times \SL_2$) of the element   $\diag(1,-1,-1,1) \in \Sp_4$. The
  fact that this element is of order 2 explains why this subgroup is
  good in characteristic $2$, and not otherwise.
The two cosets of $W_r$ correspond to two copies of $\PM^1 \times \PM^1$ (the
flag variety of $\SL_2  \times \SL_2$) inside the flag variety of
$\Sp_4$. These two copies of the flag variety are the fixed points
under the element $\diag(1,-1,-1,1) \in \Sp_4$ considered above.
\end{remark}

%\todo{affine Weyl group example}

\section{Localisation: equivariant, hyperbolic, Smith} \label{sec:localisation}

In this section we discuss three types of localisation: equivariant
localisation (for torus actions); hyperbolic
localisation (for $\CM^*$-actions); and ``Smith
localisation'' (for $\ZM/p\ZM$-actions, with characteristic $p$ coefficients).
Later we will apply these three theories to the Hecke category,
providing a bridge between the Hecke category for a Coxeter group; and
those of reflection subgroups.

There are no new results in this section, and the discussion is
deliberately informal. We try to explain why the ideas are natural, and to give references.

\subsection{The principle of localisation}  \label{sec:principal}
Suppose that a group $G$ acts on
a space $X$. Localisation attempts to relate $X$ and
$X^G$. One of its core principles is easily stated: suppose that a
``theory'' satisfies some sort of ``additivity'' and takes ``small'' values
on non-trivial $G$-orbits; then the values of this theory on $X$ and
$X^G$ will differ by ``small'' values.\footnote{All terms in scare
quotes have no precise meaning. The examples below should give some idea
what is meant.} Indeed, by additivity, the difference between the value of the theory on $X$ and
$X^G$ will be its value on $X \setminus X^G$, which consists only of
non-trivial orbits, which make small contributions!

The most fundamental example takes $G = S^1$ and the theory to be the
Euler characteristic $\chi$. Because any non-trivial homogeneous space
is isomorphic to $S^1$ and $\chi(S^1) = 0$ we expect a close relation
between the Euler characteristics of  $X$ and $X^{S^1}$. Indeed,
\[
\chi(X) = \chi(X^{S^1})
\]
for compact $S^1$-spaces $X$ of finite type. (The easiest way to see
this is to use the compactly supported Euler characteristic, which is
additive on open/closed decompositions.)

A second example of a similar nature takes $G = \mu_p = \ZM/p\ZM$ and the
theory to be the Euler characteristic modulo $p$. Because any non-trivial
homogeneous space for $G$ has $p$ points, one may easily conclude that
\[
\chi(X) = \chi(X^{\mu_p}) \mod p
\]
for a finite CW complex $X$ with $G$-action.

These two examples are fundamental but also deceptive. In both
examples the value of the theory on non-trivial orbits was zero (or at
least $0$ modulo $p$). Thus the theory is equal to its value on fixed
points.  In further examples the difference is small,
but non-zero, leading to subtle and important differences between the
theory evaluated on $X$
and $X^G$.

\subsection{Equivariant localisation for $\CM^*$}
It is natural to try to upgrade the above equality of Euler
characteristics to statements in cohomology. The most powerful tool to
do so is equivariant cohomology. Given a variety $X$ with
$\CM^*$-action\footnote{We change setting slightly from $S^1$-actions
  to $\CM^*$-actions to more closely match later considerations.},
recall that its equivariant cohomology is given by
\[
H^*_{\CM^*}(X,k) := H^*(X \times_{\CM^*} E\CM^*,k)
\]
where $E\CM^* \to B\CM^*$ is a universal principal $\CM^*$-bundle (we
can take $\CM^{\infty} \setminus \{ 0 \} \to \PM^\infty \CM$), and $X
\times_{\CM^*} E\CM^*$ denotes the quotient by the diagonal
action. The equivariant cohomology is a module (via the projection $X
\times_{\CM^*} E\CM^* \to B\CM^*$) over
\[
H^*_{\CM^*}(\pt,k) = k[x]
\]
where $x \in H^2(B\CM^*,k)$ is the Chern class of $E\CM^*$.

For simplicity, let us assume $k = \QM$. For any (algebraic) homogenous
$\CM^*$-space $X$ we have
\begin{equation}
  \label{eq:C*calc}
H^*_{\CM^*}(X,\QM) = 
\begin{cases} \QM[x] & \text{if $X = \pt$,}\\
\QM = \QM[x]/(x) & \text{otherwise.}
\end{cases}  
\end{equation}
The principle of localisation (where here ``small'' means ``torsion'')
leads us to guess the following:

\begin{thm}[Localisation theorem, easy version]
  The restriction map
\[
H_{\CM^*}^*(X, \QM) \to H_{\CM^*}^*(X^{\CM^*},\QM)
\]
becomes an isomorphism after inverting $x$.
\end{thm}

The rough idea is that the map $H_{\CM^*}^*(X, \QM) \to
H_{\CM^*}^*(X^{\CM^*},\QM)$ fits into a long exact sequence whose
third term is a torsion $H^*_{\CM^*}(\pt,\QM)$-module, essentially by
\eqref{eq:C*calc}.

% \begin{remark}
%   \todo{equivariant cohomology looks like $K$-theory in high degree.}
% \end{remark}

\subsection{Equivariant localisation for tori} Now suppose that $T
\cong (\CM^*)^m$ is an algebraic torus with lattice of characters
\[
\Chi := \Hom(T, \CM^*).
\]
As for $\CM^*$, the equivariant cohomology of a $T$-variety $X$ is
defined to be
\[
H^*_T(X,\Bbbk) := H^*(X \times_T ET, \Bbbk)
\]
where $ET \to BT$ is a universal principal $T$-bundle. (Once we choose
an isomorphism $T = (\CM^*)^m$ we can take a product of $m$ copies of $\CM^\infty \setminus 0 \to
\PM^\infty\CM$.) Just as before, the projection map $X \times_T ET \to
BT$ means that $H^*_T(X,\Bbbk) $ is always a module over
\[
H^*_T(\pt,\Bbbk) = H^*(BT,\Bbbk).
\]
The Borel isomorphism %\todo{\cite{??}}
gives us a canonical isomorphism
\[
\Chi \simto H^2(BT,\ZM) \quad \text{and} \quad
H^*_T(\pt) = S(\Chi_\Bbbk)
\]
where $\Chi_\Bbbk = \Chi \otimes_\ZM \Bbbk$, and $S$ denotes the symmetric
algebra over $\Bbbk$. Just as before we have

\begin{thm}
  The map
\[
H_T^*(X, \QM) \to H_T^*(X^{\CM^*},\QM)
\]
becomes an isomorphism after inverting all characters of $T$ which are
non-trivial on $\CM^*$.
\end{thm}

\begin{remark}
  We have brushed issues around torsion and the localisation theorem
  under the rug above. These are dealt with systematically in \cite{FW}.
\end{remark}

\subsection{Hyperbolic localisation: smooth case} \label{sec:smooth}
To motivate hyperbolic localisation we first consider a simple
situation, given by smooth variety equipped with a $\CM^*$-action. I
am given to understand that this setting was one of the origins of the
idea (see \cite{Kirwan}).

Suppose that $X$ is a smooth projective variety over $\CM$
with an algebraic $\CM^*$-action. The famous Bia\l ynicki-Birula
decomposition \cite{BB1,BB2} yields that 
\begin{enumerate}
\item Each component of the fixed point locus $X^{\CM^*}$ is smooth.
\item For each component $F \subset X^{\CM^*}$ the \emph{attracting
    set}
\[
X_F^+ := \{ x \in X \; | \; \lim_{z \to 0} z \cdot x \in F \}
\]
is locally closed.
\item The action of $\CM^*$ on $X_F^+$ may be extended to an action of the
  monoid $\CM$ under multiplication. In particular, the \emph{limit map}
\[
p_F : X_F^+ \to F : x \mapsto 0 \cdot x = \lim_{z \to 0} z \cdot x 
\]
is a morphism of algebraic varieties.
\item For any component $F$, the limit map realizes $X_F^+$ as
bundle of affine spaces over $F$.
\item There exists an ordering $F_0, F_1, \dots$ of the components of
  $X^{\CM^*}$  such that the sets
\[
X_i := \bigcup_{i \le j} X_{F_j}^+
\]
yield a filtration of $X$ by closed subvarieties.
\end{enumerate}

Of course, we could replace $\CM^*$ with its inverse action. This
amounts to instead considering the \emph{repelling set}
\[
X_F^- := \{ x \in X \; | \; \lim_{z \to \infty} z \cdot x \in F \}
\]
for which the obvious analogues of (2), (3), (4) and (5) are true.

\begin{remark}
  It is useful to consider the above statements when
\[
  X = \PM(V)
\]
  is
  the projective space of a finite-dimensional algebraic $\CM^*$-module $V$. Write $V =
  \bigoplus_{i \in \ZM} V_i$ for the decomposition of $V$ into weight
  spaces. By twisting by a character of $\CM^*$ we may assume that
  only positive weights occur, i.e. $V =\bigoplus_{i \ge 0} V_i$. We use this to write points of $V$ in ``homogenous
  coordinates'' $[ v_0 : v_1 : \dots ]$, with $v_i
  \in V_i$. Then
\[
  \PM(V)^{\CM^*}= \PM(V_0) \sqcup \PM(V_1) \sqcup \PM(V_2) \sqcup \dots
\]
and our filtration is $X_i = \PM( \bigoplus_{j \ge i} V_j)$. We have
\[
X_{\PM(V_i)}^{+} = \{ [ v_0 : v_1 : \dots ] \; | \;
v_j =0 \text{ for $j<p$ and $v_i \ne 0$} \} 
\]
and the limit map is given by
\[
X_{\PM(V_i)}^{+} \to \PM(V_i) : [v_0 : v_1 : \dots ] \mapsto [v_i].
\]
This calculation shows that the limit map realises
$X_{\PM(V_i)}^{+}$ as a direct sum of line bundles (in fact, copies of
$\OC(1)$) over $\PM(V_i)$.
\end{remark}

\begin{remark}
Suppose that there exists an algebraic
$\CM^*$-module $V$ and an equivariant embedding
\[
X \into \PM(V).
\]
This is often the case in practice, in which case it is easy to see
why (2), (3)
and (5) hold. (In fact this is always true locally, 
by the ``equivariant embedding theorem'' \cite{Sumihiro}.)
\end{remark}

Now let us apply the local-global spectral sequence (see e.g. \cite[Proof of Prop. 3.4.4]{SL}) associated to our
filtration of $X$ to
calculate the cohomology of $X$. It has the form
\begin{gather*}
E_1^{i,j} = H^{i+j}(X_{i}\setminus X_{i+1}, u_i^! k_X) \Rightarrow
H^{i+j}(X,k) \\  
\quad\text{where $u_i : X_{F_i}^+ \into X$ is the inclusion.}
\end{gather*}
Because $u_j$ is the inclusion of a smooth subvariety (of codimension
$c_j$, say) we have $u_j^! k_X = k_X[-2c_j]$ and we can rewrite the terms
of our spectral sequence as
\[
H^{i+j}(X_{F_j}^+, u_i^! k_X) 
= H^{i+j + 2c_j}(X_{F_j}^+, k ) = H^{i+2c_j}(F_j, k) 
\]
where, for the last step we use that $X_{F_j}^+$ and $F_j$ have the
same cohomology, the first being an affine space bundle over the
second.

The upshot is that we have found a spectral sequence calculating the
cohomology of $X$ in terms of the cohomology of the components $F_i$
of the fixed point locus $X^{\CM^*}$. This spectral sequence is 
useful in practice. For example it degenerates in the following
situations:
\begin{enumerate}
\item over $\QM$ for weight reasons (each differential in the spectral
  sequence goes between groups of distinct weights);
\item in general if the components $F_i$ have no odd cohomology (then
  the possibly non-vanishing terms in the spectral sequence resemble
  the black squares on a chess board, and all differentials are zero
  for trivial reasons).
\end{enumerate}

\subsection{Hyperbolic localisation for one component}

Let us assume that $X$ is smooth (and not necessarily projective),
and that $X^{\CM^*}$ consists of one component $F$. Let $X_F^+$ (resp.
$X_F^-$) denote its attractive (resp. repelling) set (defined as in the previous section). We 
consider the diagram
\[
  \begin{tikzpicture}
    \node (F) at (0,0) {$F$};
    \node (F+) at (3,0) {$X_F^+$};
    \node (F-) at (0,-3) {$X_F^-$};
    \node (X) at (3,-3) {$X$};
 \draw[->] (F+) to node[right] {$u_+$} (X);
 \draw[->] (F-) to node[below] {$u_-$} (X);
 \draw[->] (F) to[out=20,in=160] node[above] {$v_+$} (F+);
 \draw[->] (F) to[out=-110,in=110] node[left] {$v_-$} (F-);
 \draw[<-] (F) to[out=-20,in=-160] node[below] {$p_+$} (F+);
 \draw[<-] (F) to[out=-70,in=70] node[right] {$p_-$} (F-);
\draw[->] (F) to node[above,right] {$i$} (X);
  \end{tikzpicture}
\]
where $v_+,v_-,u_+,u_-$ are the inclusions, and $p_+$, $p_-$ are the
limit maps (see the previous section).

The discussion of the previous section showed that for any complex
$\FS$ of sheaves on $X$ there exists a spectral sequence whose terms
involve the cohomology of the complex
\[
(p_+)_*(u_+)^! \FS \in D^b(F).
\] 
If we assume that $\FS$ is $\CM^*$-equivariant then the attractive
proposition (see \cite{SpringerPurity} or \cite[\S~2.6]{FW}, this is essentially a kind of homotopy
invariance, using that $X_F^+$ retracts onto $F$) yields that the
canonical map provides an isomorphism
\[
(p_+)_*(u_+)^! \FS \simto (v_+)^*(u_+)^! \FS \in D^b_{\CM^*}(F).
\]
Thus the terms of our spectral sequence involve a certain variant of
``restriction to fixed points''. The functor
\begin{align*}
  D^b_{\CM^*}(X)  &\to D^b_{\CM^*}(F) \\
\FS & \mapsto (\FS)^{*!} := (v_+)^*(u_+)^! \FS
\end{align*}
is called \emph{hyperbolic localisation}. The most important basic
theorem (and the reason for drawing the big diagram above) is the
following \cite{Braden}:

\begin{thm}[Braden's theorem]
  For any $\FS \in D^b_{\CM^*}(X)$ we have canonical isomorphisms
\[
(v_+)^*(u_+)^! \FS \simto (v_-)^!(u_-)^* \FS
\]
\end{thm}

\begin{remark}
From the theorem we deduce that we have canonical maps
\begin{equation}
  \label{eq:hyplocdiag}
\begin{array}{c}
  \begin{tikzpicture}[scale=1.4]
    \node (tl) at (-3,0) {$i^!\FS$};
    \node (t0) at (0,0) {$(\FS)^{*!}$};
    \node (tr) at (3,0) {$i^*\FS$};
\draw[->] (tl) to (t0); \draw[->] (t0) to (tr);
\node (bl) at (-3,-1) {$ (v_+)^!(u_+)^! \FS$};
\node (b0) at (0,-1) {$(v_+)^*(u_+)^! \FS \simto (v_-)^!(u_-)^* \FS$};
\node (br) at (3,-1) {$ (v_-)^*(u_-)^* \FS$};
\draw[->] (bl) to (b0); \draw[->] (b0) to (br);
\node[rotate=90] at (-3,-.5) {$=$};
\node[rotate=90] at (0,-.5) {$=$};
\node[rotate=90] at (3,-.5) {$=$};
  \end{tikzpicture}  
\end{array}
\end{equation}
Thus, in a certain sense, the hyperbolic localisation functor $(-)^{*!}$ sits ``between $i^!$ and
$i^*$''. This is somewhat analogous to the fact that the intermediate
extension fucnctor $j_{!*}$ sits ``between $j_!$ and $j_*$''. (One
should be somewhat cautious with this analogy though, as we can have
e.g. $i^! \FS = (\FS)^{*!}$ or $i^* \FS = (\FS)^{*!}$, which happens
for purely repulsive or attractive actions.)\footnote{As pointed out
  to me by D. Juteau, one can also has various versions of
  ${}^pj_{!*}$ between $j_!$ and $j_*$ by
  varying the perversity $p$ \dots}
\end{remark}

\begin{remark}
  If we apply Verdier duality to the diagram \eqref{eq:hyplocdiag} we
  get a similar diagram with $\FS$ replaced with $\DM \FS$, and the
  $\CM^*$-action replaced by its inverse action. Thus there are really
  two hyperbolic localisation functors: one for the $\CM^*$-action, and
  one for the inverse action. They are interchanged by Verdier duality.
\end{remark}

\begin{remark} \label{rem:hlic}
  An important consequence of Braden's theorem is that hyperbolic
  localisation preserves weight. (Indeed, if $\FS$ is of weight $w$,
  then $(v_+)^*(u_+)^! \FS = (p_+)_*(u_+)^! \FS$ is of weight $\ge w$,
  as $(p_+)_*$ and $(u_+)^!$ only increase weight. The right hand side
  provides us with the opposite conclusion. Hence $(v_+)^*(u_+)^! \FS$
  is pure of weight $w$.) This implies, for example, that hyperbolic
  localisation preserves semi-simple complexes (with $\QM$
  coefficients). This fact is very useful
  in practice.
\end{remark}

\subsection{Hyperbolic localisation}

We now return to the general setting of \S \ref{sec:smooth}. We
consider the diagram
\[
X^{\CM^*} \stackrel{v_+}{\longto} 
\bigsqcup_{\begin{array}{c}\text{\tiny components $F$} \\ \text{\tiny of
             $X^{\CM^*}$} \end{array}} X_F^+ \stackrel{u_+}{\longto} X
\]
and define \emph{hyperbolic localisation} as before:
\begin{align*}
D^b_{\CM^*}(X) &\to D^b_{\CM^*}(X^{\CM^*})  \\
\FS &\mapsto (\FS)^{*!} := (v_+)^!(u_+)^*\FS
\end{align*}
This may be expressed as a direct sum over the components of
$X^{\CM^*}$ of the functor considered in the previous section, in
particular the discussion of the previous section applies here.

  If $X$ has a $G$-action, extending the action of $\CM^* \subset G$,
  and if $C \subset G$ denotes the centraliser of $\CM^*$, then the
  above diagram is $C$-equivariant and produces a functor
  \begin{align*}
    D^b_G(X) &\to D^b_{C}(X^{\CM^*}) \\
\FS &\mapsto \FS^{*!}
  \end{align*}
between equivariant derived categories.

\subsection{Hyperbolic localisation and cohomology} It is interesting
to ask how hyperbolic localisation interacts with equivariant
cohomology. Assume that $T$ acts on $X$ and we are given a
cocharacter $\chi: \CM^* \to T$ with which we perform hyperbolic
localisation. Let us write $X^{\chi}$ for the fixed points under this
$\CM^*$-action. We obtain maps 
\[
i^! \FS \to (\FS)^{*!} \to i^*\FS \quad \text{in $D^b_T(X^{\chi})$.}
\]
In the following we take $\QM$-coefficients for simplicity:

\begin{thm}
  The induced maps
\[
H_T^*(X^\chi, i^! \FS)
\to 
H_T^*(X^\chi, (\FS)^{*!})
\to 
H_T^*(X^\chi, i^* \FS)
\]
become isomorphisms when we invert all characters $\alpha$ of $T$
which are non-trivial on $\CM^*$ (i.e. whose composition with $\chi$
is not the identity).
\end{thm}

\subsection{Smith theory}

Let $\Bbbk$ denote either a field of characteristic $p$, the ring
$\ZM_p$, or a finite extension thereof.  The discussion of $\mu_p = \ZM/p\ZM$ earlier (in \S \ref{sec:principal}) leads us to hope
for an analogue of the hyperbolic localisation functor for actions of
$\mu_p$ on sheaves with coefficients in $\Bbbk$. We would like some functor
\[
\Sm : D^b_{\mu_p}(X, \Bbbk) \to D^b_{\mu_p}(X^{\mu_p}, \Bbbk) 
\]
($\Sm$ for ``Smith'') lying ``between'' $i^!$ and $i^*$. That is, for every
$\FS \in D^b_{\mu_p}(X, \Bbbk)$ we would like a functorial diagram
\begin{equation}
  \label{eq:5}
i^! \FS \to \Sm(\FS) \to i^*(\FS)
\end{equation}
in $D^b_{\mu_p}(X^{\mu_p}, \Bbbk)$.

As far as I know, there is no such functor. However, it is a fundamental
observation of Treumann \cite{Treumann} that the cone $C$ in the triangle
\begin{equation}
  \label{eq:SmTr}
i^! \FS \to i^* \FS \to C
\end{equation}
is a perfect complex of $\mu_p$-modules when $\Bbbk$ is a field, and
is \emph{weakly injective}\footnote{A $\Bbbk[\mu_p]$-module $M$ is
  weakly injective if it is isomorphic to a module of the form
  $M \otimes \Bbbk[\mu_p]$ for some $\Bbbk$-module $M$. A complex is
  weakly injective if it isomorphic to a bounded complex of weakly
  injective modules.} in general. (For the constant sheaf,
this is a consequence of the fact that the cone calculates the
cohomology of a space with free $\mu_p$-action, which may be
represented by a complex of free modules.)

Thus it makes sense to consider the Verdier quotient
\[
\Perf(X^{\mu_p}, \TC_{\Bbbk}) := D^b_{\mu_p}(X^{\mu_p}, \Bbbk) / (
\text{weakly injective complexes} ).
\]
In more detail, as the $\mu_p$-action is trivial, we may identify
\[
D^b_{\mu_p}(X^{\mu_p}, \Bbbk) = D^b(X^{\mu_p}, \Bbbk[\mu_p])
\]
where the right hand side denotes sheaves of $\Bbbk[\mu_p]$-modules.
Then we quotient out by the full triangulated subcategory generated by
complexes with weakly injective stalks.

We now define $\Sm$ to be the composition
\[
\Sm : D^b_{\mu_p}(X, \Bbbk) \stackrel{i^!}{\to}
D^b_{\mu_p}(X^{\mu_p}, \Bbbk)  \to \Perf(X^{\mu_p}, \TC_\Bbbk).
\]
By Treumann's observation, this agrees with a similar definition with
$i^*$ in place of $i^!$.\footnote{To quote Treumann \cite[Remark 4.4]{Treumann}: ``Smith theory is analogous to hyperbolic
  localisation in the following sense: instead of combining two
  restriction functors in a clever way, we simply erase the
  distinction between them''.} (It is here that the choice of coefficients
is crucial.)

\begin{remark}
The notation $\Perf(X^{\mu_p}, \TC_\Bbbk)$ can be taken simply as
notation for this quotient. In \cite{Treumann} it is used to remind us
that we may view objects in this category as sheaves over a certain
ring spectrum. As this point of view is new to the author, it will not
be used below!
\end{remark}

\begin{remark} \label{rem:smith shift}
The reader is warned that $\Perf(X^{\mu_p}, \TC_\Bbbk)$ is not a derived
category, and takes some getting used to. For example, the ``shift by
2 functor'' $[2]$ is
isomorphic to the identity functor on $\Perf(X^{\mu_p}, \TC_\Bbbk)$. Indeed,
the four-term exact sequence of $k[\mu_p]$-modules
\[
\Bbbk \to \Bbbk[\mu_p] \stackrel{1-\z}{\longto} \Bbbk[\mu_p] \to \Bbbk
\]
where $\zeta$ is a generator of $\mu_p$, 
gives rise to a map $\Bbbk \to \Bbbk[2]$ in $D^b_{\mu_p}(\pt, \Bbbk)$ which
becomes an isomorphism in $\Perf(\pt, \TC_\Bbbk)$. Pulling this map back to
$X^{\mu_p}$ shows that the same is true in $\Perf(X^{\mu_p}, \TC_\Bbbk)$.
\end{remark}

\begin{remark} \label{rem:smith miracle}
  An important principle explained in \cite{Treumann} is that ``$\Sm$
  commutes with all functors'' (see \cite[Theorem
  1.3(2)]{Treumann}). This implies, for example, Smith's observation \cite{Smith}
  that fixed points under $\mu_p$-actions on $p$-smooth
  spaces are $p$-smooth, as the restriction to $\mu_p$ fixed points commutes
  with Verdier duality.
It also implies the following, which will be  useful below. If
$j : Z \into X$ is the inclusion of a locally-closed subset then, for $? \in \{ ! , * \}$,
\begin{equation}
  \label{eq:smith standard}
\Sm(j_? \Bbbk_Z) \cong j_?(\Sm(\Bbbk_Z)) \cong j_? \Bbbk_{Z^\zeta}
\qquad \text{in $\Perf(X^{\mu_p}, \TC_\Bbbk)$}.
\end{equation}
\end{remark}

\subsection{Smith theory and equivariance}

We now consider a slight variant of this construction, which will
prove useful when we come to consider the Hecke category. Let
$\Bbbk$ be as in the previous section ($\FM_p, \ZM_p$ or a
finite extension thereof).

% \begin{remark}
%   The constructions of this section are inspired by the work of
%   Leslie-Lonergan \cite{LL}.
% \end{remark}

Suppose that $X$ is now a $T$-variety, for an algebraic torus $T \cong
(\CM^*)^m$. Fix an element $\z \in T$ of order $p$, and let
$\mu_p \subset T$ denote the subgroup it generates. Given an
equivariant sheaf $\FS \in D^b_T(X,\Bbbk)$, Treumann's observation
implies that the cone $\GS$ in the exact triangle
\[
i^! \FS \to i^* \FS \to \GS \triright
\]
is weakly injective (in the same sense as the previous section), when regarded as
an object in $D^b_{\mu_p}(X^\zeta, \Bbbk)$.
(More precisely, all
functors above commute with the restriction to $\mu_p \subset T$,
and so the discussion of the previous section applies.) 

In particular, if we
consider the full subcategory
\[
( \text{$\z$-perfect complexes} ) := \left \langle \FS \in
  D^b_T(X^\z,\Bbbk) \; \middle | \; \begin{array}{c} \text{the restriction of
                                   $\FS$}\\
\text{to $\mu_p \subset T$ in $D^b_{\mu_p}(X^\z,\Bbbk)$}\\
\text{is weakly-injective} \end{array} \right \rangle
\]
and the Verdier quotient
\[
\Perf_T(X^{\z}, \TC_{\Bbbk}) := D^b_{T}(X^{\z}, \Bbbk) / ( \text{$\zeta$-perfect complexes} )
\]
then we can define the \emph{equivariant Smith localisation} functor
$\Sm$ as the composition
\[
\Sm : D^b_{T}(X, \Bbbk) \stackrel{i^!}{\to}
D^b_{T}(X^{\z}, \Bbbk)  \to \Perf_T(X^{\z}, \TC_{\Bbbk})
\]
where $i : X^\zeta \into X$ is the inclusion. For the same reasons as earlier, $\Sm$ is canonically isomorphic to a
similar functor defined with $*$-restriction in place of
$!$-restriction. (It is here that our choice of coefficients
is crucial.)

More generally, the above constructions work if we replace $T$ by a
general algebraic group $G$.  If $X$ has a $G$-action, extending the action of $T \subset G$,
  and if $C \subset G$ denotes the centraliser of $\zeta$, then we may define
$\Perf_C(X^\zeta, \TC_{\Bbbk})$ by replacing $T$ by $G$ and $C$ above
as appropriate. The
  above diagram is $C$-equivariant and produces a functor
  \begin{align*}
    D^b_G(X) &\to \Perf_C(X^\zeta, \TC_{\Bbbk})\\
\FS &\mapsto \Sm(\FS)
  \end{align*}
between equivariant derived categories.

\subsection{Smith theory and cohomology}

We now see that a natural functor exists on the category $\Perf_T(X^{\z},
\TC_{\Bbbk})$. Consider the set
\[
\Chi_r := \{ \gamma \in \Chi \; | \; \gamma(\zeta) \ne 1 \}.
\]
Note that no element of $\Chi_r$ is divisible by $p$, and hence no
element of $\Chi_r$ is zero in $\Chi_{\Bbbk}$. Define
\[
R^{(r)} := R[\Chi_r^{-1}].
\]
Somewhat miraculously we have:

\begin{prop} \label{prop:smith co}
  Given any $\zeta$-perfect complex $\FS$, its
  hypercohomology $H_T(X^\zeta, \FS)$ is torsion over any $\gamma \in
  \Chi$ such that $\gamma(\zeta) \ne 1$. In particular, the functor
\[
\FS \mapsto R^{(r)} \otimes_R H^{\bullet}_T(\FS)
\]
is zero on $\zeta$-perfect complexes and we obtain an additive functor
on $\Perf_T(X^{\z}, \TC_{\Bbbk})$.
\end{prop}

% \todo{Sometime late in the piece I should rework this proof, and
%   convince myself that it really does work for general $T$.}

\begin{proof}
  We give an argument for $T = \CM^*$ and leave the reduction to this
  case to the reader. Consider first the map
\[
\pi : X \times E\CM^* \to X \times_{\CM^*} E\CM^*.
\]
It is a principal $\CM^*$-bundle and the extension
\[
c \in \Ext^2( \Bbbk_{X \times_{\CM^*} E\CM^*}, \Bbbk_{X \times_{\CM^*}
  E\CM^*}) = H^2_{\CM^*}(X,\Bbbk)
\]
corresponding to the truncation triangle
\[
\Bbbk_{X \times_{\CM^*} E\CM^*} \to \pi_* \Bbbk_{X \times E\CM^*} \to
\Bbbk_{X \times_{\CM^*} E\CM^*}[-1] \triright
\]
agrees with the image of the canonical generator of
$H^2_{\CM^*}(\pt,\Bbbk)$ in $H^2_{\CM^*}(X, \Bbbk)$. (This class is
stable under pull-back, and so we can take $X = \pt$.)

Now if we instead consider
\[
\pi' : X^\zeta \times (E\CM^*/\mu_p) = X \times_{\mu_p} E\CM^* \to X \times_{\CM^*} E\CM^*.
\]
Then this is again a principal $\CM^*$-bundle and the 
class
\[
c' \in \Ext^2( \Bbbk_{X \times_{\CM^*} E\CM^*}, \Bbbk_{X \times_{\CM^*}
  E\CM^*}) = H^2_{\CM^*}(X,\Bbbk)
\]
corresponding to the extension
\begin{equation}
  \label{eq:pext}
\Bbbk_{X \times_{\CM^*} E\CM^*} \to (\pi')_* \Bbbk_{X \times_{\mu_p} E\CM^*} \to
\Bbbk_{X \times_{\CM^*} E\CM^*}[-1] \triright  
\end{equation}
is $p$ times the class $c$ above. (This is more or less equivalent to
the fact that the Chern class of $E\CM^*/\mu_p \to B\CM^*$ is $p$
times the Chern class of $E\CM^* \to B\CM^*$.) The important point
below is that this
class is not invertible, thanks to our choice of coefficients.

Now let $\FS \in D^b_{\CM^*}(X, \Bbbk)$ be such that its restriction
$\FS' \in D^b_{\mu_p}(X,\Bbbk)$ is free. In other words, $\FS$ is a
complex on $X \times_{\CM^*} E\CM^*$ such that its pull-back via
$\pi'$ (as above) agrees with the pushfoward of a complex $\GS$ on $X
\times E\CM^*$ via 
\[
g : X \times E\CM^*  \to X \times_{\mu_p} E\CM^*.
\]
In particular, $(\pi')^* \FS = g_* \GS$ has finite-dimensional
total cohomology. (By definition of the constructible equivariant
derived category, the pushforward of $\GS$ to $X$ under the projection is a
constructible sheaf.)

By the projection formula
\[
(\pi')_*(\pi')^* \FS = \FS \otimes_{\Bbbk} ^L (\pi')_*(\pi')^*
\Bbbk_{X \times_{\CM^*} E\CM^*}
\]
and, by tensoring the truncation triangle for $ (\pi')_*(\pi')^*
\Bbbk_{X \times_{\CM^*} E\CM^*}$ with $\FS$ we get a distinguished
triangle
\[
\FS \to (\pi')_*(\pi')^* \FS \to \FS[-1] \triright
\]
where the connecting homomorphism (the tensor product of $\FS$ with
the class $c'$ above) is not invertible. Taking the long exact
sequence of cohomology gives us a long exact sequence
\[
\dots \to_{p} H^\bullet_{\CM*}(X, \FS) \to H^\bullet_{\mu_p}(X, \FS) \to
H^{\bullet-1}_{\CM*}(X, \FS) \stackrel{+1}{\to}_{p} \dots
\]
where the arrows labelled $p$ are not invertible. Because
$H^\bullet_{\mu_p}(X, \FS)$ is finite dimensional, we deduce that
$H^\bullet_{\CM*}(X, \FS)$ is only non-zero in finitely many
degrees. The proposition follows.
\end{proof}

\subsection{Parity sheaves} \label{sec:parity}

The theory of parity sheaves has two main starting points:
\begin{enumerate}
\item Parity arguments can be useful to show that connecting
homomorphisms are zero, and that spectral sequences degenerate.
\item The cohomology of spaces occurring
in geometric representation theory (flag varieties, Springer
fibres, fibres of Bott-Samelson resolutions, classifying
spaces \dots) often satisfy parity vanishing.
\end{enumerate}

We now briefly recall the theory; much more detail may be found in
\cite{JMW2} (in article form) and \cite{WICM} (in survey form).
Let us fix a variety $X$ with a nice\footnote{For example, Whitney
  would do. We certainly need to assume that $D^b_\Lambda(X,\Bbbk)$
  is preserved under Verdier duality, and this assumption is enough for
  everything we do below.}
stratification 
\[
X = \bigsqcup_{\lambda \in \Lambda} X_\lambda
\]
by locally closed subvarieties. We assume that:
\begin{gather}
  \text{each stratum is simply-connected;} \label{eq:sc assump}
\\
H^\odd(X_\lambda, \Bbbk) = 0 \quad \text{for all $\lambda \in
  \Lambda$.} \label{eq:parity assump}
\end{gather}
We write $D^b_\Lambda(X,\Bbbk)$ for the derived category of complexes
of $\Bbbk$-sheaves, which are bounded and constructible with respect
to the above stratification.

Let $j_\lambda : X_\lambda \into X$ denote the
inclusion. Fix a
complex $\FS \in D^b_\Lambda(X,k)$ and $? \in \{ *, !\}$. We say that
$\FS$ is
\begin{enumerate}
\item \emph{$?$-even}
if $\HC^{\odd}(j_\lambda^? \FS) = 0$ for all
$\lambda$;
\item \emph{$?$-odd} if $\FS[1]$ is $?$-even; 
\item \emph{even} if it is both $*$-even and $!$-even;
\item \emph{odd} if $\FS[1]$ is even.
\end{enumerate}
Finally, we say that $\FS$ is \emph{parity} if it may be written as a sum
$\FS \cong \FS_0 \oplus \FS_1$ with $\FS_0$ even and $\FS_1$ odd. We
write
\[
\Parity(X, \Bbbk) \subset D^b_\Lambda(X, \Bbbk)
\]
for the full subcategory of parity sheaves. Note that $\Parity(X, \Bbbk)$ is additive
but not triangulated.
The following theorem (whose proof is not difficult) is the starting
point of the theory:

\begin{thm}[\cite{JMW2}]
The support of any indecomposable parity complex is
irreducible. Moreover, any two indecomposable complexes $\FS, \GS \in \Parity(X,\Bbbk)$
  with the same support are isomorphic up to a shift. Thus, for any
  stratum $X_\lambda \subset X$ there is, up to isomorphism, at most one indecomposable
  $\FS \in \Parity(X,\Bbbk)$ which is supported on $\overline{X_\lambda}$
  and whose restriction to $X_\lambda$ is isomorphic to $k_{X_\lambda}[\dim_{\CM}X_\lambda]$.
\end{thm}

We denote the sheaf appearing in the theorem $\ES_\lambda$ (if it
exists). One of the main points of \cite{JMW2} is that in many
settings in geometric representation theory $\ES_\lambda$ exists for
all strata, and hence one has a bijection:
 \begin{align*}
\Lambda & \simto 
\begin{array}{c}
\left \{ \begin{array}{c}
\text{indecomposable objects} \\
\text{in $\Parity(X,\Bbbk)$} \end{array} \right \}_{ / \text{shifts} }
\end{array} \\
\lambda & \mapsto \begin{array}{c} \ES_\lambda. \end{array}
 \end{align*}

 \begin{remark}
   One also has an equivariant version of the theory, as long as one
   imposes equivariant versions of  the parity assumptions
   \eqref{eq:sc assump} and  \eqref{eq:parity assump}. Parity sheaves
   are well-adapted to equivariant cohomology, for example their
   global cohomology is often equivariantly formal \cite{FW}.
 \end{remark}

 \begin{remark} \label{rem:hlparity}
   We have seen in Remark \ref{rem:hlic} that hyperbolic localisation
   preserves semi-simple complexes (with $\QM$-coefficients). In
   \cite{JMWtilt} it is shown that (under the assumption of the
   existence of certain equivariant resolutions), hyperbolic
   localisation also preserves parity sheaves.
 \end{remark}

\subsection{Parity sheaves and Smith theory}
The ideas of this section (and even its title) are taken from
\cite{LL}. Their crucial observation is that the theory of parity
shaves is well-adapted to Smith theory.

The starting point for this discussion is our assumption \eqref{eq:parity assump}
which was crucial to get the theory of parity sheaves started. We can
rephrase it as
\begin{equation}
  \label{eq:parity assump rephrase}
\Hom_{D^b(X,k)}(k_{X_\lambda}, k_{X_\lambda}[m]) = 0 
\quad \text{for odd $m$.}
\end{equation}
At the centre of Smith theory, is the group $\mu_p = \ZM/p\ZM$. This
group appears at first sight to be poorly adapted to the theory of
parity sheaves because
\[
H^m_{\mu_p}(\pt, \Bbbk) = \Bbbk \quad \text{for all $m \ge 0$,}
\]
for any field $\Bbbk$ of characteristic $p$. With a little work (see
\cite[\S 2]{LL}) this
calculation implies that
\[
\Hom_{\Perf(\pt, \TC_{\Bbbk})} (\Bbbk_\pt, \Bbbk_\pt[m]) = \Bbbk
\]
for all $m \in \ZM$. Thus parity vanishing fails rather dramatically
in $\Perf(\pt, \TC_{\Bbbk})$, even on a point!

If one instead
instead takes $\Bbbk := \OM = \ZM_p$ one has
\begin{equation}
  \label{eq:mupco}
  H^m_{\mu_p}(\pt, \OM) = \begin{cases} \OM & \text{if $m = 0$,}\\
0 & \text{if $m$ is odd or $m < 0$,} \\
\FM_p & \text{if $m$ is even.} \end{cases}
\end{equation}
This calculation shows that some parity vanishing is still
present. Again, with a little work one may conclude
\begin{align*}
  \Hom_{\Perf(\pt,T_\OM)}(\OM, \OM[m]) 
&= \begin{cases} \FM_p & \text{if $m$ is even,} \\
0 & \text{if $m$ is odd.} \end{cases} 
\end{align*}
Thus we are in good shape on a point. Moreover, one can extend this
argument\footnote{using the map $U \times B\mu_p \to B\mu_p$} to show that if $U$ is any variety for which $H^\odd(U,\OM) =
0$ and $H^\even(X,\OM)$ is free over $\OM$, then
\begin{align*}
  \Hom_{\Perf(U,T_\OM)}(\OM_U, \OM_U[m]) 
&= \begin{cases} H^\even(U,\FM_p) & \text{if $m$ is even,} \\
0 & \text{if $m$ is odd} \end{cases} 
\end{align*}
(where, as always, $U$ has trivial $\mu_p$-action). Thus, if we think
of $U$ as being a stratum in our stratification of the previous
section, we are in good shape on each stratum.

Leslie-Lonergan \cite[\S 3]{LL}
then go on to define Tate functors
\[
T^0, T^1 : \Perf(U,T_\OM) \to Sh(U,\FM_p)
\]
where $Sh(U,\FM_p)$ denotes the abelian category of $\FM_p$-sheaves. 
(These should be thought of as analogues of cohomology functors, which
produce sheaves on $U$ out of complexes on $U$.) They then define
\emph{Tate parity complexes} by repeating the definition of parity
sheaves, with $T^0$ and $T^1$ in place of $\HC^\even$ and
$\HC^\odd$. Furthermore, they show that the Smith functor $\Sm$
takes parity sheaves with $\OM$ coefficients to Tate parity complexes.

We will consider a slight variant of the setup considered by
Leslie-Lonergan. As in the previous section, suppose that all spaces
are equipped with a $\CM^*$-action, which is trivial on $\mu_p \subset
\CM^*$. As before, consider
\[
\Perf_{\CM^*}(U, \TC_{\Bbbk}) := D^b_{T}(U, \Bbbk) / (
\text{$\zeta$-perfect complexes} ).
\]
The following gives a rather beautiful description of this category:
\begin{lem}
\[
\Perf_{\CM^*}(U, \TC_{\Bbbk}) \simto 
\left \{
\begin{array}{c}
\text{complexes of $\Bbbk$-sheaves on $U$ with} \\
\text{constructible and $2$-periodic cohomology}
\end{array} \right \} \subset D(U,\Bbbk).
\]
\end{lem}

\begin{remark}
In the above, $2$-periodic means
  ``equipped with an isomorphism $\FS \to \FS[2]$''. Thus it is
  structure, rather than a property.
\end{remark}

\begin{proof}[Sketch of proof]
We outline why this is true for $U = \pt$, the general case follows by
a similar argument. Given $\FS \in D^b_{\CM*}(\pt, \Bbbk)$ we can
consider its cohomology, which is a complex of $\Bbbk$-modules. These
cohomology sheaves are periodic in high degree, with an isomorphism
given by multiplication by a fixed generator of
$H^2_{\CM^*}(\pt,\Bbbk)$. (Indeed, this is true for $\Bbbk_\pt \in
D^b_{\CM*}(\pt, \Bbbk)$, and hence is also true for the full
subcategory which it generates, which is $D^b_{\CM*}(\pt, \Bbbk)$.) By
considering cohomology in high degree we obtain a triangulated functor
\[
 D^b_{\CM^*}(\pt, \Bbbk) \to 
\left \{
\begin{array}{c}
\text{complexes of $\Bbbk$-sheaves on $pt$ with} \\
\text{constructible and $2$-periodic cohomology}
\end{array} \right \} \subset D(\pt,\Bbbk).
\]
The kernel of this functor consists of those complexes $\FS$ whose
total cohomology is of finite type. Any $\zeta$-perfect complex
belongs to this kernel, by Proposition \ref{prop:smith co}. On
the other hand, the complex computing the $\CM^*$-equivariant
cohomology of $\CM^*$ is $\zeta$-perfect, and generates the kernel. We
conclude that the kernel consists precisely of $\zeta$-perfect
complexes, which concludes the proof.
\end{proof}

Via the above lemma, we obtain Tate cohomology functors (``cohomology
in even degree'', ``cohomology in odd degree'')
\[
T^0, T^1 : \Perf_{\CM^*}(U,T_\Bbbk) \to Sh(U,\Bbbk).
\]
This allows us to define parity sheaves in
$\Perf_{\CM^*}(U,T_\Bbbk)$. It is not difficult to check that the
Smith functors preserve equivariant parity sheaves.

\begin{remark}
  The work of Leslie-Lonergan \cite{LL} builds on the expectation
 that the theory of perverse sheaves is ``compatible with Smith
 theory''. This expectation was first suggested by 
Treumann; see the final two sections of \cite{Treumann}. However it
is still not clear how to proceed with this expectation. For example,
as $[2] = [0]$ in $\Perf(X^\zeta, \TC_{\Bbbk})$ one cannot expect
a reasonable theory of $t$-structures.
\end{remark}

\section{The many faces of the Hecke category} \label{sec:Hecke}

The Hecke category is a categorification of the Hecke
algebra. That is, it is a monoidal category $\HC_W$ with a canonical
isomorphism
\[
H_W \simto [\HC_W]
\]
of its Grothendieck group with the Hecke algebra. Here $[\HC]$ denotes
a suitable Grothendieck group, however exactly what is meant by ``Grothendieck
group'' varies based on context. (Below it usually means the split Grothendieck
group of an additive category, however it might also mean the
Grothendieck group of a triangulated category, or variants
thereof.)

Let us emphasise one important point from the outset, which is a
general principle of categorification. Whenever we categorify, it is
worthwhile asking what extra structures we inherit on the Grothendieck
group. A well-known example is that categorification by additive or
abelian categories often defines some positive cone in the
Grothendieck group, consisting of the classes of actual objects,
rather than formal differences. This cone often leads to strong
positivity on the Grothendieck group (e.g. bases with positive
structure constants; bases which are positive in terms of other bases; \dots). In the
example of the Hecke category this cone gives rise to the
Kazhdan-Lusztig basis (when our coefficients are of characteristic 0)
and the $p$-Kazhdan-Lusztig basis (when our coefficients are of
characteristic $p$).

Another often overlooked structure is a form, obtained by decategorifying the
hom pairing.\footnote{I learnt this point of view from M.~Khovanov and B.~Elias.}
 That is we define
\[
\langle [\EC], [\EC'] \rangle := \chi (\Hom(\EC,\EC'))
\]
where $\chi$ denotes some numerical invariant (e.g. Euler characteristic,
rank, graded rank, \dots) which turns the space of homomorphisms into
a number or polynomial. In the case of the Hecke category, this should
yield the form
\[
(- , - ) : H_W \times H_W \to \ZM[v^{\pm 1}]
\]
defined in \S \ref{sec:kl}. Usually, the most important thing to know
about a potential Hecke category is that this is the case. This
statement is usually referred to as ``Soergel's hom formula'', as
Soergel established it in the first non-trivial case of Soergel
bimodules.

We now return to generalities on the Hecke category. In a sense, there should only be ``one'' Hecke category. However
experience has shown that it has several different realisations, and
it is useful to know about all of them!\footnote{This is somewhat
analogous to the theory of motives, where a motive has many different
realisations, all of which are useful. In fact, recent work
  of Soergel and Wendt \cite{SW}, Soergel, Wendt and Virk \cite{SVW}
  and Eberhardt and Kelly \cite{EK} shows that this 
  is probably more than an analogy.}

Historically the Hecke category arose from two different points of view.
The first is via the so-called function-sheaf correspondence. The
Hecke algebra\footnote{of a Weyl group, specialised at $v := 1/ \sqrt{|\FM_q|}$} may be
realised as a convolution algebra of bi-invariant functions on a
finite reductive group. As with any interesting space of functions
realised on the points of a variety over a finite field, Grothendieck tells us that we should seek a
categorification via a suitable category of
sheaves. Doing so exposes extremely rich structure (e.g. the
Kazhdan-Lusztig basis) that is invisible (or at least opaque)
prior to categorification. This leads us to the geometric realisation. (For much more
detail on this genesis of the Hecke algebra, see \cite{Springer},
\cite{WICM} \dots)

The second is the observation that many categories arising in Lie
theory admit endofunctors which, upon passage to Grothendieck groups,
yield actions of Weyl groups or their group algebras. (The Ur-example is the action of
wall-crossing functors on categories of infinite-dimensional
representations of semi-simple Lie algebras.) A detailed study of
these functors convinces one that it is not enough to understand them
on the Grothendieck group alone; one must understand these functors as
a monoidal category (i.e. understand the morphisms between them). It
is probably surprising that one often comes across the same, or closely
related monoidal categories, which one comes to call the Hecke
category. 

\begin{remark}
  In the geometric realisation the $v$ in the Hecke category has a
  transparent meaning (``Tate twist'', shift of weight
  filtration, \dots). However the $v$ is often well hidden in representation
  theory. (This is
  why one often sees an action of the Weyl group, and not the Hecke
  algebra, on the Grothendieck group.) That the grading is
  nevertheless there is the source of many mysteries and delights
  (Koszul duality, Jantzen filtration, \dots).
\end{remark}

\begin{remark}
The first geometric realisation can be thought of as a
``constructible'', whereas the second is ``coherent''. That both 
are incarnations of the Hecke category can be seen as a
``coherent/constructible duality'', and is often a shadow of geometric Langlands
duality. A basic example is the theory of Soergel bimodules, which
provides a coherent (bimodules over polynomial rings) description of
constructible sheaves. A second example is Bezrukavnikov's seminal
work \cite{Bez} relating the coherent and 
constructible realisations of the affine Hecke category.
\end{remark}

\subsection{The geometric realisation} \label{sec:geom}
Consider a complex reductive group $\GS$ with a fixed choice of Borel
subgroup $\BS$ and maximal torus $\TS \subset \BS$. To this we may associate the character
lattice $\Chi := X^*(\TS)$, cocharacter lattice $\Chi^\vee :=
X_*(\TS)$, roots $\Phi \subset \Chi$, coroots $\Phi^\vee \subset
\Chi^\vee$ etc. Our choice of Borel subgroup gives rise to positive roots
$\Phi^+$ and simple roots $\{ \alpha_s \}_{s \in S}$. We denote by
$(W,S)$ the Weyl group and its simple reflections, and $\ell : W \to
\ZM_{\ge 0}$ its length function.

For applications to representation theory we will also need to
consider affine Kac-Moody groups and their flag varieties. There is no
extra difficulty in considering general Kac-Moody groups, which we do
now.  However the
reader is encouraged to keep the setting of the previous paragraph in
mind.  Let us fix a generalised Cartan matrix
\[C = (c_{st})_{s,t
  \in S} \]
and let $(\hg_{\ZM}, \{ \alpha_s \}_{s \in S}, \{
\alpha^\vee_s \}_{s \in S})$ be a Kac-Moody root datum, so that
$\hg_{\ZM}$ is a free and finitely generated $\ZM$-module, $\alpha_s
\in \Hom(\hg_{\ZM},\ZM)$ are ``roots'' and $\alpha_s^\vee \in \hg_{\ZM}$ are
``coroots'' such that $\langle \alpha_s^\vee, \alpha_t \rangle =
c_{st}$. To this data we may associate a Kac-Moody group $\GS$  (a
group ind-scheme over $\CM$) together with a canonical Borel subgroup
$\BS$ and maximal torus $\TS \subset \BS$. 

We consider the flag variety $\GS/\BS$ (a projective variety in
the case of a reductive group, and an ind-projective variety in
general). As earlier, we denote by $W$ the Weyl group, $\ell$ the
length function and $\le$ the Bruhat order. We have the Bruhat
decomposition
\[
\GS/\BS = \bigsqcup_{w \in W} X_w \quad \text{where} \quad X_w :=\BS \cdot w\BS/\BS.
\]
The $X_w$ are isomorphic to affine spaces, and are called \emph{Schubert
cells}. Their closures $\overline{X}_w \subset \GS/\BS$ are
projective (and usually singular), and are called \emph{Schubert varieties}.

Fix a field $\Bbbk$ and consider $D^b_{\BS}(\GS/\BS;\Bbbk)$,
the bounded equivariant derived category with
coefficients in $\Bbbk$ (see e.g. \cite{BLu}).\footnote{By definition, any object of
  $D^b_{\BS}(\GS/\BS;\Bbbk)$ is supported on finitely many Schubert cells,
  and hence has finite-dimensional support.} This a
monoidal category under
convolution: given two complexes $\mathscr{F}, \GSS \in D^b_{\BS}(\GS/\BS;\Bbbk)$
their convolution is
\[
\FSS * \GSS := \mult_* ( \FSS \boxtimes_\BS \GSS ),
\]
where: $\GS \times_\BS \GS/\BS$ denotes the quotient of $\GS \times
\GS/\BS$ by $(gb,g' \BS) \sim (g,bg' \BS)$ for all $g,g' \in \GS$ and
$b \in \BS$;
$\mult : \GS \times_\BS \GS/\BS \to \GS/\BS$ is induced by the multiplication on $\GS$;
and $\FSS \boxtimes_\BS \GS  \in D^b_{\BS}(\GS \times_\BS \GS/\BS;\Bbbk)$ is
obtained via descent from $\FSS \boxtimes \GSS \in
D^b_{\BS^3}(\GS \times \GS/\BS;\Bbbk)$.\footnote{The reader is
  referred to \cite{Springer, Nad} for more detail on this
  construction.} (Note that $\mult$ is proper, and so $\mult_* = \mult_!$.)

\begin{remark} Under the function-sheaf correspondence the above
  definition categorifies convolution of $B$-biinvariant functions (if
  we work over $\FM_q$ rather than $\CM$, and $B$ denotes the
  $\FM_q$-points of $B$).
\end{remark}

For any $s \in S$ we can consider the parabolic subgroup
\[
\PS_s := \overline{\BS s\BS} = \BS s\BS \sqcup \BS \subset \GS.
\]
For any $s \in S$, set
\[
\EC_s := \Bbbk_{\PS_s/\BS}[1]
\]
We define the Hecke category (in its geometric incarnation) as follows
\[
\HC^{\geom}_{W,\Bbbk} := \langle \EC_s\; | \; s \in S \rangle_{*, \oplus,
  [\ZM], \ominus}.
\]
That is, we consider the full subcategory of
$D^b_{\BS}(\GS/\BS;\Bbbk)$ generated by $\EC_s$ under convolution ($*$),
direct sums ($\oplus$), homological shifts ($[\ZM]$) and direct summands
($\ominus$).

%\subsubsection{The categorification theorem.} 
Let $[ \HC^{\geom}_{W,\Bbbk} ]_{\oplus}$ denote the split Grothendieck
group\footnote{The split Grothendieck group $[\AC]_{\oplus}$ of an
additive category is the abelian group generated by symbols
$[A]$ for all $A \in \AC$, modulo the relations $[A] = [A'] + [A'']$
whenever $A \cong A' \oplus A''$.} of
$\HC^{\geom}_{W,\Bbbk} $. Because
$\HC^{\geom}_{W,\Bbbk} $ is a monoidal category, $[\HC^{\geom}_{W,\Bbbk} ]_{\oplus}$ is an
algebra via $[ \FSS ]\cdot [\GSS] = [ \FSS * \GSS]$. We view $[
\HC_{\geom}^{\Bbbk}]_{\oplus}$ as a $\ZM[v^{\pm 1}]$-algebra via $v \cdot [\FSS] :=
[\FSS[1]]$.
Recall the Kazhdan-Lusztig basis element $b_s = \delta_s + v$ for all
$s \in S$ from \S \ref{sec:kl}.

\begin{thm} \label{thm:geom_cat}
  The assignment $b_s \mapsto [\EC_s]$ for all $s \in S$ yields an isomorphism of
  $\ZM[v^{\pm 1}]$-algebras:
\[
H \simto [ \HC^{\geom}_{W,\Bbbk}]_{\oplus}.
\]
Moreover, for any $w \in W$ there exists a unique indecomposable
object $\EC_w \in \HC^{\geom}_{W,\Bbbk}$ (defined up to non-unique
isomorphism) such that $\EC_w$ is a direct summand of
\[
\EC_{\un{w}} := \EC_s * \EC_t* \dots * \EC_u
\]
for any reduced expression $\un{w} = (s, t, \dots, u)$ for $w$; and
such that $\EC_w$ is not isomorphic to a shift of a summand of $\EC_{\un{v}}$ for any
shorter expression $\un{v}$. The classes $[ \EC_w ]$ give a basis of
$H$ under the
isomorphism above.
\end{thm}

\begin{remark} When $\Bbbk$ is of characteristic zero, the complex
  $\EC_w$ of the theorem is the intersection cohomology complex of the
  Schubert variety $\overline{X_w}$. When $\Bbbk$ is of positive
  characteristic the $\EC_w$ are parity sheaves (see \S
  \ref{sec:parity} and \cite{WICM}). The theory of parity sheaves
  yields a beautiful and simple proof of the above theorem.
\end{remark}

The inverse to the isomorphism in the theorem is given by the
\emph{character map}
\begin{gather*}
  \ch : [\HC^{\geom}_{W,\Bbbk}]_{\oplus} \simto H   \qquad
% \end{gather*}
% defined by
% \begin{gather*}
\FSS  \mapsto \sum_{x \in W} \dim_\ZM ( H^*(\FSS_{x\BS/\BS}) ) v^{-\ell(x)} \delta_x
\end{gather*}
where: $\FSS_{x\BS/\BS}$ denotes the stalk of the constructible sheaf
on $\GS/\BS$ at the point $x\BS/\BS$ obtained from $\FSS$ by forgetting
$\BS$-equivariance; $H^*$ denotes cohomology; and $\dim_\ZM H^* := \sum
(\dim H^i) v^{-i} \in   \ZM[v^{\pm 1}]$ denotes graded dimension.

\begin{remark}
  We have tried to emphasise the role of the hom form above. In the
  geometric setting it gives the following. We set
\[
\Hom^\bullet(\FS, \GS) :=
\bigoplus_{n \in \ZM} \Hom_{D^b_{\BS}(\GS/\BS;\Bbbk)}(\FS, \GS[n])
\]
which is a free module over $H^*_{\BS}(\pt, \Bbbk)$. Under the above
isomorphism we have
\[
\langle [\FS], [\GS] \rangle := \begin{array}{c}
\text{graded rank of}\\
\text{$\Hom^\bullet(\FS, \GS)$ over $H^*_{\BS}(\pt, \Bbbk)$}
\end{array}
\]
which is Soergel's hom formula in this case. A proof of this equality
may be deduced from the local global spectral sequence (see
\cite[Proof of Prop. 3.4.4]{SL}). 
\end{remark}

% \begin{proof}
%   This discussion is lifted from Suppose
%   $\FS$ is a complex of sheaves on a variety $X$ and that we wish to calculate its (derived)
%   global sections (also known as its cohomology). Suppose that $X$ is equipped with an
%   exhaustive filtration
% \[
% X = X_0 \supset X_1 \supset \dots X_r =  \emptyset  
% \]
% by closed subsets. If we choose an injective
%   resolution $\IS$ of $\FS$ then for all $p$ we can consider the
%   subsheaf $\IS_p$ of sections with support in $X_p$. We
%   obtain in this way an exhaustive filtration
% \[
%   \IS = \IS_0 \supset \IS_1 \subset \dots \supset \IS_r = 0
% \]
% of $\IS$ with successive subquotients
% \[
% \IS_p/\IS_{p+1} = j_{p*}j_p^!\FS \quad\text{where $j_p : X_p
%   \setminus X_{p+1} \into X$ is the inclusion.}
% \]
% Taking global sections, we obtain a filtered complex, and hence a spectral sequence (the
% ``local-global spectral sequence'')
% \begin{equation}
%   \label{eq:localGlobal}
% E_1^{p,q} = H^{p+q}(X_{p}\setminus X_{p+1}, j_p^! \FS) \Rightarrow H^{p+q}(X,\FS).  
% \end{equation}
% \end{proof}

\subsection{Realisation via Soergel bimodules} \label{sec:hecke sbim}

The starting point for the theory of Soergel bimodules is a Coxeter
group $(W,S)$ and a representation
\[
W \acts V.
\]
Let us write $T$ for the reflections (conjugates of $S$) in $W$. We
assume that $V$ is defined over a field of characteristic $\ne 2$ and  is \emph{reflection faithful}, that is, that our
representation is faithful and that 
\[
\text{$V^x \subset V$ is a hyperplane, if and only if $x \in T$.}
\]

\begin{remark}
In particular, our representation is a reflection representation (see
\S\ref{sec:ref2}), however reflection faithfulness is a
much stronger condition.  
\end{remark}

With this data, we now define $R$ to be the symmetric algebra of $V$,
with $V$ placed in degree 2 and consider the monoidal category
$R\gbmod R$ of graded $R$-bimodules. Set
\[
B_s := R \otimes_{R^s} R(1) \in R\gbmod R,
\]
where the shift-grading-by-1 symbol (1) means that $1 \otimes 1 \in
B_s$ lives in degree $-1$. We define the Hecke category (in its incarnation as Soergel bimodules) as follows
\[
\HC^{\Soe}_{W,V} := \langle B_s\; | \; s \in S \rangle_{\otimes_R, \oplus,
  (\ZM), \ominus}.
\]
In other words, Soergel bimodules is defined to be the smallest
strictly full
subcategory of $R\gbmod R$ which contains $B_s$ for each $s \in S$,
and is closed under tensor product ($\otimes_R$), direct sum
($\oplus$), shift of grading ($(\mathbb{Z})$) and direct summands
($\ominus$). 

%\subsubsection{The categorification theorem.} 
Let $[ \HC^{\Soe}_{W,V} ]_{\oplus}$ denote the split Grothendieck
group (see the footnote a few pages earlier)  of
$\HC^{\Soe}_{W,V} $. Because
$\HC^{\Soe}_{W,V} $ is a monoidal category, $[\HC^{\Soe}_{W,V}]_{\oplus}$ is an
algebra via $[ B ]\cdot [B'] = [ B \otimes_R B']$. We view $[
\HC^{\Soe}_{W,V}]_{\oplus}$ as a $\ZM[v^{\pm 1}]$-algebra via $v \cdot [B] :=
[B(1)]$.
Recall the Kazhdan-Lusztig basis element $b_s = \delta_s + v$ for all
$s \in S$.

\begin{thm} \label{thm:Soe_cat}
  The assignment $b_s \mapsto [B_s]$ for all $s \in S$ yields an isomorphism of
  $\ZM[v^{\pm 1}]$-algebras:
\[
H \simto [ \HC^{\Soe}_{W,V}]_{\oplus}.
\]
Moreover, for any $w \in W$ there exists a unique indecomposable
object $B_w \in \HC^{\geom}_{W,\Bbbk}$ (defined up to non-unique
isomorphism) such that $B_w$ is a direct summand of
\[
B_{\un{w}} := B_s * B_t* \dots * B_u
\]
for any reduced expression $\un{w} = (s, t, \dots, u)$ for $w$; and
such that $B_w$ is not isomorphic to a shift of a summand of $B_{\un{v}}$ for any
shorter expression $\un{v}$. The classes $[ B_w ]$ give a basis of
$H$ under the
isomorphism above.
\end{thm}

\begin{remark}
  It is easy to see that in any reflection faithful representation
  distinct reflections fix distinct hyperplanes. Thus in a reflection
  faithful representation one has a bijection between reflections and
  the hyperplanes fixed by elements of $W$. This technical assumption
  makes the theory of Soergel bimodules work, and was discovered by
  Soergel following a comment by Polo on a first version of
  \cite{SB}. It appears essential to the proofs to ``classical''
  approaches to Soergel bimodules.
\end{remark}

\begin{remark} \label{rem:SoeHom}
Let us comment on Soergel's hom formula in the setting of Soergel
bimodules. We set
\[
\Hom^\bullet(B, B') :=
\bigoplus_{n \in \ZM} \Hom_{R\gbmod R}(B, B'(n))
\]
and it is a non-trivial theorem that this is a free module over $R$. Under the above
isomorphism we have
\begin{equation}
  \label{eq:shom}
\langle [B], [B'] \rangle := \begin{array}{c}
\text{graded rank of $\Hom^\bullet(B, B')$}
\end{array}
  \end{equation}
which is Soergel's hom formula in this case (see \cite[Lemma 6.13]{SB}). The proof is technical,
involves the reflection faithful assumption in a crucial way, and occupies most of \cite{SB}.
\end{remark}

\subsection{Abe's realisation} Although usually not Soergel bimodules
themselves, a key role in the classical theory of Soergel bimodules is
played by \emph{standard bimodules}
\[
R_x \in R\gbmod R \quad \text{for $x \in W$}.
\]
These are defined to be $R$ as a left module, with right action
twisted by $x$:
\[
m \cdot r := x(r) \cdot m \quad \text{for all $m \in R_x$ and $r \in R$}.
\]
Equivalently, they are functions on the ``twisted graph'':
\[
\Graph_x := \{ (xv, v) \; | \; v \in V \} \subset V \times V.
\]

A basic fact used repeatedly in the classical theory is that
\begin{equation}
  \label{eq:stvanish}
  \Hom(R_x, R_y) = 0 \quad \text{if $x \ne y$}.
\end{equation}
For example, this fact is used to construct the standard and
costandard filtrations on Soergel bimodules, which are crucial to
Soergel's proof of his hom formula in \cite{SB}.

% \begin{remark}
%   More subtle issues (e.g. controlling $\Ext^1(R_x, R_y)$) led Soergel
%   to consider reflection faithful representations. Thus
%   \eqref{eq:stvanish} is a necessary, but possibly not sufficient
%   condition for the classical theory to work. 
% \end{remark}

% \todo{(Currently we don't know
%   whether the classical theory really does break down for faithful
%   which don't necessarily satisfy Soergel's assumptions.)}

One can check directly that \eqref{eq:stvanish} holds if and only if
our representation $V$ is faithful. Thus the classical theory of
Soergel bimodules breaks down as soon as we loose faithfulness. In
recent work, Abe has found a beautiful way around this issue \cite{Abe}, that we
now describe.

Let $V$ denote our reflection representation of $W$ as above, but now
we place no assumptions on $V$. Let $R$ denote the symmetric algebra
on $V$, and $Q$ its fraction field.

Let $\MC$ denote the category of pairs $(M, \phi)$ where $M$ is graded
$R$-bimodule, and $\phi$ is an isomorphism
\begin{equation}
    \label{eq:Abedec}
\phi : Q \otimes M \simto \bigoplus_{x \in W} M_x \quad \text{in
  $Q\bmod R$}
\end{equation}
satisfying:
\begin{equation}
  \label{eq:Abecond}
m \cdot r = x(r) \cdot m \quad \text{for all $m \in M_x$ and $r \in R$.}  
\end{equation}
Morphisms are morphisms of bimodules preserving the decomposition
\eqref{eq:Abedec}. 

\begin{remark}
  If $W$ acts faithfully on $V$ then there is at most one
$\phi$ satisfying \eqref{eq:Abecond}. Hence in this case $\MC$ is a
full subcategory of $R \gbmod R$.
\end{remark}

Condition \eqref{eq:Abecond} guarantees that right multiplication by
any non-zero $r \in R$ is an isomorphism on the right-hand side of
\eqref{eq:Abedec}, and hence on the left-hand side also. We conclude
that if $(M, \phi) \in \MC$ then $Q \otimes M$ is actually a
$Q$-bimodule.

This observation allows us to define a monoidal
structure on $\MC$. Given $(M, \phi)$ and $(N,\psi)$, we define $(M,
\phi) \otimes (N,\psi)$ to be the object whose underlying bimodule is
$M \otimes N$ and whose decomposition is given by the convolution formula
\[
(Q \otimes M \otimes_R N)_x := \bigoplus_{x = yz} M_y \otimes_Q N_z.
\]

For simplicity we continue to assume that $V$ is defined over a field.
We make the following assumption:
\begin{equation}
  \label{eq:assump}
  \text{for all $s \in S$, its root $\alpha_s \in V$ and coroot
    $\alpha_s^\vee \in V^*$ are non-zero.}
\end{equation}
 (This is automatic as soon as our characteristic is $\ne 2$). Recall
 our generating objects
\[
B_s := R \otimes_{R^s} R(1) \in R\gbmod R.
\]
Our assumptions guarantee that $B_s$ admits a unique lift to $\MC$
(see \cite[\S 2.4]{Abe}). We denote the resulting object by $(B_s,
\phi_s)$. Abe's definition is the following:
\[
\HC^{\Abe}_{W,V} := \langle (B_s, \phi_s)\; | \; s \in S \rangle_{\otimes_R, \oplus,
  (\ZM), \ominus} \subset \MC.
\]

Following Abe, let us make one additional assumption:
\begin{equation}
  \label{eq:dihed assump}
\begin{array}{c}  \text{for all $s, t \in S$, which generate a finite
  parabolic subgroup} \\
\text{the restriction of $V$ to $\langle s, t \rangle \subset W$ is reflection faithful.}
\end{array}
\end{equation}
Under this assumption, Abe shows the analogue of Theorem
\ref{thm:Soe_cat} holds for $\HC^{\Abe}_{W,V}$ \cite[Theorem
1.1]{Abe}, and that his category is equivalent to the
diagrammatic category of \cite{EWS}. Thus $\HC^{\Abe}_{W,V}$ is a
reasonable candidate to be called a Hecke category.

\begin{remark}
  One of the reasons Abe's observation is so useful is that it often allows one to
  carry out proofs in the setting of more general realisations, using
  intuition from the case of Soergel bimodules. We will see examples
  of this below. In principle one could carry these out in the
  diagrammatic language of \cite{EWS}, but this often involves
  formidable calculations (see, for example, some of the calculations
  necessary in \cite{Hazi}).
\end{remark}

\begin{remark}
  As appears always to be the case, a key step in Abe's proof of the
  above results is that the analogue of \eqref{eq:shom} holds in his
  category. It is interesting to note that here he cannot follow the
  lines of Soergel's proof, and instead has to imitate the
  construction of the Libedinsky's light leaves basis \cite{LLL}. Thus
  this part of the proof more closely follows \cite{EWS}.
\end{remark}

\begin{remark}
Note that the assumption \eqref{eq:dihed assump} is fairly harmless. For example, for affine
groups in characteristic $p$ it rules out only a small number of
primes, which can be read off the rank 2 sub root systems.  (This
should be contrasted with the fact that $V$ is \emph{never} reflection
faithful on all of $W$ in these settings.)
\end{remark}

\subsection{Other realisations}
Let us briefly mention that we have ignored entirely two other
important realisations of the Hecke category:
\begin{enumerate}
\item The theory of sheaves on moment graphs, developed by
  Braden-MacPherson \cite{BMP} and Fiebig \cite{Fiebig1,Fiebig2}.
\item The diagrammatic description, which presents the Hecke category
  by generators and relations. The first steps in this program were
  carried out by Elias-Khovanov \cite{EKh} and Libedinsky \cite{LibRA}, and an (almost)
  complete presentation was obtained by Elias and the author
  \cite{EWS}. A detailed survey of the diagrammatic category in the
  context of the present survey is available in \cite{WICM}.
\end{enumerate}

\section{The Hecke category and its localisations} \label{sec:heckeloc}

\subsection{Hyperbolic localisation of the Hecke category}

Let $\GS, \BS, \TS$ be a Kac-Moody group as above. (We have the
example of a reductive group or its loop group in mind.) Choose a
cocharacter
\[
\chi : \CM^* \to \TS
\]
and let $\LS \subset \GS$ be its centraliser. (In case $\GS$ is a
reductive group, then $\LS$ is a Levi subgroup of $\GS$.)

Set
\begin{gather*}
\Chi_r := \{ \gamma \in \Chi \; | \; \langle \chi, \gamma \rangle = 0 \},
\\
\Phi_r := \Phi \cap \Chi_r \quad \text{and} \quad  \Phi_r^+ := \Phi^+
\cap \Chi_r, \\
W_r := \langle s_\alpha \; | \; \alpha \in \Phi_r \rangle.
\end{gather*}
Because the Lie algebra of $\LS$ is the $\chi$-fixed points on the Lie
algebra of $\GS$, we deduce that
$\Phi_r$ (resp. $\Phi_r^+$, $W_r$) are the roots (resp. positive
roots, Weyl group) of $\LS$. In particular,
\[
\BS_{\LS} :=  \BS \cap \LS
\]
 is a Borel subgroup of $\LS$.

\begin{prop}
  $(\GS/\BS)^{\CM^*}$ is a disjoint union of flag varieties for
  $\LS$. These flag varieties are naturally parametrized by ${}^r W$.
\end{prop}

\begin{proof}[Sketch of proof:]
The map $x \mapsto x\BS/\BS$ provides a bijection between $W$ and the
$\TS$-fixed points on the flag variety $\GS/\BS$. We first claim that
the $\LS$-orbit through a point $x\BS/\BS$ corresponding to $x \in
{}^r W$ is a flag variety for $\LS$. Indeed, the stabiliser of
$x\BS/\BS$ in $\LS$ is $\LS \cap x\BS x^{-1}$ with Lie algebra
\[
x^{-1} \cdot \lg \cdot x \cap \bg
\]
($\lg$ is the Lie algebra of $\LS$, and $\bg$ is the Lie algebra of $\bg$).
The weights in the Lie algebra are
\[
x^{-1}(\Phi_r) \cap \Phi^+ = \Phi_r^+.
\]
(The equality follows from \eqref{eq:minl}.) Thus that the stabiliser
is the standard Borel $\BS_{\LS}$ and our claim follows. It also
follows that the $\TS$-fixed points in $\LS \cdot x\BS/\BS$ consists
of the coset $W_r x \subset W$.

Thus the $\LS$-orbits through the $\TS$-fixed points corresponding to
$x \in {}^r W$ give us distinct copies of the flag variety of
$\LS$. It is also immediate that all these flag varieties belong to
the fixed point locus $(\GS/\BS)^{\CM^*}$. Finally, a tangent space
calculation at each fixed point shows that these orbits exhaust
$(\GS/\BS)^{\CM^*}$ and the proposition follows.
\end{proof}

We now consider the hyperbolic localisation functor. The above
proposition shows that we may regard it as a functor:
\begin{align*}
D^b_{\BS} (\GS/\BS,k) & \to \bigoplus_{x \in W_r} D^b_{\BS_{\LS}}(\LS/\BS_\LS,k) \\
\FS &\mapsto \FS^{*!}
\end{align*}
Because hyperbolic localisation preserves parity sheaves, we obtain a
functor 
\[
  \HC^{\geom}_{W,k} \stackrel{(-)^{*!}}{
      \longto} \bigoplus_{x \in W_r}  \HC^{\geom}_{W_r,k}
\]

\begin{question}
  Except in several easy cases (e.g. when $\LS = \TS$) I don't know
  whether one can define a monoidal structure on the right hand side
  in order to make the above a monoidal functor.
% \todo{One can make
%     this a monoidal functor after localisation.}
\end{question}

\subsection{Hyperbolic localisation and Grothendieck groups}

Recall the hyperbolic localisation bimodule from
\S\ref{sec:parabolic}. Hyperbolic localisation
categorifies the hyperbolic bimodule. More precisely:

\begin{thm} \label{thm:hypbim}
We have a
commutative diagram
\begin{equation*}
  \begin{tikzpicture}
    % upper stuff
    \node (ul) at (-5.3,0) {$\HC^{\geom}_{W_r, \Bbbk}$};
    \node at (-4,0) {$\acts$};
    \node (ml) at (-2,0) {$\bigoplus_{x \in {}^rW}
      \HC^{\geom}_{W_r,\Bbbk}$};
    \node (mr) at (2,0) {$\HC^{\geom}_{W,\Bbbk}$};
     \node at (4,0) {$\racts$};
     \node (ur) at (5.3,0) {$\HC^{\geom}_{W,\Bbbk}$};
     % lower stuff
     \node (ll) at (-3.7,-2) {$H_{W_r}$};
         \node at (-2.2,-2) {$\acts$};
         \node (l) at (0,-2) {$\ZM[W]$};
                  \node at (2.2,-2) {$\racts$};
                  \node (lr) at (3.7,-2) {$H_{W}$};
                  % arrows stuff
                  \draw[->] (mr) to node[above] {$(-)^{*!}$} (ml);
               \draw[dashed,->] (ur) to node[right] {\small $\ch$} (lr);
               \draw[dashed,->] (ul) to node[left] {\small $\ch$} (ll);
               % middle arrows
                              \draw[dashed,->] (mr) to node[right]
                              {\small $\hat{\ch}$} (l);
                              \draw[dashed,->] (ml) to node[left]
                              {\small $\oplus \hat{\ch}$} (l);               
% def of hyperbolic bimodule
\node[rotate=90] at (0,-2.5) {$=$};
\node at (0,-3) {$\bigoplus_{x \in W} \ZM \hat{\delta}_x$};
             \end{tikzpicture}
\end{equation*}
where:
\begin{enumerate}
\item  $\ch$ is defined as in \S \ref{sec:geom};
\item $\hat{\ch}$ is
$\ch$ at $v := 1$; 
\item $\oplus \hat{\ch}$ is defined via
\[
\oplus \hat{\ch} ( (\FS_x)_{x \in {}^r W}) := \sum_{x \in {}^r W} (\hat{\ch}(\FS_x)) \hat{\delta}_x.
\]
\end{enumerate}
\end{thm}

\begin{proof}
  The commutativity of the left and right squares (or more accurately
  rombhi) is a specialisation at $v := 1$ of the fact that $\ch$ is a
  homomorphism of algebras (see  \S \ref{sec:geom}). Thus we only need to check commutativity
  of the middle triangle. The basic reason for the commutativity of this
  triangle is that hyperbolic localisation preserves Euler
  characteristic. We need to be a little more careful though, as
  $\hat{\ch}$ is not quite given by the Euler characteristic at
  $T$-fixed points. It is enough to check commutativity for a parity
  sheaf $\FS$, in which case $\hat{\ch}$ is given by the Euler
  characteristic at $T$-fixed points up to a sign, depending on
  whether $\FS$ is even or odd. Now we only need to check that
  hyperbolic localisation preserves being even or odd, which follows
  from the fact (again) that hyperbolic localisation preserves Euler
  characteristics.
\end{proof}

% \todo{Not sure how many words in the previous paragraph are
%   defined. even odd etc.}

\begin{remark}
  We leave it to the reader to take $k = \QM$ above and deduce the
  positivity properties in Theorem \ref{thm:pos2}.
\end{remark}

\subsection{Smith localisation of the Hecke category}

Let $\GS, \BS, \TS$ be as in the previous section. (Again, we have the
example of a reductive group or its loop group in mind.) We now fix a
prime $p$ and assume that our coefficients $\Bbbk$ are either a
finite field of characteristic $p$, $\ZM_p$, or a finite extension of
one of these. Fix an element of order $p$
\[
\z \in \TS
\]
and let $\LS \subset \GS$ be the identity compotent of its
centraliser.

\begin{remark}
If $\GS$ is reductive then such $\LS$ (as $\zeta$ varies over all
elements of finite but not necessarily prime order)
constitute all closed connected reductive subgroups of $\GS$ with
maximal torus $\TS$.
If we assume that $\GS$ is semi-simple and simply connected, then the
centraliser is already connected and is determined by a subset of the extended Dynkin
diagram of $\GS$ (Borel-de Siebenthal theory). Note that $\LS$ need
not be a Levi subgroup of $\GS$.
\end{remark}

\begin{ex} \label{ex:BdS}
  The reader is encouraged to keep two examples in mind:
  \begin{enumerate}
  \item $\GS = \Sp_4$, $p = 2$, and $\z = \diag(1, -1, -1, 1)$. In
    which case
\[
\LS = \SL_2 \times \SL_2.
\]
  \item $\GS$ a loop group (a central extension of $\CM^*_{\rot} \ltimes
    G((t))$ for a reductive group $G$, where $\CM^*_\rot$ acts by
    ``loop rotation''), $p$ is arbitrary and $\z \in
    \CM^*_\rot$ is a $p^{th}$ root of unity, in which case $\LS$ is a
    central extension of
\[
\CM^*_{\rot} \ltimes
    G((t^p)).
\]
  \end{enumerate}
\end{ex}

Echoing earlier notation, set
\begin{gather*}
\Chi_r := \{ \gamma \in \Chi \; | \; \chi(\z) = 1 \},
\\
\Phi_r := \Phi \cap \Chi_r \quad \text{and} \quad  \Phi_r^+ := \Phi^+
\cap \Chi_r, \\
W_r := \langle s_\alpha \; | \; \alpha \in \Phi_r \rangle.
\end{gather*}
Because the Lie algebra of $\LS$ is the $\zeta$-fixed points on the Lie
algebra of $\GS$, we deduce that
$\Phi_r$ (resp. $\Phi_r^+$, $W_r$) are the roots (resp. positive
roots, Weyl group) of $\LS$. In particular,
\[
\BS_{\LS} :=  \BS \cap \LS
\]
 is a Borel subgroup of $\LS$. As earlier we have (with essentially the same proof):

\begin{prop}
  $(\GS/\BS)^{\z}$ is a disjoint union of flag varieties for
  $\LS$. These flag varieties are naturally parametrised by ${}^r W$.
\end{prop}

\begin{ex} We continue Example \ref{ex:BdS}. Part (1) is dicussed in Remark
  \ref{rem:geomc2}. For (2), let 
\[
\Fl := G((t))/\Iw
\]
denote the affine flag variety (the flag variety for the loop
group under consideration). For $\z$ as above we have
\[
(\Fl)^{\zeta} = \bigoplus_{x \in {}^r W} G((t^p))/(\Iw \cap G((t^p))).
\]
Thus each of the components is again isomorphic to an affine flag variety.
\end{ex}

We now consider the Smith localisation functor. The above
proposition shows that we may regard it as a functor:
\begin{align*}
D^b_{\BS} (\GS/\BS, \Bbbk) & \to \bigoplus_{x \in W_r} \Perf_{\BS_\LS} ( \LS/\BS_\LS,\TS_{\Bbbk}) \\
\FS &\mapsto \Sm(\FS)
\end{align*}
As we have already noted, one may check easily that $\Sm$ preserves
parity sheaves. Thus, if we define
\[
\HC^{\Sm}_{W_r,\Bbbk} := \text{image of $\HC^{\geom}_{W_r,\Bbbk}$ in $\Perf_{\BS_\LS}( \LS/\BS_{\LS},
\TS_\Bbbk)$}
\]
we obtain a Smith localisation functor between Hecke categories: 
\[
  \HC^{\geom}_{W, \Bbbk} \stackrel{\Sm}{
      \longto} \bigoplus_{x \in W_r}  \HC^{\Sm}_{W_r,\Bbbk}.
\]

\subsection{Smith localisation and Grothendieck groups}

Recall the hyperbolic localisation bimodule from
\S\ref{sec:parabolic}, now considered for a reflection subgroup $W_r$
which is potentially not a parabolic subgroup. Smith localisation
categorifies the hyperbolic bimodule. More precisely, we have the
following theorem, whose proof is similar to that of Theorem \ref{thm:hypbim}.

\begin{thm} \label{thm:sm}
We have a
commutative diagram
\begin{equation*}
  \begin{tikzpicture}
    % upper stuff
    \node (ul) at (-5.3,0) {$\HC^{\geom}_{W_r, \Bbbk}$};
    \node at (-4,0) {$\acts$};
    \node (ml) at (-2,0) {$\bigoplus_{x \in {}^r W}
      \HC^{\Sm}_{W_r,\Bbbk}$};
    \node (mr) at (2,0) {$\HC^{\geom}_{W,\Bbbk}$};
     \node at (4,0) {$\racts$};
     \node (ur) at (5.3,0) {$\HC^{\geom}_{W,\Bbbk}$};
     % lower stuff
     \node (ll) at (-3.7,-2) {$H_{W_r}$};
         \node at (-2.2,-2) {$\acts$};
         \node (l) at (0,-2) {$\ZM[W]$};
                  \node at (2.2,-2) {$\racts$};
                  \node (lr) at (3.7,-2) {$H_{W}$};
                  % arrows stuff
                  \draw[->] (mr) to node[above] {$\Sm$} (ml);
               \draw[dashed,->] (ur) to node[right] {\small $\ch$} (lr);
               \draw[dashed,->] (ul) to node[left] {\small $\ch$} (ll);
               % middle arrows
                              \draw[dashed,->] (mr) to node[right]
                              {\small $\hat{\ch}$} (l);
                              \draw[dashed,->] (ml) to node[left]
                              {\small $\oplus \hat{\ch}$} (l);               
% def of hyperbolic bimodule
\node[rotate=90] at (0,-2.5) {$=$};
\node at (0,-3) {$\bigoplus_{x \in W} \ZM \hat{\delta}_x$};
             \end{tikzpicture}
\end{equation*}
where:
\begin{enumerate}
\item  $\ch$ is defined as in \S \ref{sec:geom};
\item $\hat{\ch}$ is
$\ch$ at $v := 1$; 
\item $\oplus \hat{\ch}$ is defined via
\[
\oplus \hat{\ch} ( (\FS_x)_{x \in {}^r W}) := \sum_{x \in {}^r W} (\hat{\ch}(\FS_x)) \hat{\delta}_x.
\]
\end{enumerate}
\end{thm}

\begin{remark}
  We leave it to the reader to use this theorem to deduce instances of
  the positive properties in Theorem \ref{thm:pos3}.
\end{remark}

% \begin{proof}
%   The commutativity of the left and right squares (or more accurately
%   rombhi) is a specialisation at $v := 1$ of the fact that $\ch$ is a
%   homomorphism of algebras (see  \S \ref{sec:geom}). Thus we only need to check commutativity
%   of the middle square. The basic reason for the commutativy of this
%   square is that hyperbolic localisation preserves Euler
%   characteristic. We need to be a little more careful though, as
%   $\hat{\ch}$ is not quite given by the Euler characteristic at
%   $T$-fixed points. It is enough to check commutativity for a parity
%   sheaf $\FS$, in which case $\hat{\ch}$ is given by the Euler
%   characteristic at $T$-fixed points up to a sign, depending on
%   whether $\FS$ is even or odd. Now we only need to check that
%   hyperbolic localisation preserves being even or odd, which follows
%   from the fact (again) that hyperbolic localisation preserves Euler
%   characteristics.
% \end{proof}

% \todo{Not sure how many words in the previous paragraph are
%   defined. even odd etc.}

\subsection{Soergel bimodules for reflection subgroups} \label{subsec:reflection}
Let $V$ denote our reflection representation with roots $\alpha$,
coroots $\alpha^\vee$ etc. as in \ref{sec:hecke sbim}. Fix a subspace
\[
V_{r} \subset V
\]
and let $W_r$ denote the reflection subgroup generated by the
reflections $s_\alpha$ in roots $\alpha \in V_r$. We have seen in \S
\ref{sec:ref1} and \S\ref{sec:ref2} that $W_r$ is a Coxeter group with
natural Coxeter generators $S_r \subset W_r$. (We try
to write $t$ for reflections in $S_r$, to remind ourselves that these
are different from the simple reflections in $S$.)

Note that $V_r$ is certainly $W_r$ stable because
\[
s_\a(\lambda) = \lambda - \langle \a^\vee, \lambda \rangle \alpha
\]
for all $\lambda \in V_r$. We make the following assumption.
\begin{equation}
  \label{eq:RFAssumption}
\text{$V_r$ is reflection faithful, as a representation of $W_r$.}
\end{equation}

\begin{remark}
  This assumption is to make sure that
  Soergel bimodules behave well. However, the reader may check that
  this assumption can be replace with a much weaker one if one instead
  uses Abe's realisation of the Hecke category. We had originally
  planned on explaining the necessary changes below, but ran out of
  time when preparing these notes. 
\end{remark}

We denote by $R$ (resp. $R_r$) the symmetric
algebra on $V$ (resp. $V_r$). We have a canonical $W_r$-equivariant embedding
\[
R_r \into R.
\]
Because $(W_r, S_r)$ is a Coxeter group, we can consider Soergel
bimodules with respect to either or $V$ or $V_r$. Our notation is as
follows:
\begin{gather*}
\HC^\Soe_{W_r,V_r}: \text{Soergel bimodules for $W_r$ and the ``little'' representation $V_r$;}\\
\HC^\Soe_{W_r,V}: \text{Soergel bimodules for $W_r$ and the ``big'' representation $V$.}
\end{gather*}
(We will often drop the superscript $\Soe$ below, to reduce clutter.)

The goal of this section is to see that there is not a great difference
between these two categories. More precisely:

\begin{thm} \label{thm:ref_equiv}
  There is an equivalence of additive categories:
\[
\HC_{W_r,V_r} \otimes_{R_r} R \simto \HC_{W_r,V}.
\]
\end{thm}

\begin{remark}
  The issue is that (as far as we know) there is no natural
  \emph{monoidal} equivalence. (Except in special situations, for
  example, when the inclusion $V_r \into V$ admits a $W_r$-equivariant
  splitting.) Thus establishing this equivalence requires a little
  care. We will deduce the theorem from considerations of a bimodule category
  on which both categories act. This issue of changing representations
  for Soergel bimodules was first addressed in \cite{Lequiv}, which
  has influenced the discussion below.
\end{remark}

Consider $R_r \gbmod R$, the category of graded $(R_r,
R)$-bimodules. This category is naturally a bimodule for the monoidal
category of graded $R_r$-bimodules on the left, and graded
$R$-bimodules on the right. In particular this category is a bimodule
for the appropriate categories of Soergel bimodules acting on the left
and the right:
\[
\HC_{W_r,V_r} \acts R_r \gbmod R \racts \HC_{W_r,V}
\]

Given $t \in S_r$ we denote the generators of the appropriate Hecke
categories as follows:
\[
B_{r,t} := R_r \otimes_{(R_r)^t} R_r(1) \in \HC_{W_r,V_r}\quad \text{and} \quad
B^r_t := R \otimes_{R^t} R(1) \in \HC_{W_r,V}.
\]
(The sub- and superscripts are in order to distinguish these objects
from Soergel bimodules for the larger group, which will be considered
momentarily.) Given an expression $\un{w} = (t_1, t_2, \dots, t_m)$ in $S_r$ we consider the
corresponding Bott-Samelson bimodules
\begin{gather*}
  \label{eq:1}
B_{r,\un{w}} := B_{r,t_1} B_{r,t_2} \dots B_{r,t_m} \in \HC_{W_r,V_r}\quad \text{and} \quad
  B^{r}_{\un{w}} := B^r_{t_1} B^r_{t_2} \dots B^r_{t_m}\in \HC_{W_r,V}.
\end{gather*}

The following relates the left and right action considered above:

\begin{prop} \label{prop:BSslide}
For any expression $\un{w} = (t_1, t_2, \dots, t_m)$ in $S_r$ we
  have a canonical isomorphism
  \[
B_{r,\un{w}} \otimes_{R_r} R = R \otimes_{R} B^r_{\un{w}} \quad \text{in $R_r \gbmod R$,}
  \]
where on both sides of the equality $R$ is regarded as an $(R_r,R)$-bimodule. 
\end{prop}

\begin{proof}
  It is enough to check this when $\un{w}$ consists of a single simple
  reflection. In this case the left hand side is
\begin{gather*}
B_{r,t} \otimes_{R_r} R = R_r \otimes_{(R_r)^t} R_r \otimes_{R_r} R(1)
= R_r \otimes_{(R_r)^t} R(1)
\end{gather*}
and the right hand side is
\begin{gather*}
R \otimes_R B_t^r = R \otimes_{R^t} R (1).
\end{gather*}
As graded right $R$-modules, both sides are free of graded rank $(v
+v^{-1})$ with basis $\{ 1 \otimes 1, \alpha_t \otimes 1 \}$. The
obvious map from the left hand side to the right hand side maps 
a basis to a basis, and hence is an isomorphism.
\end{proof}

\begin{prop} \label{prop:homisos}
  \begin{enumerate}
  \item Given Soergel bimodules $B', B'' \in \HC_{W_r,V_r}$, action on $R \in
  R_r \gbmod R$ yields an isomorphism
\[
\Hom_{\HC_{W_r,V_r}}(B', B'') \otimes_{R_r} R \simto \Hom_{R_r \gbmod
  R}( B' \otimes_R R, B'' \otimes_{R'} R).
\]
\item Given Soergel bimodules $B', B'' \in \HC_{W_r,V}$, action on $R \in
  R_r \gbmod R$ yields an isomorphism
\[
\Hom_{\HC_{W_r,V}}(B', B'') \simto \Hom_{R_r \gbmod
  R}( R \otimes_R B', R \otimes_R B'').
\]
  \end{enumerate}
\end{prop}

\begin{proof} The isomorphism in (1) is a commutative algebra fact: As
  $R_r \otimes R$ is a flat $R_r \otimes R_r$-algebra, and $B'$ is
  finitely-generated (hence finitely-presented) over $R_r \otimes R_r$
  this follows from \cite[Proposition 2.10]{Eisenbud}.

The isomorphism (2) is more delicate. Certainly restriction of scalars
yields an embedding
\[
\Hom_{\HC_{W_r,V}}(B', B'') \into \Hom_{R_r \gbmod
  R}( R \otimes_R B', R \otimes_R B'')
\]
and our task is to see that it is an isomorphism. Now we may assume that $B'$ and $B''$
are Bott-Samelson bimodules corresponding to expressions $\un{w}$ and
$\un{w}'$. In this case we have:
\begin{align*}
  \Hom_{\HC_{W_r,V}}(B^r_{\un{w}}, B^r_{\un{w}'}) & \into
                                                   \Hom_{R_r \gbmod
  R}(R
                                                   \otimes_R
                                                   B^r_{\un{w}}, R
                                                   \otimes_R
                                                   B^r_{\un{w}'}) \\
& = \Hom_{R_r \gbmod
  R}(B_{r,\un{w}} \otimes_{R_r} R,
  B_{r,\un{w}'}  \otimes_{R_r} R )\\
& = \Hom_{\HC_{W_r,V_r}}(B_{r,\un{w}},
  B_{r,\un{w}'}  ) \otimes_{R_r} R.
\end{align*}
The first equals sign follows from Proposition \ref{prop:BSslide} and
the second equals sign follows from part (1). Now Soergel's hom
formula (see Remark \ref{rem:SoeHom}) tells us that
\[
\Hom_{\HC_{W_r,V}}(B^r_{\un{w}},
B^r_{\un{w}'}) \quad \text{and} \quad \Hom_{\HC_{W_r,V_r}}(B_{r,\un{w}},
  B_{r,\un{w}'}  ) \otimes_{R_r} R\]
have graded rank (as free right $R$-modules respectively) given by the same expression in the Hecke algebra. We 
  conclude that our injection is in fact an isomorphism.
\end{proof}

We can now prove the main theorem of this section:

\begin{proof}[Proof of Theorem \ref{thm:ref_equiv}]
  Consider $\QC$, the strictly full, additive, graded and Karoubian subcategory of $R_r
  \gbmod R$ generated by $R \in R_r \gbmod R$ under the left action of
  $\HC_{W_r,V_r}$.
(That is, objects of $\CC$ are those graded $(R_r,
  R)$-bimodules which are isomorphic to direct sums of shifts of
  direct summands of objects of the form $B_{r,\un{w}} \otimes_{R_r}
  R \in R_r \gbmod R$.)
By Proposition \ref{prop:BSslide} this
  coincides with the full subcategory generated by $R \in R_r \gbmod
  R$  under the right action of $\HC_{W_r,V}$. Thus we can view $\QC$
  as a bimodule category:
\[ \begin{tikzpicture}[scale=0.7]
  \node at (-4,0) {$\HC_{W_r,V_r}$};
\node at (-2,0) {$\acts$};
\node at (0,0) {$\QC := \langle R \rangle $};
\node at (2,0) {$\racts$};
  \node at (4,0) {$\HC_{W,V}$};
\node[rotate=-90] at (0,-1) {$\subset$};
\node at (0,-2) {$R_r \bmod R$};
\end{tikzpicture}
\]
By Proposition
  \ref{prop:homisos} we have equivalences of additive categories
\[
\HC_{W_r,V_r} \otimes_{R_r} R \simto \QC \simfrom \HC_{W_r,V} 
\]
and the theorem follows.
\end{proof}

\subsection{The endomorphism ring of a Soergel bimodule}
\label{sec:end}

In the next subsection we will prove an important result concerning
localisations of Soergel bimodules. However before we come to this we
need to recall some structure theory concerning endomorphism
rings.

\begin{remark}
  The material in this subsection is an easy consequence of properties
  of the light leaves and double leaves basis, see in particular \cite[\S 6]{EWS} and
  \cite[\S 11.3]{soergelbook}. The light leaves basis was introduced
  in \cite{LLL}.
\end{remark}

Let $B \in \HC_{W,V}$ denote an indecomposable Soergel
bimodule. We wish to study the endomorphism ring of $B$.
By the classification of indecomposable Soergel bimodules,
$B$ is isomorphic (up to a shift) to $B_w$ for some $w \in W$. Thus
we may assume:
\begin{equation*}
%  \label{eq:3}
  B = B_w \quad \text{for $w \in W$.}
\end{equation*}

Now we may filter $\End(B_w)$ as follows. For any $x \in W$, let
$(\HC_{(W,V)})_{\le x}$ (resp. $(\HC_{(W,V)})_{< x}$) denote the full graded additive subcategory generated by
indecomposable Soergel bimodules $B_y$ with $y \le x$ (resp. $y < x$). We consider
\[
I_{\le x} = \left \{ f \in \End(B_w) \; \middle | \;
\begin{array}{c}
  \text{$f$ admits a factorisation} \\
\text{$B_w \to B' \to B_w$ with $B' \in (\HC_{(W,V)})_{\le x}$}
\end{array}\right \}
\]
and define $I_{<x}$ similarly. 

Clearly $I_{\le x}$ and $I_{<x}$ are ideals in $\End(B_w)$. It is known that:
\begin{gather}
  \label{eq:id}
    I_{\le w}/ I_{<w} = \id_{B_w} \cdot R; \\
  \label{eq:free}
    I_{\le x}/ I_{< x} \text{ is free as a
      right $R$-module, for all $x \in W$}; \\
  \label{eq:idealproduct}
  I_{\le x} \cdot I_{\le y} \subset \sum_{z \in W \atop z \le x; z
  \le y} I_{\le z} \quad \text{for all $x, y \in W$}; \\
  \label{eq:idealsquare}
I_{\le x} \cdot I_{\le x} \subset \sum I_{\le x} \cdot R^+ +
  \sum_{z < x} I_{\le z}\quad \text{for all $x \in W$ with $x < w$}
\end{gather}
where $R^+$ denotes the elements of $R$ of positive degree.

By property \eqref{eq:free} we can choose a basis $\{ f_i \}_{i =
  0}^m$ for $\End(B_w)$ such that $f_0 = \id_{B}$ and all $f_i \in
I_{<w}$.

\begin{lem} \label{lem:fis}
  There exists a positive integer $M$ such that any product of $M$ elements of the
  set $\{ f_i \; | \; 1 \le i \le m \}$ belongs to $\End(B_w) \cdot R^+$. 
\end{lem}

%\todo{ CHECK ME AGAIN }
\begin{proof} 
  Let us fix an enumeration
\[
w = w_0 > w_1 > \dots > w_N = \id
\]
of the elements less than $w$, refining the Bruhat order. We will
prove the following statement, by induction on $j$:
\begin{equation}
  \label{eq:6}
  \begin{array}[c]{c}
\text{Any product of $2^j$ elements of the
  set $\{ f_i \; | \; 1 \le i \le m \}$ belongs to } \\
\text{$I_j := \End(B_w) \cdot R^+ + \sum_{k > j} I_{\le w_k}$.}
  \end{array}
\end{equation}
where we interpret $I_{\le w_k} = 0$ if $k > N$. Clearly the statement
for $j = N$ implies the lemma (with $M = 2^N$).

The statement clearly holds for $j = 0$. Now consider a product $\pi$ of
$2^j$ elements of $\{ f_i \; | \; 1 \le i \le m \}$. We can write this
as the product $\pi = \pi_1 \pi_2$, where $\pi_1$ and $\pi_2$ belong
to $\End(B_w) \cdot R^+ + \sum_{k > j-1} I_{\le w_k}$. Then
properties \eqref{eq:idealproduct}, \eqref{eq:idealsquare} and the
fact that $\End(B_w) \cdot R^+$ is an ideal imply that $\pi \in I_j.$
The lemma follows.
\end{proof}

\subsection{Localising Soergel bimodules for reflection subgroups}

Recall our representation $V$ and its subspace $V_r \subset V$. As
above we let $R$ denote the symmetric algebra on $V$. Let $\mg$ denote
the ideal generated by $V_r \subset V$, and $R_\mg$ with the
corresponding local ring, with maximal ideal $\mg$. Set
\[
\hat{R} := \lim_{\leftarrow} R_{\mg} /(\mathfrak{m}^\ell)
\]
for the completion of $R_{\mg}$ along $\mg$. It will be important
below that $\hat{R}$ is a complete local ring.

% For what follows it will be necessary to make a fixed but arbitrary
% choice of the following:
% \begin{equation}
%   \label{eq:4}
%   \begin{array}{c}\text{A $W_r$-stable subset $\Xi$ of $V_r$ such
%            that:}\\
% \text{$\Xi$ generates $V_r$ and $0 \notin \Xi$}. \end{array}
% \end{equation}
% (For example, if $V_r$ is generated by the roots it contains, we could
% take $\Xi$ to be the set of such roots.)

% Let $\mathfrak{m} := ( \Xi) \subset R$ denote the
% ideal that it generates, and $R_{(\mathfrak{m})}$ the localisation of
% $R$ at $\mathfrak{m}$. Set
% \[
% R^{(r)} := R_{(\mathfrak{m})}[ \Sigma^{-1}] \qquad \text{where $\Sigma
%   = \left \{ \frac{\sigma}{\sigma'} \; \middle | \; \sigma, \sigma'
%     \in \Sigma \right \}$.}
% \]
% Note that $\mathfrak{m}$ becomes principal in $R^{(r)}$. Let
% \[
% \hat{R} := \lim_{\leftarrow} R^{(r)}/(\mathfrak{m}^\ell)
% \]
% denote the completion of $R^{(r)}$ along $\mathfrak{m}$. Then
% $\hat{R}$ is a discrete local ring.

% \begin{remark}
%   Below we will only need that $\hat{R}$ is a complete local
%   ring.
% \todo{
% XXX Is this true? Goal is to write the rest of this section with
% $\hat{R}$ having either meaning, and see if I actually need the fact
% that it is a DVR.
% XXX }
% \end{remark}

% \begin{remark}
%   The ring $R^{(r)}$ is obtained from $R_{(\mathfrak{m})}$ by ``blowing
%   up the maximal ideal''. This is a standard construction in
%   commutative algebra. See XXX not sure if this remark is needed. XXX
%   The idea to consider this ring is Hazi's.
% \end{remark}

Note that $W_r$ acts naturally on $\mathfrak{m}$, $R_{\mg}$ and
  $\hat{R}$. This allows us to imitate the definition of Soergel
  bimodules for the reflection subgroup, with $R$ replaced by
  $\hat{R}$. Define
\[
\hat{B}^r_t := \hat{R} \otimes_{ \hat{R}^t}  \hat{R}
\]
and
\[
\hat{\HC}_{W_r,V} := \langle \hat{B}^r_t \; | t \in S_r\;
\rangle_{\otimes, \oplus, \ominus } \subset
\hat{R} \bmod \hat{R}
\]
to be the smallest strictly full subcategory of
$\hat{R}$-bimodules containing $\hat{B}^r_t$ for all $t \in
S_r$, and closed under taking tensor products, direct sums and direct summands.

\begin{remark}
  Note that $\hat{R}$ is no longer graded. Hence
  any grading (on $R$-bimodules etc.) is lost after passage to these rings.
\end{remark}

\begin{prop} \label{prop:Rhat}
  Given any $M \in \HC_{W_r, V}$ the $(R, \hat{R})$-bimodule $M
  \otimes_R \hat{R}$ is naturally an
  $\hat{R}$-bimodule. We obtain in this way a monoidal functor:
  \begin{align*}
    \HC_{W_r,V} &\to \hat{\HC}_{W_r,V} \\
M &\mapsto M \otimes_R \hat{R}.
  \end{align*}
\end{prop}

\begin{proof}
  Suppose that $\un{w} = (t_1, \dots, t_m)$ is an expression in $S_r$. We will prove that
  the natural inclusion
\[
B^r_{\un{w}} \otimes_R \hat{R}
\to 
%\hat{B}^r_{\un{w}} := 
\hat{R}
\otimes_{\hat{R}^{t_1}} \dots
  \otimes_{\hat{R}^{t_m}} \hat{R} 
\]
is an isomorphism. This implies the proposition.

By induction it is enough to check this when $\un{w} = (t)$. That is,
we want to show that the natural map
\[
R \otimes_{R^t} R \otimes_R \hat{R} = R \otimes_{R^t} \hat{R} 
\to
\hat{R} \otimes_{\hat{R}^t} \hat{R}\]
is an isomorphism. Now $B_t^r$ (resp. $\hat{R} \otimes_{\hat{R}^t} \hat{R}$) is free as a right $R$-
(resp. $\hat{R}$-)
module with basis $1 \otimes 1$ and $\alpha_t \otimes 1$. Hence $B_t^r
\otimes_R \hat{R}$ is a free $\hat{R}$-module with basis
$1 \otimes 1 \otimes 1$ and $\alpha_t \otimes 1 \otimes 1$. The
natural map above sends this basis to our basis of $\hat{R}
\otimes_{\hat{R}^t} \hat{R}$, and hence is an isomorphism as claimed.
\end{proof}

\begin{remark}
  For a related result, see \cite[Lemma 3.20]{EWS}.
\end{remark}

% Given any $R$-bimodule $B$, we can consider its localisation $B
% \otimes_R \hat{R}$, which is an $(R, \hat{R})$-bimodule.

\subsection{The Krull-Schmidt property} \label{sec:locKS}
Recall that an additive  category is \emph{Krull-Schmidt} if: every object is
isomorphic to a finite direct sum of indecomposable objects, and an
object is indecomposable if and only if its endomorphism ring is
local. In a Krull-Schmidt category every object decomposes uniquely
as a direct sum of indecomposables.

\begin{prop} \label{prop:locKS}
  $\hat{\HC}_{W_r,V}$ is Krull-Schmidt.
\end{prop}

Before we prove it, we need a preparatory lemma:

\begin{lem} \label{lem:lemExt}
  Let $B, B' \in \HC_{W_r, V}$. We have an isomorphism
\[
\Hom_{\HC_{W_r, V} }(B,B') \otimes_R \hat{R} \simto \Hom_{\hat{\HC}_{W_r, V} }(B \otimes_R \hat{R}, B' \otimes_R \hat{R}).
\]
\end{lem}

\begin{proof} 
Because $B$ is finitely-generated over $R \otimes
  R$, we have
\[
\Hom_{R\bmod R}(B, B') \otimes_R \hat{R}\simto 
\Hom_{R\bmod \hat{R}}(B \otimes_R \hat{R}, B' \otimes_R \hat{R})
\]
by elementary facts in commutative
  algebra (e.g. \cite[Proposition 2.10]{Eisenbud}). Then we have
\[
 \Hom_{R\bmod \hat{R}}(B \otimes_R \hat{R}, B' \otimes_R \hat{R}) =
\Hom_{\hat{\HC}_{W_r, V} }(B \otimes_R \hat{R}, B' \otimes_R \hat{R})
\]
by Proposition \ref{prop:Rhat}.
\end{proof}

\begin{proof}
It is known (see e.g. \cite[\S 11.4.2]{soergelbook}) that an additive
category is Krull-Schmidt if it is Karoubian and the endomorphism ring
of any object is semi-perfect.(Recall that a ring is $A$ is semi-perfect if we
can write the identity as a sum $e_1 + \dots + e_m$ of orthogonal idempotents,
with each $e_iAe_i$ local.) Our category   $\hat{\HC}_{W_r,V}$ is
Karoubian by definition, so it remains to check that endomorphism
rings are semi-perfect.

 Any object in   $\hat{\HC}_{W_r,V}$ is a
direct summand of an object of a form $B \otimes_R \hat{R}$ for some
$B \in \HC_{W_r,V}$. By Lemma \ref{lem:lemExt} we have
\[
\End_{\hat{\HC}_{W_r,V}}(B \otimes_R
\hat{R}) = \End_{R\bmod \hat{R}}(B \otimes_R
\hat{R}) = \End_{\HC_{W_r,V}}( B) \otimes_R \hat{R}.
\]
By Soergel's hom formula, $\End_{\HC_{W_r,V}}( B) $ is free and
finitely generated as a right $R$-module. Hence $\End_{\hat{\HC}_{W_r,V}}(B \otimes_R
\hat{R}) $ is free and finitely-generated over $\hat{R}$. Thus
the same follows for any object of $\hat{\HC}_{W_r,V}$. We are done,
because it is known (see e.g. \cite[\S 11.4.4]{soergelbook}, which
apparently goes back to \cite{Azumaya}) that
any algebra which is finite over a complete local ring is semi-perfect.
\end{proof}

\subsection{Indecomposable objects stay indecomposable}
The goal of this section is to prove:
%\todo{ read through this proof and adapt to new notation}

\begin{thm} \label{thm:BstaysIndecomposable}
  If $B \in \HC_{W_r,V}$ is indecomposable, then so is
\[
B \otimes_R \hat{R} \in \hat{\HC}_{W_r,V}.
\]
\end{thm}

\begin{proof}[Proof of Theorem \ref{thm:BstaysIndecomposable}:]
By the discussion in \S\ref{subsec:reflection}, we can find a Soergel
bimodule $B' \in \HC_{W_r,V_r}$ such that
\begin{equation}\label{eq:formOfB}
B' \otimes_{R_r} R = R \otimes_R B \quad \text{in $R_r \gbmod
  R$}.
\end{equation}
and moreover
\begin{equation}
  \label{eq:endeq}
  \End_{R_r \gbmod
  R}(B' \otimes_{R_r} R) = \End_{\HC_{W_r,V}}(B).
\end{equation}

Let us choose a basis $\{ f_i\}_{i = 0}^m$ as for $\End(B')$ as right
$R_r$-module as in \S \ref{sec:end}. In particular:
\begin{gather} \label{eq:endprop1}
  f_0 = \id_{B'}; \\\label{eq:endprop2}
\text{$\{f_1, \dots, f_m \}$ is a basis (as a right $R_r$-module)
  for an ideal $I \subset \End(B')$}; \\\label{eq:endprop3}
\id_{B'} \cdot R \simto \End(B')/I.
\end{gather}
By Lemma
\ref{lem:fis} there exists an $M$ such that any product of at least
$M$ elements of $\{ f_i\}_{i = 1}^m$ belongs to $\End(B') \cdot \mg$.
If we set $f_i' \in \End(B)$ to be the image of $f_i \otimes 1$ under the
isomorphism \eqref{eq:endeq}, then the analogue (with $B'$ replaced by
$B$, and $f_i$ replaced with $f_i'$) of
properties \eqref{eq:endprop1}, \eqref{eq:endprop2} and
\eqref{eq:endprop3} still hold. Moreover, it is still the case that
any product of $M$ elements of $\{ f_i' \}_{i = 1}^m$ belongs to
$\End(B) \cdot \mg$.

Now let $e \in \End(B \otimes_R \hat{R})$ be an idempotent. By Lemma \ref{lem:lemExt}
we can write
\[
e = \sum f'_i \otimes r_i
\]
with our $f'_i$ as above, and $r_i \in \hat{R}$. By considering the
image of $e$ under the quotient map
\[
\End(B \otimes_R \hat{R}) = \End(B) \otimes_R \hat{R} \onto (\End(B)/I)
\otimes_R \hat{R} = (\id_B \cdot R) \otimes_R \hat{R}
\]
we deduce that $r_0$ is either 0 or 1. By replacing $e$ with $1-e$ if
necessary, we may assume that $r_0 = 0$. Now, $e$ is an idempotent, and hence
\begin{gather*}
\sum f_i \otimes r_i = e = 
e^M = \sum_{1 \le i_1, \dots, i_M\le m} (f_{i_1} \circ \dots \circ f_{i_M})
\otimes r_{i_1} \dots r_{i_M} \in \End(B) \otimes_R \mg
\end{gather*}
(where $M$ is the same $M$ from earlier).
Hence, for any $j \ge 0$, $e = e^j \in \End(B) \otimes_R
\mg^j \subset \End(B) \otimes_R
\hat{R}$ from which it follows that $e = 0$, by a version of
the graded Nakayama lemma. Hence $B \otimes_R \hat{R}$ is indecomposable as
claimed.
\end{proof}

%\todo{need to remark that $U \cap R_m^+ = \emptyset$.}

\subsection{Indecomposable objects in the localised category}

\begin{prop} \label{prop:locindec}
  Any indecomposable object in $\hat{\HC}_{W_r,V}$ is isomorphic to
  $B_x^r \otimes_R \hat{R}$ for some $x \in W_r$.
\end{prop}

\begin{proof}
  Any indecomposable object in $\hat{\HC}_{W_r,V}$ is isomorphic to a
  direct summand of $B \otimes_R \hat{R}$, for some $B \in \HC_{W_r,V}$. By the
  classification of indecomposable Soergel bimodules, we can choose a decomposition
\[
B = \bigoplus_{x \in W_r} (B^r_x)^{\oplus m_x}
\]
for certain $m_x \in \ZM_{\ge 0}$. Now each $B^r_x$ stays
indecomposable when we apply $(-) \otimes_R \hat{R}$ (by Theorem
\ref{thm:BstaysIndecomposable}). Hence 
\[
B \otimes_R \hat{R}= \bigoplus_{x \in W_r} (B^r_x \otimes_R \hat{R})^{\oplus m_x}
\]
is one possible decomposition of $B \otimes_R \hat{R}$ into
indecomposables. As $\hat{\HC}_{W_r,V}$ is Krull-Schmidt, we deduce
that any summand of $B \otimes_R \hat{R}$ is isomorphic to some $B_x^r
\otimes_R \hat{R}$. This concludes the proof.
\end{proof}

% \todo{
% Grothendieck groups here...}

% Categories with unique decompositions.
% \cite[\S 11.4.1]{soergelbook}

\subsection{The hyperbolic bimodule} \label{sec:hypbim}

In this subsection everything will revolve around a certain
bimodule for the Hecke categories corresponding to $W$ and our
reflection subgroup $W_r$.

Let $\hat{R}$ be as in the previous section, and consider the bimodule
category $\hat{R}\bmod R$. Tensor product over $\hat{R}$ (resp. $R$) equips
this category with a left (resp. right) action of $\hat{\HC}_{W_r,V}$
(resp. $\HC_{W,V}$).
% In formulas:
% \[
% \HC^{(r)}_{W_r,V} \acts \hat{R} \gbmod R \racts \HC_{W,V}
% \]
Let $\CC$ denote the strictly
full additive subcategory generated, under the actions of $\HC^{(r)}_{W_r,V}$
and $\HC_{W,V}$, by $\hat{R} \in \hat{R} \bmod R$. (So any object of
$\CC$ is isomorphic to a finite direct sum of direct summands of
$(\hat{R},R)$-bimodules of the form $B^r \otimes_{\hat{R}} \hat{R} \otimes_R B$ for
$B^r \in \HC^{(r)}_{W_r,V}$ and $B \in \HC_{W,V}$.)  We call $\CC$ the
\emph{categorified hyperbolic bimodule}.  In formulas:
\[ \begin{tikzpicture}[scale=0.7]
  \node at (-4,0) {$\HC^{(r)}_{W_r,V}$};
\node at (-2,0) {$\acts$};
\node at (0,0) {$\CC := \langle \hat{R} \rangle $};
\node at (2,0) {$\racts$};
  \node at (4,0) {$\HC_{W,V}$};
\node[rotate=-90] at (0,-1) {$\subset$};
\node at (0,-2) {$\hat{R} \bmod R$};
\end{tikzpicture}
\]
% \[
% \HC^{(r)}_{W_r,V} \acts \CC := \langle \hat{R} \rangle \subset \hat{R} \bmod R\racts \HC_{W,V}
% \]

For any $x \in W$, consider
\[
\hat{R}_x := \hat{R} \otimes_R R_x.
\]
Thus $\hat{R}_x$ is isomorphic to $\hat{R}$ as a left
$\hat{R}$-module, and the right $R$-action is given by $m \cdot r =
x(r)m$ for $m \in \hat{R}_x$ and $r \in R$.

Recall the minimal coset representatives ${}^r W$ from \S \ref{sec:ref1}.

\begin{lem} \label{lem:hyp}
  Suppose that $z \in {}^rW$ and $s \in S$. We have:
  \begin{equation}
    \label{eq:2}
    \hat{R}_z \cdot B_s = \begin{cases} \hat{R}_z \oplus \hat{R}_{zs}
      &\text{if $zs \in {}^rW$;} \\
B^r_t \cdot \hat{R}_z  & \text{if $zs = tz$ with $t \in S_r$.}
\end{cases}
  \end{equation}
(These two cases are mutually exclusive, see \S \ref{sec:ref1}.)
\end{lem}

\begin{proof}
We have an isomorphism of $(\hat{R},R)$-bimodules
\[
\hat{R}_z \cdot B = \hat{R}_z \otimes_{R^s} R \cong \hat{R}
\otimes_{R^t} R_z = (\hat{R} \otimes_{\hat{R}^t} \hat{R}) \otimes_{\hat{R}} R_z
\]
where $t = zsz^{-1}$.

If $\alpha_t$ is invertible in $\hat{R}$ (which
is the case and only if $\alpha_t \notin V_r$, which is the case if
and only if $zs \in {}^rW$) then 
\[
(\hat{R} \otimes_{\hat{R}^t} \hat{R}) \cong \hat{R} \oplus \hat{R}_t
\]
 (see
\cite[(3.7)]{EWS}, alternatively this can be done by hand). This implies the first isomorphism of the lemma.
% \todo{needs some care based on what rep we take, what assumptions we
% place etc.}
If 
$zs = tz$ with $t \in S_r$ then $\hat{R} \otimes_{\hat{R}^t} \hat{R}
= \hat{B}_t^r$ and the second isomorphism of the lemma follows.
\end{proof}

The following gives a kind of ``block decomposition'' of $\CC$:

\begin{lem}
  \label{lem:block}
Suppose $z, z' \in {}^r W$ with $z \ne z'$. Then for $\hat{B}, \hat{B}' \in
\hat{\HC}_{W_r,V}$ we have
\[
\Hom_{\CC}( \hat{B} \otimes_{\hat{R}} \hat{R}_z,  \hat{B}' \otimes_{\hat{R}} \hat{R}_{z'}) = 0.
\]
\end{lem}

\begin{proof}
  We may assume without loss of generality that $\hat{B}$ and $\hat{B}'$ are
  obtained via extension of scalars (i.e. $(-)\otimes_R \hat{R}$) from bimodules $B, B' \in
  \HC_{W_r,V}$. Now, $B$ and $B'$ are Soergel bimodules, and hence admit ``standard''
  filtrations
  \begin{gather*}
    0 \subset \Gamma_{x_0} B \subset     \Gamma_{x_0,x_1} B \subset \dots
    \subset B,  \\
    0 \subset \Gamma_{x_0} B' \subset     \Gamma_{x_0,x_1} B' \subset \dots
    \subset B' 
  \end{gather*}
with $i^{th}$ subquotient isomorphic to direct sums of shifts of
$R_{x_i}$, for some enumeration $\id = x_0, x_1, x_2, \dots$ of $W_r$
compatible with the Bruhat order. Now the result follows easily from the vanishing
\[
\Hom_{\hat{R}\bmod R}(\hat{R}_{x_i} \otimes_{\hat{R}} \hat{R}_z,
\hat{R}_{x_j} \otimes_{\hat{R}} \hat{R}_{z'}) = 0
\]
which holds because $\hat{R}_{x_i} \otimes_{\hat{R}} \hat{R}_z \cong
\hat{R}_{x_iz}$ and $\hat{R}_{x_j} \otimes_{\hat{R}} \hat{R}_{z'}
\cong \hat{R}_{x_jz'}$, and $x_iz$ and $x_jz'$ never coincide, as our
assumptions mean that they lie in different $W_r$-cosets.
\end{proof}

We can now prove:

\begin{thm} \label{thm:hyperbolic}
The additive category $\CC$ is Krull-Schmidt. Moreover, we have a
bijection:
\begin{align*}
W_r \times {}^rW
&\simto 
  \left \{ \begin{array}{c}
    \text{indecomposable } \\
\text{objects in $\CC$}
  \end{array} \right \}_{/ \cong} \\
(x,z) & \mapsto \hat{B}_x \otimes_{\hat{R}} \hat{R}_z.
\end{align*}
\end{thm}

\begin{proof} We first establish that $\CC$ is Krull-Schmidt.
 As in \S~\ref{sec:locKS}, we only need to check that the
 endomorphism ring of any object in $\CC$ is semi-perfect. By Lemma
 \ref{lem:hyp}, any object in $\CC$ is isomorphic to a finite direct sum of
 objects of the form $\hat{B}^r \otimes_{\hat{R}} R_z$, for $\hat{B}^r
 \in \hat{\HC}_{W_r,V}$. By Lemma \ref{lem:block} it is enough to
 check that the endomorphisms of such objects are semi-perfect. We have
 \begin{equation}
   \label{eq:indec}
\End_\CC(\hat{B}^r \otimes_{\hat{R}} R_z) = \End_{\hat{R}\bmod R}
(\hat{B}^r) = \End_{\hat{\HC}_{W_r,V}}(\hat{B}^r)   
 \end{equation}
and the result follows from Proposition \ref{prop:locKS}.

We now turn to the classification of indecomposable objects. By
\eqref{eq:indec} we see that $\hat{B}^r \otimes_{\hat{R}} R_z$ is
indecomposable if and only if $\hat{B}^r$ is. By Proposition
\ref{prop:locindec} (the classification of indecomposable objects in $\hat{\HC}_{W_r,V}$) this is the case if and
only if $\hat{B}^r \cong \hat{B}^r_x$ for some $x \in W_r$. The fact
that $\hat{B}_x \otimes_{\hat{R}} \hat{R}_z$ and $\hat{B}_{x'}
\otimes_{\hat{R}} \hat{R}_{z'}$ are isomorphic if and only if $x = x'$
and $z = z'$ follows from Lemma \ref{lem:block} and Proposition
\ref{prop:locindec}.
\end{proof}

% \begin{remark}
%   I have a strong impression that $\hat{\HC}$ is actually a ring!
%   Describe it... It is somewhat subtle it seems ... it contains a copy
%   of $\ZM[W_r]$ (standard bimodules).

% I have no idea what its Grothendieck group is...
% perhaps this can be left as an open question.
% \end{remark}

% \begin{question}
%   Can one develop deeper theory of hyperbolic localisation for Soergel
%   bimodules, i.e. really reproduce the two geometric categories above.
% \end{question}

The following is the bimodule analogue of Theorems \ref{thm:hypbim}
and \ref{thm:sm}, and shows that $\CC$ deserves to be called the
categorified hyperbolic bimodule:

\begin{thm} \label{thm:cathypbim}
We have a
commutative diagram
\begin{equation*}
  \begin{tikzpicture}[scale=0.8]
    % upper stuff
    \node (ul) at (-3.7,0) {$\HC^{\Soe}_{W_r, V}$};
    \node at (-2.2,0) {$\acts$};
    \node (ml) at (0,0) {$\CC$};
     \node at (2.2,0) {$\racts$};
     \node (ur) at (3.7,0) {$\HC^{\Soe}_{W, V}$};
     % lower stuff
     \node (ll) at (-3.7,-2) {$H_{W_r}$};
         \node at (-2.2,-2) {$\acts$};
         \node (l) at (0,-2) {$\ZM[W]$};
                  \node at (2.2,-2) {$\racts$};
                  \node (lr) at (3.7,-2) {$H_{W}$};
                  % arrows stuff
               \draw[dashed,->] (ur) to node[right] {\small $\ch$} (lr);
               \draw[dashed,->] (ul) to node[left] {\small $\ch$} (ll);
               % middle arrows
                              % \draw[dashed,->] (mr) to node[right]
                              % {\small $\hat{\ch}$} (l);
                              \draw[dashed,->] (ml) to node[left]
                              {\small $\hat{\ch}$} (l);               
% def of hyperbolic bimodule
\node[rotate=90] at (0,-2.5) {$=$};
\node at (0,-3) {$\bigoplus_{x \in W} \ZM \hat{\delta}_x$};
             \end{tikzpicture}
\end{equation*}
where $\ch$ are the isomorphisms of Theorem \ref{thm:Soe_cat}, and
$\hat{ch}$ is the unique isomorphism making the right hand square
commute.
\end{thm}

\begin{remark}
  We leave it to the reader to combine the above theorem with the
  classification of the indecomposable objects in $\CC$ (Theorem
  \ref{thm:hyperbolic}) to deduce the positivity statements in Theorem
  \ref{thm:pos3}.
\end{remark}

\subsection{Localisation and the anti-spherical module}
\label{sec:locas}

We keep $(W,S)$ and $V$ as above. We fix the following:
\begin{enumerate}
\item A standard parabolic subgroup $W_f \subset W$ with simple
  reflections $S_f \subset S$;
\item A $V$-good (in the sense of \S\ref{sec:good}) reflection subgroup
  $W_r \subset W$ with canonical  Coxeter generators $S_r$.
\end{enumerate}
The following assumption is fundamental:
\begin{equation}
  \label{eq:ASassump}
  W_f \subset W_r \quad \text{and hence $S_f \subset S_r$ and ${}^fW
    \supset {}^rW$.}
\end{equation}
(We have the alcove picture from \S \ref{sec:affine} in mind.)

\begin{remark}
  The example we have in mind is $W = W_{\fin} \ltimes \ZM \Phi^\vee$
  an affine Weyl group, $W_f = W_\fin$ and $W_r = W_\fin \ltimes 
  p\ZM\Phi^\vee$, which is good in characteristic $p$.
% \todo{ NEED to explain this!!}
\end{remark}

We briefly recall how (following \cite{AB,RWpcan}) to $W_f \subset W$ we may
associate its \emph{anti-spherical module}.  We set
\[
\AS := \HC_{W,V} / \IC \quad \text{where} \quad \IC := \langle
B_x \; | \; x\notin {}^f W \rangle_{\oplus}.
\]
It turns out that $\IC$ is a right ideal in $\HC_{W,V}$ and hence
$\AS$ is a right $\HC_{W,V}$-module category (see \cite[Lemma
4.2.3]{RWpcan}).

Consider the quotient (of additive categories)
\[
\hat{\AS} := \CC / \IC \quad \text{where} \quad \IC := \langle B^r_u \otimes_{\hat{R}} \hat{R}_z \; | \; z
\in {}^r W; u \notin {}^f W \rangle_{\oplus}.
\]
One can check (by an analogue of \cite[Lemma
4.2.3]{RWpcan} again) that $\IC$ is stable
under the right action of $\HC_{W,V}$. Hence $\hat{\AS}$ has the
structure of a right $\HC_{W,V}$-module category. Our next goal is to
explain that $\hat{\AS}$ admits functor (of right $\HC_{W,V}$-module
categories) from the anti-spherical category, and that $\hat{\AS}$ is a
rather simple localisation of the anti-spherical category. For this we
need the following lemma:

\begin{lem}
  Suppose that $x \notin {}^f W$. Let us write
  \begin{equation}
    \label{eq:xLoc}
    \hat{R}\otimes B_x \cong \bigoplus_{u \in W_r; z \in {}^r W} (B^r_u \otimes_{\hat{R}}
    \hat{R}_z)^{\oplus a(u,z,x)}
  \end{equation}
for $a(u,z,x) \in \ZM_{\ge 0}$. Then
\[
a(u,z,x) = 0 \quad \text{if $u \in {}^f W$.}
\]
\end{lem}

\begin{proof}
  Because $x \notin {}^f W$ there exists $s \in S_f$ such that $sx <
  x$. Hence ${}^p h_{y,x} = v \cdot {}^ph_{sy,x}$ for all $ y\in W$
  with $sy > y$. (This echoes a well-known property of the
  Kazhdan-Lusztig basis. It can be deduced from this by
  \cite[Proposition 4.2(5)]{JW-pcan}). In particular:
  \begin{equation}
    \label{eq:xLoc1}
 {}^p h_{y,x}(1) = {}^ph_{sy,x}(1) \quad\text{for all $y\in {}^s
  W.$}
  \end{equation}
Let us fix such a $y$, and write $y = vz$ with $v \in W_r$ and $z \in
{}^r W$. Note that $y \in {}^s W$ is equivalent to $v \in {}^s W_r$.

The left-hand side (resp. right-hand side) of
\eqref{eq:xLoc1} is the coefficient of $\hat{\delta}_y$
(resp. $\hat{\delta}_{sy}$) in the character of
$\hat{R}\otimes B_x$ in the hyperbolic bimodule (see Theorem
\ref{thm:cathypbim} for the definition of $\hat{ch}$ and
$\hat{\delta_x}$). On the other hand, 
using the right hand side of \eqref{eq:xLoc}, the coefficient of
$\hat{\delta}_y$ (resp. $\hat{\delta}_{sy}$) is given by
\begin{equation} \label{eq:xLocOtherSide}
\sum_{v \in W_r} a(u,z,x) \cdot {}^p h_{v,u}(1)
\quad
\text{(resp. $\sum_{v \in W_r} a(u,z,x) \cdot {}^p h_{sv,u}(1)$)}
\end{equation}
By unimodality (in the Hecke category for $W_r$) we have
\[
{}^ph_{v,u}(1) \ge {}^ph_{sv,u}(1).
\]
By \eqref{eq:xLoc1} both terms in \eqref{eq:xLocOtherSide} are
equal. Taking $u = v$ (and using that ${}^p h_{v,v}(1) = 1$ and ${}^p
h_{sv,v} = 0$) we see that
\[
a(v,z,x) = 0
\]
which is what is claimed in the lemma.
\end{proof}

Let us set
\[
\mathbb{1}_{\hat{\AS}} := \text{image of $\hat{R}$ in $\hat{\AS}$}.
\]
We conclude from the lemma that
\begin{equation}
  \label{eq:7}
  \mathbb{1}_{\hat{\AS}}\otimes B_x = 0 \quad \text{if $x \notin {}^f W$.}
\end{equation}
Hence the functor $B \mapsto \mathbb{1}_{\hat{\AS}}\otimes B$ descends
to a functor of right $\HC_{W,V}$-module categories
\begin{align*}
\phi : \AS \to \hat{\AS}.
\end{align*}
Because $S_f \subset S_r$ we have $V_f \subset V_r$. Hence we have a
morphism
\[
R/(V_r) \to \hat{R}/(\mathfrak{m}).
\]

\begin{thm} \label{thm:asloc}
$\phi$ induces an equivalence of additive categories
\[
(\hat{R} \otimes_{(R/(V_r))} \AS)^{e} \simto \hat{\AS}.
\]
(here the superscript $e$ denotes idempotent completion).
\end{thm}

\begin{remark}
  This theorem can be unpacked to give the upper left square of
  Figure \ref{fig:locspec}.
\end{remark}

% \begin{remark}
%   When $W_f = W_r$ then we recover \cite{LWasph} (up to the quotient field
%   being replaced by ...). This theorem also answers the question posed
%   in the introduction to \cite{LWasph}.
% \end{remark}
 
% \begin{remark}
%   We deduce more positive properties.
% \end{remark}

\section{Acknowledgements}
Initially I had titled these notes ``Variations on a theme of Hazi''. Later I changed the title as the notes
broadened in scope; I remain very grateful for the impetus Hazi's
paper  \cite{Hazi} gave me. I would like to thank my collaborators, as well as N.~Abe, 
D.~Treumann, I.~Grojnowski, A.~Hazi and G.~Lonergan for useful discussions
on the subject of this  paper. I am particularly grateful for
discussions with T.~Jensen, which were very motivating when working on this
project. Thanks to N.~Abe, D.~Juteau, G.~Lusztig, U.~Thiel,
D.~Treumann, L.~Patimo and O.~Yacobi for useful feedback on a first draft.
I am forever grateful to W.~Soergel and P.~Fiebig, for
teaching me everything I know about the philosophy of deformations.

\def\cprime{$'$} \def\cprime{$'$} \def\cprime{$'$}

% \bibliographystyle{myalpha}
% \bibliography{gen}

%%  The bibliography

%\begin{thebibliography}{9}
%%  Use \bibitem{r1} or \bibitem[Surname(2010)]{r1} (for authoryear case)
%
%\bibitem{}
%
%\end{thebibliography}

\end{document}